\documentclass[a4paper,12pt]{amsart}

\usepackage{amsthm}
\usepackage{amsmath}
\usepackage{amsfonts}
\usepackage{amssymb}
\usepackage{times}
\usepackage{graphicx}
\usepackage{mathabx}

\theoremstyle{definition}
\newtheorem{defn}{\indent\bf Definition}
\newtheorem{rem}[defn]{\indent\bf Remark}
\newtheorem{ex}[defn]{\indent\bf Example}
\theoremstyle{plain}
\newtheorem{lemma}[defn]{\indent\bf Lemma}
\newtheorem{prop}[defn]{\indent\bf Proposition}
\newtheorem{thm}[defn]{\indent\bf Theorem}
\newtheorem{cor}[defn]{\indent\bf Corollary}

\def\tilde{\widetilde}

\def\PSL{\mathop{\rm PSL}\nolimits}
\def\SL{\mathop{\rm SL}\nolimits}
\def\PGL{\mathop{\rm PGL}}

\def\Aut{\mathop{\rm Aut}}

\def\Z{\mathbb Z}
\def\R{\mathbb R}
\def\cR{{\mathcal R}}

\def\N{\mathbb N}
\def\C{\mathbb C}
\def\bH{\mathbb H}
\def\L{{\mathcal L}}

\def\U{{\mathcal U}}
\def\H{{\mathcal H}}
\def\A{{\mathcal A}}
\def\S{{\mathcal S}}

\def\O{{\mathcal O}}
\def\V{{\mathcal V}}
\def\D{{\mathcal D}}
\def\B{{\mathcal B}}

\def\cC{{\mathcal C}}
\def\cR{{\mathcal R}}

\def\P{{\mathcal P}}

\def\cX{{\mathcal X}}

\begin{document}

\title[Endomorphisms of spaces of  virtual vectors]{Endomorphisms of spaces of virtual  vectors fixed by  a discrete group}
\author[Florin R\u adulescu]{Florin R\u adulescu${}^*$
\\ \\
Dipartimento di Matematica\\ Universit\` a degli Studi di Roma ``Tor Vergata''}
\dedicatory{Dedicated to Professor Pierre De La Harpe on the occasion of his 70th anniversary}
\maketitle

\thispagestyle{empty}

\renewcommand{\thefootnote}{}
\footnotetext{${}^*$ Member of the Institute of  Mathematics ``S. Stoilow" of the Romanian Academy}
\footnotetext{${}^*$
Supported in part by PRIN-MIUR, and by a grant of the Romanian National Authority for Scientific Research, project number PN-II-ID-PCE-2012-4-0201
}
\footnotetext{${}^*$ Visiting Professor, Department of Mathematics, University of Copenhagen}

\begin{abstract}

Consider  a  unitary representation $\pi$  of a  discrete group  $G$, which, when restricted to an almost normal subgroup $\Gamma\subseteq G$, is of type II. We analyze the associated unitary representation $\overline{\pi}^{\rm{p}}$ of $G$ on  the Hilbert space of "virtual" $\Gamma_0$-invariant vectors, where $\Gamma_0$ runs over a suitable  class of finite index subgroups of $\Gamma$. The unitary representation $\overline{\pi}^{\rm{p}}$ of $G$ is uniquely determined by the requirement that the  Hecke operators, for all  $\Gamma_0$, are the "block matrix coefficients" of $\overline{\pi}^{\rm{p}}$.

If $\pi|_\Gamma$ is an integer multiple of the  regular representation,  there exists a subspace $L$ of the Hilbert space of the representation $\pi$, acting as a fundamental domain for $\Gamma$. In this case, the space of   $\Gamma$-invariant vectors is identified with $L$.
 When $\pi|_\Gamma$ is  not an integer multiple of the regular representation, (e.g. if $G=PGL(2,\mathbb Z[\frac{1}{p}])$, $\Gamma$ is the modular group,  $\pi$ belongs to the discrete series of  representations of $\PSL(2,\mathbb R)$, and the  $\Gamma$-invariant vectors are the cusp forms) we assume that $\pi$ is  the restriction to a subspace $H_0$ of a larger unitary representation having a  subspace $L$ as above.

 The operator angle   between the projection $P_L$ onto $L$ (typically the characteristic function of the fundamental domain) and   the projection $P_0$ onto the subspace $H_0$ (typically a Bergman projection onto a space of analytic functions),  is   the analogue of the space of  $\Gamma$- invariant vectors.

 We prove that
the  character of the unitary  representation $\overline{\pi}^{\rm{p}}$
 is uniquely determined by the character of the representation $\pi$.


\end{abstract}

\vskip20pt

\section{Introduction and main results}

Let $G$ be a countable discrete group, and let $\Gamma$ be an almost normal subgroup. Let $\S$ be the minimal, intersection lattice of finite index subgroups of $\Gamma$, that is closed with respect  to the operation
\begin{equation}\label{gammasigma}
\Gamma_0\rightarrow  (\Gamma_0)_\sigma= \Gamma_0 \cap \sigma\Gamma_0\sigma^{-1},\quad \sigma \in G.
\end{equation}
 We assume throughout this paper that the intersection of the subgroups in $\S$ is the identity element.

 We consider a unitary (projective) representation $\pi$ of $G$ in a Hilbert space $H$.
 Throughout   this paper we  make the assumption that $\pi|_{\Gamma}$ is a multiple of the left regular representation $\lambda_{\Gamma}$  (\cite{Ta}, \cite{Sak}) of $\Gamma$ (eventually skewed with  a cocycle if $\pi$ is projective). Under  this hypothesis, we construct Hilbert spaces $H^{\Gamma}$ of $\Gamma$-invariant vectors. More generally, we  construct Hilbert spaces $H^{\Gamma_0}$ of $\Gamma_0$-invariant vectors,  when $\Gamma_0$ runs over $\S$.

  We will  refer to  vectors as above, for $\Gamma_0\in \S$,   as to  "virtual" $\Gamma_0$-invariant vectors, as they generally correspond to $\Gamma_0$-invariant, densely defined, linear functionals on $H$.
%
The Hilbert  spaces  $H^{\Gamma_0}$, $\Gamma_0\in \S$, are not proper subspaces of $H$; they are located in an enlargement, in the sense of Gelfand triples (see \cite{GV})  of the given Hilbert space (see Definition \ref{formalism}).  For all subgroups $\Gamma_0,\Gamma_1\in \S$ such that $\Gamma_1\subseteq \Gamma_0$, there exists a canonical isometrical embedding $H^{\Gamma_0}\subseteq H^{\Gamma_1}.$

We exemplify this  in the case of the left regular representation $\lambda_{\Gamma}$ of $\Gamma$ into the unitary group $\U(l^2(\Gamma))$ of the Hilbert space $l^2(\Gamma)$. In this case the space of $\Gamma$-invariant vectors is the one dimensional space generated by the constant functions on $\Gamma$. This space may be considered as a space of densely defined, linear functionals on    $l^2(\Gamma)$. In general, in this formalism,
 $l^2(\Gamma)^{\Gamma_0}=\ell^2(\Gamma_0\backslash \Gamma)$.

In general, if the unitary representation $\pi$ into the Hilbert space $H$
is an integer multiple of the left regular representation $\lambda_\Gamma$ of the discrete group $\Gamma$, then there exists a  subspace $L$ such that, $\Gamma$-equivariantly, we have that
$$H\cong l^2(\Gamma)\otimes L, \quad \pi|_\Gamma\cong\lambda_\Gamma \otimes {\rm Id}_L .$$
In this case too, it is  obvious that  one may identify the space of $\Gamma$-invariant vectors with the Hilbert space $L$. In this identification $L$ is no longer a proper subspace of $H$. Moreover,  $L$ is not unique.

If the unitary representation $\pi$ is not an integer multiple of the left regular representation $\lambda_\Gamma$,
we will make use  of a  subspace $L$ from a larger representation of $G$, which contains the given representation $\pi$ as a subrepresentation.
In this case we will prove below that the operator angle between the projection onto the subspace of the subrepresentation and the projection onto $L$ may be used to construct the space of $\Gamma$-invariant vectors and its scalar product.

This  is analogous to the Petersson scalar product formula (\cite{Pe}) for automorphic forms, where to introduce the scalar product on automorphic forms which are "virtual"  $PSL(2,\mathbb Z)$-invariant vectors for the restrictions to $PSL(2,\mathbb Z)$ of the  representations $\pi_n, n\geq 2$ in the  discrete series  of $PSL(2,\mathbb R)$, one uses a "measuring scale" consisting of a fundamental domain $F$ for the action of  $PSL(2,\mathbb Z)$ on the upper halfplane. The larger unitary representation, containing the representation $\pi_n$  as a subrepresentation, is obtained as follows: if $\pi_n$ is realized as unitary representation on a space $H^2(\mathbb H,\nu_n)$   consisting of square  integrable analytic functions with respect to a measure $\nu_n$ on $\mathbb H$, the the larger unitary representation containing $\pi_n$ as a subrepresentation is realized on the corresponding Hilbert space  $L^2(\mathbb H,\nu_n)$.
The space of square summable functions supported in $F$  plays the role of the subspace $L$ as above.

Let $\sigma$ be a group element in $G$. We use the notation introduced in formula (\ref{gammasigma}).
It is obvious that $\sigma\Gamma_{\sigma^{-1}}\sigma^{-1}=\Gamma_\sigma$.  This implies that   the representation $\pi$ induces a unitary transformation, denoted by  $\overline{\pi}^{\rm p}(\sigma)$ that maps
$H^{ \Gamma_{\sigma^{-1}}}$ onto $H^{ \Gamma_{\sigma}}$.

The main problem we are approaching in this paper is the analysis of the unitary representation $\overline {\pi}^{\rm p}$, induced by the unitary representation $\pi$   on the Hilbert space $\overline H^{\rm p}$ obtained by taking the inductive limit of the  Hilbert spaces $H^{\Gamma_0}$, $\Gamma_0\in \S$.   The content of the classical Ramanujan-Petersson Problem (\cite {Hej}, \cite {Sar1}, \cite{Bo}) is transformed into  a harmonic analysis problem, concerning the weak unitary containment
of the representation $\overline {\pi}^{\rm p}$ in the unitary representation obtained by restriction to $G$, of the left regular representation of the Schlichting completion (\cite{Sch}, \cite{KLM}) of $G$ with respect to the subgroups in the family $\S$.

The representation $\overline {\pi}^{\rm p}$ is determined by the associated  Hecke operators (\cite {He}, \cite {Hej}, \cite {Sar})  at all levels $\Gamma_0$ .
The Hecke operators are "block-matrix coefficients" of the associated   unitary representation
$\overline{\pi}^{\rm{p}}$.
We are using the notation $\overline {\pi}^{\rm p}$ for the representation on the space of all "virtual" $\Gamma_0$-invariant vectors, $\Gamma_0\in \S$ to recall the profinite completion procedure that is used in the construction of this representation. When the initial representation $\pi$ is projective with cocycle $\epsilon$, and the cocycle $\epsilon$ admits an extension to the Schlichting completion $\overline G$ introduced below, the construction in this paper works ad litteram in the projective case.

The representation $\overline {\pi}^{\rm p}$ is widely used in a slightly different form in the literature (see e.g. \cite{Bo}, \cite{Cass}, \cite{Hal}). In this paper we   use an  operator algebra construction of $\overline {\pi}^{\rm p}$. This construction enables  us to get unitarily equivalent forms of
$\overline {\pi}^{\rm p}$ that are more suitable for the computations of traces.

The main result of this paper is the  correspondence between the two representations
 $\pi$ and $\overline {\pi}^{\rm p}$.  The representation $\overline {\pi}^{\rm p}$ has a canonical block matrix structure associated with a given choice of coset representatives for subgroups in $\S$. Consequently, we represent the associated Hecke operators as block matrices, whose entries are "localized sums" over cosets, of the values of  original representation $\pi$, restricted to the space $L$ introduced above (which is  the analogue  of a fundamental domain for the action of the group $\Gamma$).
 As corollary we obtain a precise formula relating the characters of $\pi$ and $\overline {\pi}^{\rm p}$.

 To establish a relation between the two representations, we introduce a summability condition, which warrants the convergence of the sums for the matrix entries in the representations that  we are  obtaining for the Hecke operators.

 \begin{defn}\label{TC}
  Let $\pi_0$ be a  unitary representation of $G$, such that $\pi_0|_\Gamma$ is a (not necessary integer) multiple of the left regular representation $\lambda_\Gamma$.  We assume that there exists  a  unitary representation $\pi$  of $G$ into the unitary group of a Hilbert space $H$ containing $H_0$, with the following properties:

  \item(i)  $\pi|_\Gamma$ is an integer multiple  of the left regular representation $\lambda_\Gamma$. Consequently, there exists a Hilbert subspace $L$ of $H$ such that  $$H \cong l^2(\Gamma) \otimes L, \quad \pi| _{\Gamma} \cong \lambda_{\Gamma} \otimes {\rm Id}_L.$$
 \item (ii) Let $e$ be the identity element of the group $G$.
We denote the orthogonal projection from $H$ onto $L\cong \C e\otimes L $ by $P_L$. We assume that  the projections in the family $\{\pi(g)P_L\pi(g^{-1})| g\in G\}$ form a commutative family.

 \item(iii)  Assume that $\pi_0$ is a subrepresentation of $\pi$. Consequently $H_0\subseteq H$  is $\pi(G)$-invariant and  $\pi_0$
 is the restriction of $\pi$ to $H_0$. Denote by $P_0$ the orthogonal projection from $H$ onto $H_0$. Clearly, in this case we have $\pi_0(g)=P_0\pi(g)P_0$, $g\in G$.

%
%
%

 \item {(iv)} The product of the operators  $P_0$ and $P_L$ is trace class.

 \item{(v)} For every $g$ in $G$ and $\Gamma_0$ in $\S$, the following sum,  over the coset $\Gamma_0g$,
\begin{equation}\label{TCsum}
\mathop{\sum}\limits_{\theta \in \Gamma_0g}P_L\pi_0(\theta)P_L,
\end{equation}
 is convergent in the space of Hilbert-Schmidt operators $\mathcal C_2(L)$.

 \item {(vi)}
The sum of traces of the operators in the sum in formula (\ref{TCsum})  is absolutely convergent. The sum in formula (\ref{TCsum}) is a  trace class operator, whose  trace  is equal to the sum of traces of the operators in the sum.
\item {(vii)} The Murray von Neumann dimension number dim$_{\{\pi_0(\Gamma)\}''}H_0$ (\cite{Ta}, \cite {Sak}, \cite{Jo})
 is a finite, strictly positive number.

\end{defn}

The condition (vii) above may be interpreted as saying that $H_0$ is a finitely generated  left Hilbert  module over  the representation
$\pi_0|_{\Gamma}$, which is a multiple of the the left regular representation of $\Gamma$.

In the main examples considered in this paper, where the representations $\pi_0$ are obtained from unitary representations in the discrete series of (projective) unitary  of $PSL(2,\mathbb R)$ by restriction  to $G=PGL(2,\mathbb Z[\frac{1}{p}])$, the conditions (iii), (iv) are a consequence of the computations in \cite{Za}. The condition (ii) is automatic if the representation $\pi$ comes from a Koopman unitary representation. Indeed, in this case the projection $P_L$ is the operator of multiplication with the characteristic function of a fundamental domain  for $\Gamma$.
Furthermore, the projections in the family are the operators of multiplication  by the characteristic functions of translations of the fundamental domain with elements in $G$.

A unitary representation $\pi_0$ of $G$ as above determines, in a canonical way, a unitary representation
$\overline {\pi_0}^{\rm p}$  of the Schlichtling completion (\cite{Sch}) $\overline G$ of $G$, with respect  to the subgroups in the family $\S$. The unitary representation $\overline {\pi_0} ^{\rm p}$   acts on $\overline H ^{\rm p}$,  the Hilbert space completion
 of the inductive limit of the Hilbert spaces $H^{\Gamma_0}$, $\Gamma_0\in \S$. This construction is rigorously presented in Theorem   \ref{traces}.
The following result gives an explicit description of the block matrix coefficients of the representation $\overline {\pi_0} ^{\rm p}$. These are Hecke operators corresponding to subgroups $\Gamma_0\in \S$ associated with the original unitary representation $\pi_0$. This representation is  used to compute the character of the representation $\overline {\pi_0} ^{\rm p}$. In the following theorem, we denote $(\Gamma_0)_\sigma =\sigma \Gamma_0\sigma^{-1}\cap \Gamma_0$, $\Gamma_0 \in \S$, $\sigma\in G$.

 \

 \begin{thm}\label{L0} Let $\pi_0$ be a unitary representation of $G$ into the unitary group $\U(H_0)$ of a Hilbert space $H_0$. Assume that $\pi_0$ verifies the technical condition from Definition \ref{TC}. Let $\overline{\pi_0}^{\rm p}$ be the unitary representation of $\overline G$ introduced above.

 For every $\Gamma_0$ in $\S$, fix a family of coset representatives for $\Gamma_0$ in $\Gamma$ so that $\Gamma = \mathop{\bigcup}\limits_{i = 1}^{[\Gamma : \Gamma_0]}\Gamma_0s_i$.
We represent $B(l^2(\Gamma_0\setminus\Gamma))$ by the matrix unit $$(e_{\Gamma_0s_i}, e_{\Gamma_0s_j})_{i, j = 1, 2, \ldots, [\Gamma : \Gamma_0]}.$$

For $\Gamma_0\in \S$ and $\sigma \in G$, the Hecke operators
  associated to $\pi_0$ that correspond to the double coset $\Gamma_0\sigma\Gamma_0$ are:
\begin{equation}\label{finalhecke}
[\Gamma_0: (\Gamma_0)_\sigma] P_{H_0^{\Gamma_0}}\overline{\pi}_0^{\rm p}(\sigma)P_{H_0^{\Gamma_0}}.
\end{equation}
Using the  unitary equivalence $$B(l^2(\Gamma_0 \setminus\Gamma)) \otimes B(L)\cong B(H^{\Gamma_0}),$$and the above choice of a matrix unit, the Hecke operators in formula (\ref{finalhecke})
are unitarily equivalent  to  the operators  in $B(l^2(\Gamma_0 \setminus \Gamma)) \otimes B(L)$ given by
\begin{equation}\label{heckematrix0}
\mathop{\sum}\limits_{i,j}\mathop{\sum}\limits_{\theta \in s_i^{-1}\Gamma_0\sigma\Gamma_0s_j}P_{L}\pi_0(\theta)P_L \otimes e_{\Gamma_0s_i, \Gamma_0s_j}.
\end{equation}

\end{thm}

Formula (\ref{heckematrix0}) allows us to immediately conclude the following:

  \begin{cor} For $\sigma \in G, \Gamma_0\in \S$,
the trace of the Hecke operator introduced  in formula (\ref{finalhecke})
 is computed by
 \begin{equation}\label{f8}
 {\rm Tr}\big[ [\Gamma_0: (\Gamma_0)_\sigma] P_{H_0^{\Gamma_0}}\overline{\pi}_0^{\rm p}(\sigma)P_{H_0^{\Gamma_0}}\big]=
\mathop{\sum}\limits_{s_i}\; \mathop{\sum}\limits_{\theta \in s_i\Gamma_0\sigma\Gamma_0s_i^{-1}}{\rm Tr}(P_L\pi_0(\theta)P_L).
\end{equation}

\end{cor}

\vskip10pt

When specializing for example to the case $\Gamma_0=\Gamma$,  formula (\ref{heckematrix0})
 gives the following:

 Consider the case when
$\dim_{{\{\pi_0(\Gamma)\}''}}H_0$ is finite.
The space $H_0^{\Gamma}$ of  "virtual" vectors, invariant   to  the unitary representation $\pi_0 |_\Gamma$ acting on  $H_0$,
is unitarily equivalent to the range  of the projection
\begin{equation}\nonumber
\P_{\Gamma, L} = \mathop{\sum}\limits_{\gamma \in \Gamma} P_{L}\pi_0(\gamma)P_{L}\in B(L),
\end{equation}
which is a finite dimensional subspace of $L$.
The same unitary equivalence transforms  the Hecke operator corresponding to a coset $\Gamma\sigma\Gamma$ and acting on "virtual" $\Gamma$-invariant operators into the operator
\begin{equation}\nonumber
\mathop{\sum}\limits_{\gamma \in \Gamma\sigma\Gamma} P_{L}\pi_0(\gamma)P_{L}\in B(L),\quad \sigma \in G.
\end{equation}
The convergence in the above formulae  is a consequence of the technical assumption introduced in Definition \ref{TC}.

\

\

As a corollary of formula (\ref{f8}), we obtain a formula for the  character (\cite{HC}, \cite{Sal}, \cite{Ca}, \cite{GeGr})
of the representation $\overline\pi_0^{\rm p}$. This proves that the character of $\overline \pi_0^{\rm p}$ is determined by the    positive definite function $\phi_0$ on $G$,  given by
\begin{equation}\label{pdf}
\phi_0(g)={\rm Tr}(P_L\pi_0(g)P_L), \quad g \in G.
\end{equation}

We use the notations and definitions from  the statement of  Theorem \ref{L0}. 
%
%
We assume that $\theta_{\overline{\pi}_0^{\rm p}}$, the character of the representation $\overline{\pi}_0^{\rm p}$ of $\overline{G}$, is locally integrable with respect to Haar measure on $\overline G$.
Recall that by \cite{Ca}, formula (13), (for a different proof see also Lemma \ref{cartiertr}), we have
\noindent

\begin{equation}\label{hyperfinite}
\theta_{\overline{\pi}_0^{\rm p}}(\sigma) = \mathop{\lim}\limits_{\mathop{\Gamma_0 \downarrow \; e}\limits_{\Gamma_0 \in \S}}
{\rm Tr}(P_{H_0^{\Gamma_0}}\overline{\pi}_0^{\rm p}(\sigma)P_{H_0^{\Gamma_0}}).\end{equation}

For every $g \in G$, let $\Gamma^{\rm st}_g$ be the normalizer group of $g$ in $\Gamma$ defined by $$\Gamma^{\rm st}_g = \{ \gamma \in \Gamma | \gamma g = g\gamma \}.$$

Using the formulae (\ref{f8}) and (\ref{hyperfinite}) we obtain:

\begin{cor}\label{plancherel} \label{theta111}
\noindent The  value of the character $\theta_{\overline{\pi}_0^{\rm p}}$ at an element $\sigma\in G$ is computed by the formula:
\begin{equation}\label{realresult}
\theta_{\overline{\pi}_0^{\rm p}}(\sigma)=
\mathop{\lim}\limits_{\mathop{\Gamma_0 \downarrow \; e}\limits_{\Gamma_0 \in \S}} \frac{1}{[\Gamma_0: (\Gamma_0)_\sigma] } \mathop{\sum}\limits_{\Gamma =
\mathop{\cup}
\limits_{i = 1}^{[\Gamma : \Gamma_0]}
\Gamma_0s_i}\; \mathop{\sum}\limits_{\theta \in s_i^{-1}\Gamma_0\sigma\Gamma_0s_i}{\rm Tr}(P_L\pi_0(\theta)P_L).
\end{equation}


  If the group $\Gamma^{\rm st}_g$ is trivial then

\begin{equation}\label{berezin}
\theta_{\overline{\pi}_0^{\rm p}}(g)=\big[ \mathop{\lim}\limits_{\mathop{\Gamma_0 \downarrow \; e}\limits_{\Gamma_0 \in \S}} \frac{1}{[\Gamma_0: (\Gamma_0)_g] }\big]
\mathop{\sum}\limits_{\gamma \in \Gamma
}{\rm Tr}(P_L\pi_0(\gamma g\gamma^{-1})), g \in G.
\end{equation}

\end{cor}

Under suitable conditions, the term on the right hand side of equation (\ref{berezin}) coincides with the character of the unitary representation $\pi_0$. Assume  there exists a  unitary representation  $\overline{\pi}_0^R$, extending $\pi_0$ to a   a locally compact group $\overline{G}^R$, that contains $G$ as a  dense subgroup  Then, the formula for the trace  character of the representation $\overline{\pi}_0$ depends only on    the trace character  of a representation  $\overline{\pi}_0^R$.

\vskip10pt

\begin{lemma} \label{characteroriginal}
Consider a unitary  representation $\pi_0$ verifying the conditions in Definition \ref{TC}.
We also consider a larger locally compact group $\overline{G}^R$, containing $G$ as a dense subgroup, and  such that  $\Gamma$ is a lattice in $\overline{G}^R$.
%
Assume the following conditions:
 \item (i) $\pi_0$  extends to a  unitary representation $\overline{\pi}_0^R$ of $\overline{G}^R$ into the unitary group of $H$. Moreover,
    the representation $\overline{\pi}_0^R$ has a locally integrable character with respect  to Haar measure on $\overline{G}^R$,
 denoted by "${\rm Tr}(\overline{\pi}_0^R(\cdot))$"=$\theta_{\overline{\pi}_0^R}$.
 \item(ii) The representation $\pi$ from Definition \ref{TC} also extends to a unitary representation $\overline{\pi}^R$ of  $\overline{G}^R$, and the equality $\overline{\pi}_0^R(g)=P_0\overline{\pi}^RP_0$ holds for $g\in \overline{G}^R.$
\item(iii) The set of elements $g \in \overline{G}^R$ with non-trivial group $\Gamma^{\rm st}_g$ has zero measure with respect to Haar measure on
$\overline{G}^R$.

Then
\item(i) For any element $g\in G$ as in assumption (ii), we have
\begin{equation}\label{berezin1}
\theta_{\overline{\pi}_0^R}(g)=
\mathop{\sum}\limits_{\gamma \in \Gamma}{\rm Tr}_{B(L)}(P_L\pi_0(\gamma g\gamma^{-1})P_L).
\end{equation}

\item(ii)
The character of the unitary representation  $\overline{\pi}_0$ restricted to $G$ (with the exception of  the points $g\in G$ such that
$\Gamma^{\rm st}_g$ is non-trivial) depends only on the character  of the representation $\overline{\pi}_0^R$.

\end{lemma}

\begin{rem}  We consider the  sum appearing in formula (\ref{realresult})
\begin{equation}
\phi_{\Gamma_0}(\sigma)=\mathop{\sum}\limits_{\Gamma
\mathop{\cup}
\limits_{i = 1}^{[\Gamma : \Gamma_0]}
\Gamma_0s_i}\; \mathop{\sum}\limits_{\theta \in s_i^{-1}\Gamma_0\sigma\Gamma_0s_i}{\rm Tr}(P_L\pi_0(\theta)P_L),\quad \sigma \in G.
\end{equation}

When $\sigma=e$ and $\Gamma_0$ is normal subgroup, this term is equal  to
$$[\Gamma:\Gamma_0]\mathop{\sum}\limits_{\Gamma =
\mathop{\cup}
\limits_{i = 1}^{[\Gamma : \Gamma_0]}
\Gamma_0s_i}\; \mathop{\sum}\limits_{\theta \in \Gamma_0}{\rm Tr}(P_L\pi_0(\theta)P_L)=
[\Gamma:\Gamma_0]\mathop{\sum}\limits_{\gamma \in \Gamma_0}{\rm Tr}(P_L\pi_0(\gamma)P_L).$$

Normalizing and taking the limit
$$\Phi_{\pi_0}(\sigma)=\mathop{\lim}\limits_{\mathop{\Gamma_0 \downarrow \; e}\limits_{\Gamma_0 \in \S}} \frac{\phi_{\Gamma_0}(\sigma)}{{\rm dim} H^{\Gamma_0}},\quad \sigma \in G,$$
we obtain a character of the group $G$ of the type considered in
\cite{PT}, \cite {DM}, \cite{BGK}, \cite{LB} \cite{Ve}.

 The characters we obtain in formula
Corollary \ref{plancherel} are of a different nature: they take infinite value at the identity and possibly take infinite value at other elements of the group.

\end{rem}

\

We exemplify below the content of  Theorem \ref{L0} and of Corollary \ref{theta111}, in the particular  case where  the  $\Gamma$-invariant vectors are the  automorphic forms.
We take  $G=\PGL(2,\mathbb Z[\frac{1}{p}])$, $p$ a prime number, $\Gamma$ the modular group,  and   the representation $\pi_0=\pi_n|_G, n\in \mathbb N, n\geq 2$ is obtained by restricting to $G$ a (projective)  unitary representation in the analytic, discrete series $(\pi_n)_{n\geq 2}$ of the semisimple Lie group $\PSL(2,\mathbb R)$. Let  $F$ be a fundamental domain for the action of the modular group on the upper halfplane.  Let  $\nu_0=( {\rm Im}z)^{-2}{\rm d}\bar{z}{\rm d} z$ be the canonical measure on $\bH$, that is invariant to the action of  $\PSL(2,\mathbb R)$ by Moebius transforms. Then, the projection $P_L$ introduced in Definition \ref {TC}, is the operator $M_{\chi_F}$ of multiplication with  the characteristic function of $F$,  on $H=L^2(\bH,  ({\rm Im}z)^{-2}{\rm d}\bar{z}{\rm d} z)$.  The unitary representation $\pi$ is the Koopman unitary representation of $G$ on $L^2(\bH,  ({\rm Im}z)^{-2}{\rm d}\bar{z}{\rm d} z)$, corresponding to the action of $G$ on $\bH$
 (see Examples \ref {koop} and \ref{auto}).

In this particular case the  projection $P_0$ is the Bergman projection onto the Hilbert space of the representation $\pi_n$, which is the space of analytic functions on $\bH$, square summable with respect to the measure $\nu_n=( {\rm Im}z)^{n-2}{\rm d}\bar{z}{\rm d} z.$ Moreover, the technical condition (v) from Definition \ref{TC} is  equivalent to the $L^2$-convergence condition   for Berezin's reproducing kernels (\cite{Be}) of the operators in the sum. This condition holds true in the particular case described here, because of the computations in \cite{Za} and \cite {GHJ}, Section 3.3.

For a bounded operator $A$ on the Hilbert space $H_n$, we denote by $\widehat {A}(\bar{z},\zeta)$, $z,\zeta\in \bH$ its Berezin symbol (\cite{Be}).
Formula (\ref{berezin}) implies in this case

\begin{cor}\label{deligne}
Let $\sigma$ be an element in $G$ with trivial group $\Gamma^{\rm st}_\sigma$. Then
\begin{equation}\label{mainresult}
 \theta_{\overline{\pi}_n^{\rm p}}(\sigma)= \big[ \mathop{\lim}\limits_{\mathop{\Gamma_0 \downarrow \; e}\limits_{\Gamma_0 \in \S}} \frac{1}{[\Gamma_0: (\Gamma_0)_g] }\big]
\mathop{\int}\limits_{\bH
} \widehat{\pi_n(\sigma)}(\bar{z}, z){\rm d}\nu_0(z)
\end{equation}
\begin{equation}\nonumber
= \big[ \mathop{\lim}\limits_{\mathop{\Gamma_0 \downarrow \; e}\limits_{\Gamma_0 \in \S}} \frac{1}{[\Gamma_0: (\Gamma_0)_\sigma] }\big]
\mathop{\int}\limits_{\bH}\frac{1}{\sigma z-\overline z} {\rm d}\nu_0(z)
= \big[ \mathop{\lim}\limits_{\mathop{\Gamma_0 \downarrow \; e}\limits_{\Gamma_0 \in \S}} \frac{1}{[\Gamma_0: (\Gamma_0)_\sigma] }\big] \theta_{\pi_n}(\sigma).
\end{equation}

\end{cor}

The fact that the  term  $\mathop{\int}\limits_{\bH
} \widehat{\pi_n(\sigma)}(\bar{z}, z){\rm d}\nu_0(z)$, in the above formula,  is   the character $\theta_{\pi_n}$ of the representation $\pi_n$ is proved  in  \cite{Ne}, using Berezin's quantization.  In this context, the formula for the  sum in (\ref{berezin})
is  computed, by a different method, in \cite{Za}.

In the last section, we apply the  construction of  $\Gamma$-invariant "virtual" vectors from the previous theorem,  to diagonal representations of $G$ of the form $\pi_0\otimes\pi_0^{\rm op}$, acting on the Hilbert space $H_0$. Here $\pi_0^{\rm op}$ is the complex,  conjugate representation associated to the representation $\pi_0$. We use a unitary equivalent representation of the Hilbert spaces consisting of $\Gamma$-invariant vectors.

 This representation is unitarily equivalent to the unitary representation $\mbox{\rm Ad\,}\pi_0$, acting on the Hilbert space consisting of the ideal of Hilbert-Schmidt operators  $\mathcal C_2(H_0) \subseteq B(H_0)$.
  The space $\V^{\Gamma_0}$ of "virtual" $\Gamma_0$-invariant vectors is  the von Neumann algebra of operators  $X\in B(H_0)$ that commute with $\pi_0(\Gamma_0)$. This algebra   is (see Example \ref{ad})  the commutant algebra $\mathcal A_{\Gamma_0}=\{\pi_0(\Gamma_0)\}'$.

 For a type II$_1$ factor $M$ with trace $\tau$, we denote the  Hilbert space associated to $\tau$ through the GNS construction (\cite{Ta}) by
 $L^2(M,\tau).$ This Hilbert space is obtained, via hilbertian  completion, from the scalar product induced by the trace on the vector space $M$.
By assumption, dim$_{\{\pi_0(\Gamma_0)\}}H_0$ is finite, hence we conclude (\cite{Ta}) that
 $\A_{\Gamma_0}$ is a type II von Neumann algebra.
 %
Then the  Hilbert space of $\Gamma_0$-invariant vectors is the $L^2$ space $L^2(\A_{\Gamma_0},\tau)$ associated to the type II$_1$ von Neumann factor $\A_{\Gamma_0}$. The inductive limit of these Hilbert spaces, when $\Gamma_0$ runs over  $\S$, has a natural interpretation in terms of the Jones basic construction (Example \ref {ad},  see also the construction in \cite{Ra4}).

 This is particularly interesting when $\pi_0$ is the representation $\pi_n$  mentioned above, ($G=\PGL(2,\mathbb Z[\frac{1}{p}])$, $p$ a prime,  $\Gamma$  the modular group, $\pi_n$ obtained, by restriction to $G$, from the discrete series of unitary representations
 $\PSL(2,\mathbb R)$).
 Consider  the unitary Koopman representation  $\pi_{\rm Koop}$  (see Example \ref {koop}) associated to the measure preserving action of $G$ on $\mathbb H$, endowed with the measure $\nu_0$ introduced above.

  Then, by \cite{Re} or by Berezin's quantization theory (\cite{Be}, see also \cite{Ra}), we have, up to unitary conjugation that
  \begin{equation}\label{sr}
   \pi_{\rm Koop}\cong\pi_n\otimes\pi_n^{\rm op},\quad n\geq 2.
   \end{equation}

Because of the above equivalence, the analysis of the representation of $G$ on the spaces of "virtual" $\Gamma_0$-invariant vectors for the representation
  $\pi_n\otimes\pi_n^{\rm op}$ leads to the analysis of the spaces of $\Gamma_0$-invariant vectors, $\Gamma_0\in S$,  for the Koopman representation, and of the corresponding unitary action of $G$ on these spaces. The latter representation corresponds to the action of the Hecke operators on Maass forms (\cite{Ma}). Using this method, we obtain in Section \ref{squareroot} (see also \cite{Ra}) concrete algebraic formulae relating  matrix coefficients of the representation $\pi_n$ with the expression of the Hecke operators associated to $\pi_{\rm Koop}$.

This method is useful in understanding the Koopman unitary representation $\pi_{\rm Koop}$, from which the action of the Hecke operators on Maass forms is derived. We are exploiting a natural "square root" of the representation $\pi_{\rm Koop}$, given by the representation in the discrete series, as in formula (\ref{sr}). The representation    $\pi_n\otimes\pi_n^{\rm op}$ is more malleable to handle, since using the operator algebra interpretation, the Hilbert spaces of "virtual" $\Gamma_0$-invariant vectors, $\Gamma_0 \in \S$, are canonically defined. We explain this construction below.

  Let $\mathcal R(G)$ be the von Neumann algebra generated by right convolution  operators $\rho(g), g\in G$, acting on $\ell^2(G)$. This is the commutant algebra of the algebra of left convolutors $\mathcal L(G)$ generated by the left convolution operators $\lambda_g, g \in G$, acting on the same Hilbert space.

  We consider the following algebras associated with the inclusions, $\Gamma \subseteq G$,  $K\subseteq \overline {G}$. Let $\H_0(K,\overline G)=\C(K\backslash \overline {G}/ K)$  be the algebra of double cosets of $K$ in $G$ (\cite{An}, \cite {Krieg}).    We denote the Hecke algebra
  $\mathbb C(\Gamma\backslash G/ \Gamma)$ generated by  double cosets of $\Gamma$ in $G$ by $\mathcal H_0(G,\Gamma)$.
  The latter algebra  is isomorphic to the Hecke algebra $\mathcal \H_0(K,\overline G)$. The Hecke algebra has a natural involution operation which is inducing a $\ast$-algebra structure.

The Hecke algebra  has a canonical, faithful $\ast$- representation (\cite{BC}) into $B(\ell^2 (\Gamma \backslash G))$.
We denote the elements of the standard basis of
$\ell^2 (\Gamma \backslash G)$ by $[\Gamma\sigma]$, $\sigma \in G$. The hypothesis $[\Gamma:\Gamma_\sigma]= [\Gamma:\Gamma_{\sigma^{-1}}]$, $\sigma \in G$,
that we assumed  on the equality of the indices  has the effect that the  state $\langle \cdot [\Gamma],[\Gamma]\rangle$, which is defined on $B(\ell^2 (\Gamma \backslash G))$, restricts to a trace on the image of the Hecke algebra.
   The reduced $C^\ast$-algebra of the Hecke algebra is the norm closure, in the  representation into $B(\ell^2 (\Gamma \backslash G))$ of the Hecke algebra $\mathcal H_0(G,\Gamma)$.   It will  be denoted by $\mathcal H_{\rm red}(G,\Gamma)$.  Clearly  $\mathcal H_0(G,\Gamma)\cong \H_0(K,\overline G)$.


   The $C^\ast$-algebra $C^\ast(\overline G)$ contains a canonical operator system (\cite{Pi}), associated to the maximal compact subgroup $K$ which we introduce in the next definition. We will explain below that $\ast$-representations of this operator system, as in formula
  (\ref {multiplicativity1}) in Lemma \ref{rephecke},  encode all the information about the representation
  $\pi_0$. These morphisms are the basic building blocks of the Hecke
  algebra representation associated to a representation of the form
  $\pi_0\otimes\pi_0^{\rm op}$ (see Theorem \ref{double} and Theorem \ref{multhecke}).  The essential tool for analyzing  unitary representations of $G$  of the form $\pi_0\otimes\pi_0^{\rm op}\cong {\rm Ad\, }\pi_0$ is a canonical representation of the Hecke algebra $\H_0(G,\Gamma)$
  into  $\mathcal R(G) \otimes B(L)$ that will be described below.

   \begin{defn}\label{thecanonicalos}
    Let $L(K, \overline {G})=\C(\chi_{\sigma K} | \sigma \in G)$ be  the linear subspace of  $C^{\ast}(\overline{G})$ generated by the characteristic functions of right cosets. Then  $L(K, \overline {G})^\ast=\C(\chi_{K\sigma } | \sigma \in G)$.
%
We also consider  the space
     $$\tilde L(K, \overline {G})= L^\infty (\overline {G},\mu) L(K, \overline {G})\subseteq C^\ast( \overline {G}\rtimes L^\infty (\overline {G},\mu)). $$
   We consider the following operator systems (\cite{Pi}) :
   \begin{equation}\label{canonicalos}
\O(K, \overline {G})=L(K, \overline {G})\cdot (L(K, \overline {G}))^\ast \subseteq C^{\ast}(\overline{G}),
\end{equation}
\begin{equation}\nonumber
\tilde \O(K, \overline {G})=
 \tilde L(K, \overline {G})\cdot (\tilde L(K, \overline {G}))^\ast\subseteq C^\ast( \overline {G}\rtimes L^\infty (\overline {G},\mu)).
\end{equation}
\noindent In the above formula the product is calculated in the $C^\ast$-algebra $C^\ast(\overline G)$ and, respectively, in the $C^\ast$-algebra  $C^\ast( \overline {G}\rtimes L^\infty (\overline {G},\mu))$.

Then,  clearly:


(i) $\O(K, \overline {G})$ is linearly generated by the characteristic functions of the form
$\chi_{\sigma_1 K\sigma_2},\ \sigma_1, \sigma_2 \in G$.

(ii) $L(K, \overline {G})$ and $(L(K, \overline {G}))^\ast$ are subspaces of $\O(K, \overline {G})$.

(iii) We have a canonical pairing
\begin{equation}\label{pairingos}
L(K, \overline {G})\times L(K, \overline {G})\rightarrow \O(K, \overline {G}).
\end{equation}
The pairing maps $x,y \in L(K, \overline {G})$ into $xy^\ast\in \O(K, \overline {G}) $.

Let $\mathcal E$  be a $C^\ast$-algebra.  We will say that a linear map $\Phi: \O(K, \overline {G})\rightarrow \mathcal E$ is a $\ast$-representation of the operator system $\O(K, \overline {G})$ if, by definition
 $$\Phi(xy^\ast)=\Phi(x)(\Phi(y))^\ast,\quad  x,y \in L(K, \overline {G}).$$
\noindent  A $\ast$-representation for the operator system $\tilde\O(K, \overline {G})$ will have to be in addition linear as a bimodule over the algebra  $L^\infty (\overline {G},\mu)$.

 The Hecke algebra $\mathcal H_0(K, \overline {G})$ is the intersection
$L(K, \overline {G})\cap  (L(K, \overline {G})^\ast$. Clearly it is closed with respect to the above pairing operation.
Obviously, a  $\ast$-representation  of the operator system $\O(K, \overline {G})$ becomes, when restricted to the Hecke algebra,
a $\ast$-algebra representation.

 \end{defn}

 \

 \

Let  $G, \Gamma, \pi, \pi_0$, $P_0$, $P_L$ be as in  Definition \ref{TC}. We assume, for simplicity of the presentation, for the following results,  that the groups $\Gamma,G$ have infinite, non-trivial conjugacy classes. This assumption implies that the von Neumann algebras associated below with the above groups have unique traces.

 We construct  $\ast$-representations of the operator system $\O(K, \overline {G})$  that  take values in the
     associated von Neumann algebras described below.
%
We consider the following von Neumann algebras:
\begin{equation}\label{defa}
\A = \{\pi(\Gamma)\}' \cong \cR(\Gamma)\; \overline{\otimes}\; B(L) \subseteq \B = \cR(G)\; \overline{\otimes}\; B(L),
\end{equation}
%
\begin{equation}\nonumber
 \A_0 = \pi_0(\Gamma)' = P_0\A P_0.
 \end{equation}
 Note that $\A,\B$ are type II von Neumann algebras. 
  If the space $L$ is infinite, then $\A,\B$ are  type II$_\infty$ von Neumann algebras.
In the representation for $\{\pi(\Gamma)\}'$ introduced in formula (\ref{defa}), the projection  $P_0$, which by hypothesis commutes with  $\pi(\Gamma)$ and thus belongs to $\A$,  has the  formula:
\begin{equation}\label{newproj}P_0=
\mathop{\sum}\limits_{\gamma \in \Gamma}\rho(\gamma) \otimes P_L\pi_0(\gamma)P_L
\in \cC_1(\A).
\end{equation}

  For a given von Neumann algebra $\mathcal M$, endowed with a faithful, semifinite trace $T=T_\mathcal M$, we denote by $\mathcal
   C_1(\mathcal M)$ the ideal of trace class operators associated to $\mathcal M$ (\cite{Ta}).  If $\mathcal M$ is a type II$_\infty$ von Neumann algebra endowed with faithful semifinite trace $T$, we denote by $\cC_1(\mathcal M, T)$ the ideal of trace class operators in $\mathcal M$ (\cite{Ta}). If the dependence on the trace $T$ is contextual, we write just $\cC_1(\mathcal M)$.

We associate to the  unitary representation $\pi_0$  a  $\ast$-representation of the operator system constructed in Definition \ref {thecanonicalos}.

 \begin{lemma}\label{rephecke}  We assume that $G, \Gamma, \pi_0$, $P_0$, $P_L$ are as in  Definition \ref{TC}.

\noindent For   a coset $C = g\overline{\Gamma}_0$ in $\overline{G}$  define, in analogy with formula (\ref{newproj}),
$$
\tilde{\Phi}_{\pi_0, L}(C) = \mathop{\sum}\limits_{\theta \in C}
\rho(\theta)\otimes P_L\pi_0(\theta)P_L \in \cC_1(\mathcal B).
$$

\noindent Then
\item (i)
 The application $\tilde{\Phi}_{\pi_0, L}$ restricted to $
\O(K, G)$  is  a $\ast$-representation of the operator system
$
\O(K, G)$.
 In particular
\begin{equation}\label{multiplicativity1}
\tilde{\Phi}_{\pi_0, L}(\chi_{\sigma_1 K})\big[\tilde{\Phi}_{\pi_0, L}(\chi_{\sigma_2K})\big]^{\ast} = \tilde{\Phi}_{\pi_0, L}(\chi_{\sigma_1 K\sigma_2^{-1}}),\quad \sigma_1, \sigma_2 \in G.
\end{equation}

\item (ii)
Because of formula (\ref{newproj}), we have  $\tilde{\Phi}_{\pi_0, L}(\chi_{K})=P_0$. Hence,  because of (i),
$$
\begin{aligned} \tilde{\Phi}_{\pi_0, L}(\chi_{\sigma_1 K}) & = \tilde{\Phi}_{\pi_0, L}(\chi_{\sigma_1 K})P_0 , \\
 \tilde{\Phi}_{\pi_0, L}(\chi_{K\sigma_1}) & = P_0\tilde{\Phi}_{\pi_0, L}(\chi_{K\sigma_1}),\quad \sigma_1 \in G.
 \end{aligned}
$$


\item (iii) Because of (ii), the  restriction $\tilde{\Phi}_{\pi_0, L}$  to the Hecke
 algebra $ \H_0(K,\overline G)\cong \mathcal H_0(\Gamma, G)$ takes
   values in the algebra $ \A_0=P_0\B P_0 $.

\item (iv)  $\tilde{\Phi}_{\pi_0, L}|_{\H_0(K,\overline G)}$
    is trace preserving, and hence it        extends continuously to C$^\ast$-representation of the reduced $C^\ast$-Hecke algebra $\H_{\rm red}(\Gamma , G)$ with values in $\A_0$.


\end{lemma}
  The  $\ast$-algebra representation of the Hecke algebra from point (iii)  is used to describe the Hecke operators on $\Gamma$-invariant vectors associated to the diagonal unitary  representation $\pi_0\otimes\pi_0^{\rm op}$ of $G$. As explained above, the latter representation is
  unitarily equivalent to ${\rm Ad\, }\pi_0$, and the Hilbert space of $\Gamma$-invariant vectors is canonically identified to the $L^2$ space (\cite{Ta}) associated to the commutant von Neumann algebra $\{\pi_n(\Gamma)\}'$.
  The formula of the Hecke operators is computed in the next theorem. The case where $\dim_\C L=1$ was used in [Ra] to obtain estimates on the essential spectrum of Hecke operators acting on Maass forms.   The general case of arbitrary Murray von Neumann dimension is treated in \cite{Ra1}.

  The Hecke operators are automatically completely positive maps. They are obtained using von Neumann algebras expectations (\cite{Sak}, \cite{Ta}).
Let $$E_{P_0(\cR(\Gamma) \otimes B(L))P_0}^{P_0(\cR(G) \otimes B(L))P_0}$$
\noindent  be the canonical normal conditional expectation, mapping the type II$_1$ factor $P_0(\cR(G) \otimes B(L))P_0$
\noindent  onto the subfactor $P_0(\cR(\Gamma) \otimes B(L))P_0.$

  \

\begin{thm}\label{double}
   We use the definitions and notations introduced above.
%
%
The Hecke operator  $\Psi({[\Gamma\sigma\Gamma]})$, associated to the representation ${\rm Ad\, }\pi_0$, corresponding to a coset $[\Gamma\sigma\Gamma]$ for $\sigma$ in $G$,
is a selfadjoint operator acting on   the space
$L^2(\A_0,\tau)=L^2( \{ \pi_0(\Gamma) \}' ,\tau).$

 Then, the formula for $\Psi({[\Gamma\sigma\Gamma]})$ is determined  by its values on the  algebra $\A_0$.
%
%
%
 For $\sigma \in G$, the Hecke operator $\Psi({[\Gamma\sigma\Gamma]})$ associates to
$$X \in \pi_0(\Gamma)' = \A_0 = P_0\B P_0 = P_0(\cR(\Gamma) \otimes B(L))P_0
$$ \noindent  the operator
\begin{equation}\label{formulapsi}
\Psi({[\Gamma\sigma\Gamma]})(X)=
E_{P_0(\cR(\Gamma) \otimes B(L))P_0}^{P_0(\cR(G) \otimes B(L))P_0}(\tilde{\Phi}_{\pi_0, L}(\Gamma\sigma\Gamma)X(\tilde{\Phi}_{\pi_0, L}(\Gamma\sigma\Gamma))^\ast).
\end{equation}

\end{thm}

  \

  This is Theorem 3.2 in \cite{Ra1}, generalizing results in \cite{Ra}. The statement is adapted to the framework of the present paper. The present formalism proves that once a choice for the space $L$ has been made, the Hecke algebra representation is canonical.

  We construct  below  a canonical $\ast$-representation $\overline {\mathcal D\Phi}$ of the Hecke algebra
  $\H_0(K,\overline G)$ into a  canonical "double" algebra (see below) such that all the representations of the Hecke algebra,
  as in formula (\ref{formulapsi}), are obtained by composing  $\overline {\mathcal D\Phi}$ with a quotient map.

  We will perform  the construction of the representation $\overline {\mathcal D\Phi}$  only  in the case where the   Neumann dimension  dim$_{\{\pi_0(\Gamma\}''}H_0$ is equal to 1.  In this context $\mathcal A_0$ is simply the algebra $\mathcal R(\Gamma)$.
The  argument may be easily extended to cover the case of von Neumann dimension larger than $1$, but the formulae become more complicated.

  The fact that the von Neumann dimension of the type II von Neumann algebra generated by $\pi_0(\Gamma)$  is 1 means that the representation $\pi_0|_\Gamma$ admits a cyclic trace vector $\zeta$.
  In this case, the $\ast$-representation $\tilde{\Phi}_{\pi_0, L}$ introduced in Definition \ref{rephecke} is replaced
  by   a representation $t$, constructed below,  of the canonical system introduced in Lemma \ref{thecanonicalos}, taking values in $\L(G)$. Recall that the algebra $\L(G)$   is anti-isomorphic to $\mathcal R(G)$.


We define, for $\sigma_1,\sigma_2\in G$, and for a subset of $G$ of the form $A=\sigma_1\Gamma\sigma_2 $,
 \begin{equation}\label{ta}
 t(\sigma_1K\sigma_2)= \sum_{\theta\in A}
 \overline{ \langle \pi_0(\theta)\zeta,\zeta\rangle}\lambda_\theta \in \L(G).
 \end{equation}
 Then, by Lemma \ref{rephecke}, letting the space $L$ be $\mathbb C\zeta$, we infer (see  Remark \ref{1dim}) that  $t$ defines a
 $\ast$-representation of the operator system
%
$\O(K, G)$ with values in the von Neumann algebra $ \mathcal L(G)$
(the complex conjugation in formula (\ref{ta}) is due to the fact that, via the canonical anti-isomorphism, we are switching from the algebra  $\mathcal R(G)$ to the algebra $ \mathcal L(G)$).

 The Hecke algebra $\ast$-representation $\Psi$, introduced in  formula (\ref {formulapsi}),
 takes now values in  $\mathcal L(\Gamma)$ and is given by the formula:

%
\begin{equation} \label{oldpsi}
\Psi([\Gamma\sigma\Gamma])(X)= E^{\mathcal L(G)}_{\mathcal L(\Gamma)}(t({\Gamma\sigma\Gamma}) X (t({\Gamma\sigma\Gamma}))^\ast),\quad \sigma \in G.
\end{equation}
The linear operator $\Psi([\Gamma\sigma\Gamma])$, $\sigma \in G$  extends obviously to the GNS space $\ell^2(\Gamma)$ (see e.g. \cite {Ta}),
taken with respect the canonical trace.

The Hecke operators
$\Psi([\Gamma\sigma\Gamma])$, $\sigma \in G$ are all completely positive and unital. Hence the representation in formula (\ref{oldpsi}) extends, by forgetting the algebra structure on the range space, to a $\ast$-representation $\Psi_0$ of $\mathcal H_0(G,\Gamma)$ with values  into $B(\ell^2(\Gamma)\ominus \mathbb C 1)$.
As explained in Example \ref {auto} (see also \cite{Ra}, \cite {Ra4}), the validity of the Ramanujan-Petersson Conjecture on Maass forms, in the particular example $G=\PGL(2,\mathbb Z[\frac{1}{p}])$,  is equivalent to the fact that the representation $\Psi_0$ extends to a representation of the reduced  $C^\ast$-Hecke algebra $\mathcal H_{\rm red}(G,\Gamma)$.

We will  prove that the representation $\Psi$ of the Hecke algebra introduced in formula (\ref{oldpsi})  is obtained  from a canonical representation $\overline{ \mathcal D \Phi}$ of the Hecke algebra.

We introduce the following two crossed product $C^\ast$ -algebras. They are a two variable extension  of the Roe algebras (\cite {Roe}, \cite {Bo}).
\begin{defn}\label{doublecrossedproduct}
We consider the following crossed product  $C^{\ast}$-algebras:

\begin{equation}\label{doubleroe}
\mathcal D\mathcal R(G)=\chi_K(C^{\ast}((G \times G^{\rm op}) \rtimes L^{\infty}(\overline{G},\mu)))\chi_K,
\end{equation}
\begin{equation}
\mathcal D\mathcal R(\overline G)=\chi_K(C^{\ast}((\overline G \times  \overline G^{\rm op}) \rtimes L^{\infty}(\overline{G},\mu)))\chi_K
\end{equation}

The $C^\ast$-algebra $\mathcal D\mathcal R(G)$  has a canonical, Koopman type, representation $\alpha$ into $B(\ell^2(\Gamma))$ obtained as follows: let  $G \times G^{\rm op}$ act as groupoid on $\Gamma$ by left and right multiplication, the operation being defined whenever the result of the multiplication belongs to $\Gamma$. On the other hand $C(K)$ acts by multiplication on $l^\infty(\Gamma)$ and hence on $\ell^2(\Gamma)$.

 \end{defn}

\
We observe that both algebras introduced in the above definition are corners (reduced by the projection $\chi_K\in L^{\infty}(\overline{G},\mu)$) in the larger crossed product $C^\ast$-algebras, corresponding to the measure preserving actions of $G \times G^{\rm op}$ and respectively $\overline G \times  \overline G^{\rm op}$ on $\overline G$. Since the action is measure preserving, the two algebras have obvious reduced crossed product $C^\ast$-algebra counterparts.

In the following statement, we prove that the  Hecke algebra representation $\Psi$ introduced in formula (\ref{oldpsi}), which is  in fact the Hecke algebra representation   associated with  the representation $\pi_0\otimes\pi_0^{\rm op}$, is  obtained from an intrinsic representation  of the Hecke algebra, denoted by $\mathcal D\Phi$, with values in  the "double"  algebra $\mathcal D\mathcal  R(G)$. To obtain the Hecke algebra representation $\Psi$, one composes the $\ast$-algebra representation   $\mathcal D\Phi$ with the Koopmann type representation $\alpha$ of the algebra $\mathcal D\mathcal R (G)$
introduced in Definition \ref{doublecrossedproduct}.

 Since the $\ast$-algebra representation    $\mathcal D\Phi$ extends (as it preserves the trace)
 to the reduced $C^\ast$-Hecke algebra,  the obstruction (if any) to extending the representation $\Psi_0$ to the reduced
$C^\ast$-algebra Hecke algebra $\mathcal H_{\rm red}(G,\Gamma)$ (which, as explained above, in the case of the representation $\pi_n$, is the obstruction for the Ramanujan-Petersson Conjecture on Maass forms to hold true, (\cite{Ra})),  lies in   the kernel of the above canonical representation $\alpha$, eventually  intersected with the image  of $D\Phi$.

If $f$ is a function on $\overline G$, we denote the corresponding convolution element in $C^\ast (\overline G)$ by $L(f)$. We have:


  \begin{thm}\label{multhecke}

 The following statements hold true:

\item(i)
The  correspondence
 $$
[K\sigma K] \to \chi_K(L(\chi_{K\sigma K}) \otimes L(\chi_{K\sigma K})^{\rm op})\chi_K,\quad \sigma \in G,
$$
\noindent extends by linearity to  a $\ast$-algebra representation $\overline{ \mathcal D \Phi}$ of the Hecke algebra $\H_0(K,G)$  into $\mathcal D\mathcal R(\overline G)$. This representation is trace preserving, and hence it obviously extends  to $\H_{\rm red}(\Gamma, G)$ when values are considered in the reduced crossed product $C^\ast$-algebra associated with $\mathcal D\mathcal R(\overline G)$.

\item (ii)  The   $\ast$-representation $t$ of the operator system $\O(K,\overline G)$, introduced in formula (\ref{ta}), extends to a
$\mathcal D\mathcal R( G)$-valued $\ast$-representation $t_2$ of an  operator system $\mathcal O$ contained   in   $\mathcal D\mathcal R(\overline G)$. The  operator system $\mathcal O$ contains the image of $\overline{ \mathcal D \Phi}$.

\item(iii)
The composition ${ \mathcal D \Phi}= t_2 \circ \overline{ \mathcal D \Phi}$, mapping
$$[\Gamma\sigma\Gamma] \to \chi_K(t({\Gamma\sigma\Gamma}) \otimes (t({\Gamma\sigma\Gamma}))^{\rm op})\chi_K,$$ extends to a   $\ast$-algebra representation of the Hecke algebra $\H_{\rm red}(\Gamma, G)$
into the reduced crossed product $C^\ast$-algebra  associated with $\mathcal D\mathcal R(\overline G)$.

\item(iv)
%
Let $\alpha$ be as in Definition \ref{doubleroe}, and let  $\Psi$ be the Hecke algebra introduced in
formula (\ref{oldpsi}).
Then
\begin{equation}\label{mformula}
\Psi= \alpha \circ \mathcal D \Phi= \alpha \circ t_2 \circ \overline{ \mathcal D \Phi}.
\end{equation}

\end{thm}

  {\bf Acknowledgement} We are deeply indebted to R. Grigorchuk, A. Lubotzky and L. P\u aunescu for encouraging us into writing this paper. We are deeply indebted to U. Haagerup, P. De La Harpe,
P. Kutzko, H. Moscovici,  R. Nest, T. Steger, Paul Garett, Vicentiu Pa\c sol, Alexandru Popa, and  the user plusepsilon.de on mathoverflow  for several comments on questions related to this paper. We are deeply indebted to
R. Grigorchuk and Tatiana Smirnova-Nagnibeda for discussions around the topic of this paper, and in particular on central characters of group. We are deeply indebted to J. Petersson and A. Thom for explaining the signification of the  results \cite{PT}.   We are deeply indebted to N. Ozawa for providing to the author for publication, his edited   notes (\cite{Ra2}), on the author's paper \cite{Ra}. We are deeply indebted to F. Boca for  comments on this paper, and for helping the author improve the style of the presentation. The author is thanking the Department of Mathematics in Copenhagen, for the warm welcome during the last stage of this work.

\section{Outline of the paper}

We outline  the construction of Hilbert spaces of "virtual" $\Gamma$-invariant vectors and of the unitary action  of $G$ on the inductive limit of these spaces.
We recall from introduction that $\Gamma$ is an almost normal subgroup
of $G$: that is, for all $\sigma$ in $G$, the group $\Gamma_{\sigma} = \sigma\Gamma\sigma^{-1}\cap\Gamma$ has finite index in $\Gamma$.
We assume in this paper that, for all $\sigma \in G$, the values of the  group indices $[\Gamma: \Gamma_{\sigma}]$ and $[\Gamma: \Gamma_{\sigma^{-1}}]$ are equal. This hypothesis is automatic if $\Gamma$ is a group with infinite non-trivial conjugacy classes, in the presence of the representation $\pi_0$ as in Definition \ref{TC} (e.g. by Jones's index theory \cite{Jo}).
 We also assume that the family $\S$, generated through the intersection operation by the subgroups $\Gamma_{\sigma}, \sigma \in G$, separates points of $\Gamma$. Let $K$ be the profinite completion of $\Gamma$, with respect to the family $\S$.

We first consider a unitary representation $\pi$ of $G$ having properties (i), (ii) from Definition \ref{TC}. In this case $\pi|_\Gamma$ is an   integer multiple of the regular representation.
To construct $\Gamma$-invariant vectors, we assume that $H$ is contained in a larger vector space $\V$, and that the representation $\pi$ extends to a representation $\pi_{\V}$ of $G$ into the linear isomorphism group of $\V$. We assume that the space  $\V^{\Gamma}$, consisting of vectors in $\V$ fixed by $\pi_{\V}(\Gamma )$, is non-trivial. To construct the representation $\overline{\pi}^{\rm p}$, one
 considers simultaneously the spaces $\V^{\Gamma_0}$ of  vectors in $\V$ fixed  by the action of $\pi_{\V}(\Gamma_0 )$, where $\Gamma_0$ is any group conjugated in $G$ to a subgroup in $\S$.
Note that $\pi_\V(\sigma)$, for $\sigma$ in $G$, will  move the vector space  $\V^{\Gamma_0}$ into $\V^{\sigma\Gamma_0\sigma^{-1}}$.

We  construct  a Hilbert space structure on the spaces $\V^{\Gamma_0}$, which determines  Hilbert spaces $H^{\Gamma_0}$ of $\Gamma_0$-fixed vectors, $\Gamma_0\in\S$. On the  scalar product  on the spaces of invariant vectors, we impose the condition  that   the  inclusions
$$
H^{\Gamma_0} \subseteq H^{\Gamma_1}, \quad
\Gamma_0, \Gamma_1 \in \S,\quad  \Gamma_1 \subseteq \Gamma_0,$$
\noindent are isometric, and   that $\pi_{\V}(\sigma)$  maps $H^{\Gamma_0\cap \Gamma_{\sigma^{-1}}}$ isometrically onto $H^{\sigma\Gamma_0\sigma^{-1}\cap\Gamma_\sigma}$, for $\sigma \in G$.

Using this procedure, we obtain in formula (\ref{scalar}) a prehilbertian space structure on the reunion $\mathop{\bigvee}\limits_{\Gamma_0 \in \; \S} H^{\Gamma_0}$. Let $\overline{H}^{\rm p}$ be the Hilbert space completion of the reunion.
Clearly $\pi_{\V}$ induces a representation of $G$ into the isomorphisms of $\mathop{\bigvee}\limits_{\Gamma_0 \in \; \S} \V^{\Gamma_0}$.

The conditions we are imposing on the scalar products on the Hilbert  spaces $H^{\Gamma_0}$, $\Gamma_0\in S$ imply that $\pi_{\V}$
induces an unitary representation $\overline{\pi}^{\rm p}$ of $G$ into the unitary group $\U(\overline{H}^{\rm p})$ of the Hilbert space $\overline{H}^{\rm p}$.
The notation $\overline{\ \cdot\ }^{\rm p}$  over $\pi$ and $H$,  stands for  the above  completion operation, consisting into passing from $\pi$ to the unitary representation $\overline\pi^{\rm p}$ on Hilbert space completion of the reunion of the space of $\Gamma_0$ invariant vectors, $\Gamma_0\in \S$.

 We recall that the Schlichting completion    $\overline{G}$ (see e.g. \cite{Sch}, \cite{Tz}, \cite{KLM}, \cite{LLN}) of the discrete group $G$ with respect to the subgroups
 in $\S$  is the locally compact, totally disconnected group obtained as the disjoint union of the double cosets $K\sigma K$, with the obvious multiplication relation, where $\Gamma\sigma\Gamma$ runs over a set of representatives for double cosets for $\Gamma$ in $G$ (see also \cite{Bi}, \cite{BC}, \cite{Hal}).
 
  Let   dimension $d_{\pi}={\rm dim}_{\{\pi(\Gamma)\}''}H_0\in \mathbb N $ be the von Neumann dimension of $H$ as a module over the type II factor $\{\pi(\Gamma)\}''$  (see \cite{GHJ}), section 3.3 for definitions and  notations). We prove (Lemma \ref{compact}) that  the left regular representation $\lambda_K$ of $K$ into the unitary group $\U(L^2(K, \mu))$ has multiplicity $d_{\pi}$   in $\overline{\pi}^{\rm p}|_{K}$.

Let $\mu$ be the normalized Haar measure on $K$ and extend it on the locally compact group $\overline{G}$ to the Haar measure $\mu$, normalized by $\mu(K) = 1$.
The assumption that $[\Gamma : \Gamma_{\sigma}]=[\Gamma : \Gamma_{\sigma^{-1}}]$ for all $\sigma$ in $G$, implies that the measure $\mu$ is bivariant under left and right translations by elements in the group $\overline{G}$.

We extend the construction of the above unitary representation to the case where $\pi_0$ is a unitary representation as in Definition \ref{TC}. In this case the unitary representation  $\pi_0|_\Gamma$ is no longer a multiple of the left regular representation of $\Gamma$, but it is a subrepresentation of a larger representation $\pi$ with the above properties.

We prove that the above construction that we performed to obtain the unitary representation $\overline\pi^{\rm p}$ may be repeated for $\pi_0$.
We obtain a unitary representation $\overline\pi_0^{\rm p}$ of $\overline G$ that is associated with the action of the initial group $G$ on the spaces of $\Gamma_0$-invariant vectors, $\Gamma_0\in \S$.
The representation $\overline{\pi}_0^{\rm p}$ naturally extends to the $C^\ast$-algebra of the locally compact group
$\overline{G}$.
It also  extends   to a $C^{\ast}$-algebra representation of the full $C^{\ast}$-algebra
associated simultaneously with the groups $G$ and $\overline G$
(see Definition \ref{two}).


 In Theorem \ref{a0formula}, we note that "block matrix coefficients" of the representation $\overline{\pi}_0^{\rm p}$ are   Hecke operators associated to the representation $\overline{\pi}_0^{\rm p}$, corresponding to level $\Gamma_0$, where the group $\Gamma_0\S$ is determined by the size of the "block". We use this in Theorem \ref{L0} to determine the explicit formula  (see formula (\ref {heckematrix0})) of the Hecke operators in terms of the data from the representation $\pi_0$.

 One outcome of this paper is  the relation between the
representations $\pi_0$ and $\overline{\pi}_0^{\rm p}$.
The representation $\overline{\pi}_0^{\rm p}$ is a type I representation of the group $\overline{G}$. Hence, it has an associated character (\cite{HC}, \cite{Sal}, \cite{GeGr}),  which we denote by $\theta_{\overline{\pi}_0^{\rm p}}$= "${\rm Tr}$"$\; \overline{\pi}_0^{\rm p}$.
In Corollary \ref{plancherel}, we prove that the values of the character
$\theta_{\overline{\pi}_0^{\rm p}}(g)$ computed at $g\in G$, are determined by  summing   over cosets  the values of  a positive definite function $\phi$  on $G$, introduced in formula (\ref{pdf}). The same positive definite function  also  determines the formula of  "${\rm Tr}\; \pi_0(g)$" of the original representation $\pi_0$.


The basic   model of unitary representations  as above  is the representation of the spaces of $\Gamma_0$-invariant vectors, in the case of  the left regular representation $\lambda_G$ of $G$ into $\U(l^2(G))$ (a more detailed exposition  Example \ref {regular} in Section \ref{theexamples}). In this case the Hilbert spaces $H^{\Gamma_0}$, $\Gamma_0 \in \S$, are the Hilbert spaces $l^2(\Gamma_0 \setminus G)$ with scalar product normalized, so that the inclusions $l^2(\Gamma_0 \setminus G) \subseteq l^2(\Gamma_1 \setminus G)$ are isometric for $\Gamma_1 \subseteq \Gamma_0$. The Hilbert space $\overline{H}^{\rm p}$ is simply $L^2(\overline{G}, \mu)$, and in this particular case the representation $\overline{\pi}^{\rm p}$ is simply the left regular representation of $\overline{G}$ into the unitary group $\U(L^2(\overline{G}, \mu))$.

We recall from the introduction that the Hecke algebra $\H_0(\Gamma, G)$
of double cosets of $\Gamma$ in $G$ has a canonical $\ast$-algebra embedding into $B(\ell^2(\Gamma\backslash G))$. The closure in the uniform norm is a $C^\ast$-algebra $\H_{\rm red}(\Gamma, G)$ called  the reduced $C^\ast$-Hecke algebra, by analogy with reduced $C^\ast$-algebra of a discrete group (\cite{BC}, \cite{Ha},\cite{Tz}, \cite{BCH}, \cite{Cu}). The  representation  $$\H_{\rm red}(\Gamma, G)\subseteq B(\ell^2(\Gamma\backslash G))$$
\noindent is called the left regular representation of the Hecke algebra. The commutant is generated by the right quasi-regular representation $\rho_{\Gamma\backslash G}$ of $G$
into the unitary group of  $\ell^2(\Gamma\backslash G)$ (see e.g. \cite{BC}).

The content of the Ramanujan-Petersson Problem is the determination of bounds on the growth of matrix coefficients and eigenvalues for representations of the Hecke algebra, on Hilbert spaces of $\Gamma$-invariant vectors associated to unitary representations $\pi$ of  $G$ as above.  The Ramanujan-Petersson Conjecture is asking in fact (see \cite{Ra}) when the matrix coefficients of the representation of  $\H_0(\Gamma,G)$, associated to the unitary representation $\pi$, are weakly limits of convex combinations of matrix coefficients of the Hecke algebra coming from  left regular representation of the Hecke algebra.

 In the terminology of this paper this is equivalent to determining when the  spherical functions (matrix coefficients
  corresponding to vectors fixed by $K$) associated to the representation $\overline {\pi_0}^{\rm p}$ of $\overline G$ are weakly contained in the left regular representation of the Hecke algebra. Clearly, this is equivalent to the weak containment of the representation $\overline {\pi_0}^{\rm p}$  in the left regular representation of $\overline G$, acting  by left translations on $L^2(\overline G,\mu)$. The Ramanujan-Petersson Conjecture has been proven to hold true for automorphic forms by Deligne (\cite{De}). For Maass forms (\cite{Ma}), the general problem is open (see   \cite{Sh}, \cite{Sar1},   \cite{Hej},\cite{BLS},  \cite{Ra}).
We  formulate the following problem:

 {\bf Problem.} {\it
[Generalized Ramanujan-Petersson problem]:  Determine  conditions on the representation $\pi_0$ such that   $\bar{\pi_0}^{\rm p}$ is weakly contained in the left regular representation $\lambda_{\overline{G}}$ of $\overline{G}$ on
$L^2(\overline{G}, \mu)$. It is enough to  determine conditions on the unitary representation $\pi$ of $G$, so that the unitary representation   $\bar{\pi_0}^{\rm p}|_G$  of $G$ is weakly contained in the  unitary representation $\lambda_{\overline{G}}|_G$.}

In the case of the unitary representation $\pi_n$ of $G=\PGL(2,\Z[\frac{1}{p}])$, $p$  a prime number, $n\in \mathbb N$, obtained by restriction from the discrete series of representations of $\PSL(2,\R)$, the representations $\overline{\pi}_n^{\rm p}$ encode all the group harmonic analysis information about the spaces of automorphic forms
(here the group $\Gamma$ is $\PSL(2,\Z)$). The spherical matrix coefficients of $\overline{\pi}_n^{\rm p}$ encode the information about the eigenvalues of Hecke operators.
\vskip20pt

Below we present examples of representations $\pi$ and of the associated representation $\overline{\pi}^{\rm p}$. For a more detailed exposition see Section \ref{theexamples}. The easiest case is when a unitary representation $\pi$ of $G$ as above has also the property that
$\pi|_\Gamma$ is an integer multiple of the left regular representation of $\Gamma$. In the terminology of Murray-von Neumann dimension (see \cite{GHJ}, section 3.3), this is the case when ${\rm dim}_{\{\pi(\Gamma)\}''}H$ is an integer.  In this case, as  explained above, there exists a Hilbert subspace $L$ of $H$ with the property introduced next. 
\begin{defn}\label{wandering}
Let $\pi$ be a unitary representation of $G$ as above. We consider a subspace $L$ of $H$  such   that $\pi(\gamma)L$ is orthogonal to $L$ for $\gamma \neq e$. We will call such a space $\Gamma$-wandering subspace for $\pi$. If in addition we have that  $H = \mathop{\bigvee}\limits_{\gamma \in \Gamma}\pi(\gamma)L$, we will refer to   a subspace, with the two properties, as to  a $\Gamma$-wandering, generating subspace.
\end{defn}

\

In the above context, the Hilbert spaces of $\Gamma_0$-invariant vectors are canonically identified with   the Hilbert spaces $L\otimes l^2(\Gamma_0 \setminus G)$, $\Gamma_0\in \S$. Recall that $K$ is the profinite completion of $\Gamma$ with respect to the subgroups in $\S$. The space $\overline H^{\rm p}$ is identified to $L\otimes L^2(K,\mu)$. The problem consists in   the identification of the  representation  of $G$ and $\overline G$ on $\overline H^{\rm p}$.

One case (see Example \ref {koop}) in which the above situation occurs is when $(\cX, \nu)$ is an infinite measure space on which $G$ acts by measure preserving transformations. In this case $H=L^2(\cX, \nu)$ and $\pi$ is the Koopmann representation $\pi_{\rm Koop}$ (see e.g. \cite{Ke})
\begin{equation}\label{defkoop}
(\pi_{\rm Koop}(g)f)(x) = f(g^{-1}x), \quad x \in \cX, \; g \in G, \; f \in L^2(\cX, \nu).
\end{equation}

We assume that the  restriction of the action of $G$ to $\Gamma$ admits
 a fundamental domain  $F$. Then we may take $L=L^2(F,\nu|_{F})$. This is a $\Gamma$-wandering, generating subspace, associated with the representation  $\pi_{\rm Koop}$.
In this case ${\rm dim}_{\{\pi_{\rm Koop}(\Gamma)\}''}H=\infty$.
  Then $\overline H^{\rm p}= L^2(F, \nu |_{F})\otimes L^2(K,\mu)$. The representation $\overline \pi_{\rm Koop}^{\rm p}$ is determined by  the $\Gamma$-valued cocycle on $G\times F$, determined by the action of $G$ on $\cX$, in the identification $\cX\cong F\times \Gamma$. To obtain the representation $\overline \pi_{\rm Koop}^{\rm p}$, one views this cocycle as having values in $ K$.

Assume that $\pi_0$ is a representation of $G$ with the properties listed in Definition \ref{TC}. This is typically the case
 when
 ${\rm dim}_{\{\pi_0(\Gamma)\}''}H$ is not an integer.
 Such a situation occurs when the $\Gamma$-invariant  vectors are the automorphic forms.

 In this case $G=\PGL(2,\mathbb Z[\frac{1}{p}])$, $p$ a prime, $\Gamma$ is the modular group, and $H_n = H^2(\bH, \nu_n)$. Here  $\nu_n$ is the measure $\nu_n = ({\rm Im}z)^{n-2}d\bar{z}dz$ on the upper half plane $\bH$.
The representations $\pi_0=\pi_n$ are obtained by restricting to $G$ the representations in the discrete series $(\pi_n)_{n \in \N, n\geq 2}$ of unitary representations of $\PSL_{2}(\R)$ (see e.g. \cite {la}). In this case there is no canonical wandering subspace $L$, since $${\rm dim}_{\{\pi_n(\Gamma)\}''}H_n= {\rm dim}_{\Gamma}H_n = \frac{n-1}{12},$$ as proved in \cite{GHJ}, Section 3.3.d. The reason for the previous  non-existence statement is the fact  that, if such a space $L$ exists, then $${\rm dim}_\C L=  {\rm dim}_{\Gamma}H_n.$$
This is impossible if $\frac{n-1}{12}$ is not an integer.

In the theory of automorphic forms, the construction of Hilbert spaces of $\Gamma$-invariant vectors is  solved  by using a fundamental domain and the Petersson scalar product (\cite{Pe}), which consists in integration over the fundamental domain. In the framework  of this paper, we substitute the integration over a fundamental domain by the action of the projection $P_L$ onto the space $L$.

  One   assumes that there exists  a    unitary representation $\wideparen{\pi}_n$  on a larger Hilbert space, containing $\pi_n$ as a subrepresentation as in Definition \ref{TC}. In this  case (see Example \ref{auto}), the larger space Hilbert space is   $H=L^2(\bH, \nu_n)$. The unitary representation  $\wideparen{\pi}_n$ acts  on functions on $\mathbb H$ by the same formula as $\pi_n$. The invariance properties of the measure $\nu_n$ imply, as in the case of $\pi_n$, that $\wideparen{\pi}_n$ is a unitary representation of $G$.

  We use the notations from Definition \ref{TC}.
Let $P_0$ be the orthogonal  projection from $L^2(\bH, \nu_n)$  onto the space $H_n$ of square summable analytic functions. We have that $[P_0, \wideparen{\pi}_n(g)] = 0$ for all $g \in G$.  Hence $\pi_n(g) = P_0 \wideparen{\pi}_n(g)P_0$, $g \in G$.
We take the space $L=L^2(F, \nu_n)$ as a canonical choice for the $\Gamma$-wandering, generating subspace $L$ for $\wideparen{\pi}_n$. Recall that $P_L$ is the orthogonal projection onto $L$. It is obvious that in this case $P_L$ is precisely $M_{\chi_F}$, the operator of multiplication with the characteristic function  of the fundamental domain $F$, acting  on $L^2(\bH, \nu_n)$.

The computations in Section 3.3 of \cite{GHJ} imply  that the  product   $P_0M_{\chi_F}$ is a trace class operator, with trace equal to the
 Murray-von Neumann dimension ${\rm dim}_{\{\pi_n(\Gamma)\}''}H_n$.

To abstractly define the space of $\Gamma$-invariant vectors,
we make use of the relative position (the operator angle) of the projections $P_0$ and $M_{\chi_F}$.
In this situation the technical hypothesis (Definition \ref{TC}) is the convergence in the space of Hilbert Schmidt class operators of the series
$$
\mathop{\sum}\limits_{\theta \in \Gamma\sigma\Gamma}P_L \pi_n(\theta)P_L, \quad {\rm for} \; \sigma \in G.
$$

In the present example, this condition holds true since the reproducing kernel for the   projection onto the space of automorphic forms, and the reproducing kernels of the associated Hecke operators,  are the sum of the operator kernels (Berezin's reproducing  kernels, \cite{Be}, \cite{Ra3}) of the operators
\begin{equation}\label{suma}
\mathop{\sum}\limits_{\theta \in \Gamma\sigma\Gamma} M_{\chi_F} \pi_n(\theta) M_{\chi_F},\quad   \sigma \in G.
 \end{equation}
Moreover the sum of the traces of the corresponding operators is also absolutely convergent. This follows from the computations in  \cite{Za} and \cite{GHJ}, Section 3.3.

We prove in Theorem   \ref {a0formula} that  the sum  in  formula (\ref{suma}) is a projection, when taking the  sum  over $\Gamma$  (e. g. $\sigma$ is the identity element). We prove that the range of this projection (which is a subspace of $L$) is unitarily equivalent to the Hilbert space of $\Gamma$-invariant vectors. Moreover, the same unitary equivalence will transform the Hecke operator corresponding to a double coset $[\Gamma\sigma\Gamma]$  into  the sum in  formula (\ref{suma}).

The advantage of this point of view on spaces of automorphic forms is that  formula (\ref{suma}) allows a direct computation of the traces of the Hecke operators at any level $\Gamma_0\in \S$. This is used to compute (Theorem \ref{deligne})  values of the character
  $\theta_{\overline{\pi}_n^{\rm p}}(g), g \in G$. These values   are the   partial sums of  traces of  operators as  in   formula (\ref{suma}) (see Theorem \ref{a0formula} 
  and Remark \ref{pin}).

The  construction  of  the spaces of $\Gamma$-invariant vectors for the unitary representation $\overline {\pi_{\rm Koop}}^{\rm p}$ may be obtained alternatively if the above representation admits a "square root" as described below (see Example \ref {ad} in Section \ref{theexamples}).   Let  the infinite measure space  be $\bH$, endowed with the $\PSL(2,\R)$ invariant measure $ \nu_0=({\rm Im}z)^{-2}d\bar{z}dz$.
 Let $\pi = \pi_{\rm Koop}$ be the corresponding Koopmann unitary representation of $\PSL(2,\R)$ into the unitary group of $L^2(\bH, \nu_0)$.
We denote by  $\pi_n^{\rm op}$ the conjugate representation of $\pi_n$.
Because of Berezin's quantization method (see \cite{Re}, \cite{Be}), the representation $\pi_{\rm Koop}$ factorizes as $$
\pi_{\rm Koop}\cong\pi_n \otimes \pi_n^{\rm op}.
$$


As in the previous example, we let $G=\PGL(2,\mathbb Z[\frac{1}{p}])$, $p$ a prime, and let $\Gamma$ be the modular group.
The  factorization of the representation $\pi_{\rm Koop}$ gives a canonical choice for the Hilbert spaces of $\Gamma_0$-invariant vectors, $\Gamma_0 \in \S$.
Indeed, the representation $\pi_n \otimes \pi_n^{\rm op}$ is unitarily equivalent to the adjoint representation ${\rm Ad\, } \pi_n(g)$ into the unitary group of the Hilbert -Schmidt operators $\cC_2(H_n) \cong H_n \otimes \overline{H}_n^{\rm p}$.

The larger vector space containing $\cC_2(H_n)$ is $\V = B(H_n)$, the space of bounded linear operators on $H_n$.  Then the adjoint representation ${\rm Ad}\, \pi_n(g)$ extends to a representation into the inner automorphism group of $B(H_n)$.  The  space $\V^{\Gamma_0}$ of $\Gamma_0$ invariant vectors    is in this  situation the type II$_1$ factor:
$$
\A_n(\Gamma_0) = \{ \pi_n(\Gamma_0) \}'= \{ X \in B(H_n) | [X, \pi_n(\gamma)] = 0, \gamma \in \Gamma_0 \}.
$$
The fact that the commutant algebra $\A_n(\Gamma_0)$ is a type II$_1$ factor is a consequence of the fact that ${\rm dim}_{\Gamma}H_n$ is finite (see \cite{GHJ}, Section 3.3.d).

Then the Hilbert space $H^{\Gamma_0}$ is simply $L^2(\A_n(\Gamma_0), \tau)$,  the GNS Hilbert space associated to the unique trace $\tau$ on $\A_n(\Gamma_0)$. The family $\{ \A_n(\Gamma_0) \}_{\Gamma_0 \in \; \S}$ is a directed  family of II$_1$ factors. Let $\A_n^\infty$ be the type II$_1$ factor obtained as the inductive limit of the above  directed family of II$_1$ factors. We also denote by $\tau$ the unique trace on $\A_n^\infty$.

 Then the space $\overline{H}^{\rm p}$ is $L^2(\A_n^\infty, \tau)$ and $\overline{{\rm Ad \,} \pi_n}^{\rm p}$ is the extension of ${\rm Ad\, }\pi_n(g)$.   In Theorem \ref{double} (Theorem 3.2 in  \cite{Ra1}) we prove that the $K$-spherical matrix coefficients for $\overline{{\rm Ad\, } \pi_n}^{\rm p}$ are explicitly  computed from  a $C^{\ast}$-representation determined by the $K$-spherical matrix coefficients for $\overline{\pi}_n^{\rm p}$. This representation is in fact the main algebraic tool in \cite{Ra}.

Let again $\pi_0$ be a representation of $G$ as in Definition \ref{TC}.
In all constructions above, the main building block for the representations $\overline{\pi_0}^{\rm p}$ is a completely positive map
 $\Phi$ (see Theorem \ref{cp}) supported on $C^{\ast}(G)$ with values in $B(L)$, and extending to $C^{\ast}(\overline{G})$.
The map $\Phi$ encodes the sums from formula (\ref{suma}). We extend $\Phi$  to $C^{\ast}(\overline{G})$, by defining, for the characteristic function of a closed subset $C$ of $\overline{G}$,
$$
\Phi(\chi_C) = \mathop{\sum}\limits_{\theta \in C}P_L\pi(\theta)P_L.
$$

Then $\Phi $ is a completely positive map on $C^\ast(\overline{G})$ with values in $B(L)$, and $\Phi $ is $\ast$-preserving, multiplicative representation of the operator system (Definition \ref{canonicalos})
$$\O(K, G)=\big[\C(\chi_{\sigma K} | \sigma \in G)\big]\cdot \big[\C(\chi_{\sigma K} | \sigma \in G)\big]^{\ast}\subseteq C^\ast(\overline {G}).$$ The  $\ast$-preserving, multiplicative property     means that for any two $K$-cosets  $K\sigma_1$, $K\sigma_2$  in $\overline{G}$,  we have
$$
\Phi(\chi_{K\sigma_1})^{\ast}\Phi(\chi_{K\sigma_2}) = \Phi(\chi_{\sigma_1K})\Phi(\chi_{K\sigma_2}) = \Phi(\chi_{\sigma_1 K \sigma_2}).
$$

We prove in Lemma \ref{matrix1}  that the representation $\overline{\pi_0}^{\rm p}$ is entirely reconstructible from the completely positive map $\Phi$.

Then $\Phi$ is  an "operator valued eigenvector" for the Hecke algebra.
Indeed, by the multiplicativity property,   denoting the convolution operation on functions on $\overline G$ by
$\cdot\ $, we obtain that:

\begin{equation}\label{eigenvector111}
\Phi(\chi_{K\sigma_1K})\Phi(\chi_{K\sigma_2}) = \Phi(\chi_{K\sigma_1K}{\cdot}\chi_{K\sigma_2}),\quad  \sigma_1, \sigma_2 \in G.
\end{equation}

\section{Axioms for constructing the Hilbert spaces\\ of $\Gamma$-invariant vectors}\label{axioms}

Let  $\Gamma \subseteq G$ be an almost normal subgroup  as in the introduction, satisfying the assumption
$[\Gamma : \Gamma_\sigma]=[\Gamma : \Gamma_{\sigma^{-1}}]$ for all $\sigma$ in $G$.
Let $\pi$ be a (projective) unitary representation  of $G$ into the unitary group $\U(H)$ of a Hilbert space $H$, with the properties
(i) and (ii) from Definition \ref{TC}.
In particular, we have that ${\rm dim}_{\{\pi(\Gamma)\}''}H$ is an integer or $\infty$. As observed in the previous section , this implies the existence of a $\Gamma$-wandering, generating subspace $L$ for $\pi|_\Gamma$ (Definition \ref{wandering}).
We recall that this means    that $L$ is orthogonal to $\pi(\gamma)L$ for $\gamma \in \Gamma$, $\gamma \neq e$, and  that $H = \overline{\mathop{\bigvee}\limits_{\gamma \in \Gamma}\pi(\gamma)L}$.

We  construct, the Hilbert spaces $H^{\Gamma_0}$ of $\Gamma_0$-invariant vectors, $\Gamma_0 \in \S$.
Such spaces $H^{\Gamma_0}$ will be isometrically isomorphic to $l^2(\Gamma_0 \setminus \Gamma) \otimes L$ for $\Gamma_0 \in \S$. The main problem that we first consider in this section  is to construct the representation of $G$ on the reunion of spaces of $\Gamma_0$-invariant vectors.

 The particular examples presented in the previous section suggest that one possibility to address this problem is to find an embedding of the Hilbert space $H$ into a larger vector space $\V$, such that the representation $\pi$ extends to a representation $\pi_{\V}$ of $G$ into the group of linear isomorphisms of $\V$, and such that $\pi_{\V}$  invariates the subspace $H$ of $\V$. We explain the construction first in this case, then perform the construction based on properties (i) and (ii) as in Definition \ref{TC}.

We will  work with subgroups $\Gamma_0$ that are conjugated in $G$ to subgroups in $\S$. We denote this enlarged class of subgroups of $G$ by $\tilde\S$.

 For $\Gamma_0\in \tilde\S$, we consider the spaces $\V^{\Gamma_0}$ of $\pi_{\V} (\Gamma_0)$-invariant vectors in $\V$. Then $\pi_{\V}$  has an obvious extension to
 \begin{equation}\label{vinfinity}
\V_{\infty}=\mathop{\bigvee}\limits_{\Gamma_0 \in \tilde\S}\V^{\Gamma_0}=\mathop{\bigvee}\limits_{\Gamma_0 \in \S}\V^{\Gamma_0}.
 \end{equation}
The remaining problem is to identify a $\pi_{\V}(G)$-invariant subspace of $\V_{\infty}$, that is endowed with  a $\pi_{\V}(G)$-invariant prehilbertian scalar product. Using this scalar product      we define the  Hilbert spaces $H^{\Gamma_0}\subseteq \V^{\Gamma_0} $, $\Gamma_0 \in \tilde \S$. Then $\pi_{\V}$ induces  a unitary representation of $G$ on $\mathop{\bigvee}\limits_{\Gamma_0 \in \tilde\S} H^{\Gamma_0}$.

The construction in this section is certainly similar to other constructions
in the literature (see e.g \cite{Bo}, \cite{Hal}). However, in Theorem \ref{hecke} we employ this construction to introduce  specific $\ast$-representations of the Hecke algebra, involving expressions as in formula (\ref{suma}). These are generalized in the next section to the case when ${\rm dim}_{\pi(\Gamma)}H$ is not an integer and hence there exist no $\Gamma$-wandering, generating subspace.

In the following definition, we introduce a general formalism that is used to construct the representation $\overline{\pi}^{\rm p}$, starting with an extension of the given representation $\pi$ to a larger vector space that contains  vectors invariant to the action of the subgroups in $\S$.
In practice, as will be done later in this section, we construct directly the Hilbert spaces corresponding to "virtual" $\Gamma_0$-invariant vectors, $\Gamma_0\in \S$.

\vskip10pt

\begin{defn}\label{formalism}[Formalism of $\Gamma$-invariant vectors.] Let  $\Gamma \subseteq G$ be as in the introduction, and consider  a (eventually projective) representation  $\pi$ of $G$   into the unitary group of a Hilbert space $H$. We make the following assumptions:

\item (i) There exists a larger vector space $\V$, containing $H$, and a representation $\pi_{\V}$ of $G$ into the linear isomorphisms of $\V$, such that $\pi_{\V}(g)$ invariates $H$ and $\pi_{\V}(g)|_{H} = \pi(g)$ for all $g\in G$. For $\Gamma_0$ in $\tilde\S$, denote by $\V^{\Gamma_0}$ the subspace of $\V$ consisting of vectors fixed by the action of $\Gamma_0$.
Consider the vector space
$
\V_{\infty}$ introduced in formula (\ref{vinfinity}).
%

\item(ii) There exists a dense $\pi(G)$-invariant subspace $\D_{\V} \subseteq H$,
and there exists  a complex valued bilinear form $\langle \cdot, \cdot \rangle_{\infty}$ on
$$
\V_{\infty} \times (\V_{\infty} \vee \D_{\V})
$$

\noindent with the following properties:

{\rm 1)} The restriction of $\langle \cdot, \cdot \rangle_{\infty}$ to $\V_{\infty} \times \V_{\infty}$ is a positive definite, prehilbertian scalar product.

{\rm 2)} For every $\Gamma_0 \in \tilde\S$, $v \in \V^{\Gamma_0}$, the linear map on $\D_{\V}$, defined by the restriction of the linear form $\langle v, \cdot \rangle_{\infty}$ to $\D_{\V}$, is $\Gamma_0$-invariant.

{\rm 3)} $\langle \cdot, \cdot \rangle_{\infty}$ is $\pi_{\mathcal V}(G)$-invariant:
$$
\langle \pi_{\V}(g)v_1, \pi_{\V}(g)v_2 \rangle_{\infty} = \langle v_1, v_2 \rangle_{\infty},\quad
g\in G,v_1 \in \V_{\infty}, v_2 \in \V_{\infty} \vee \D_{\V}.
$$

If the above assumptions hold true,
we let  $H^{\Gamma_0}$ be the Hilbert space completion of $\V^{\Gamma_0}$ with respect to the scalar product $\langle \cdot, \cdot \rangle_{\infty}$, and let $\overline{H}^{\rm p}$ be the Hilbert space completion of $\V_{\infty}$ with respect to the above scalar product.

\end{defn}

\vskip10pt

The following lemma is an obvious consequence of the assumptions in the definition.

\begin{lemma}\label{propertiespi}
 We assume the context of Definition \ref{formalism}.
Then the restriction of $\pi_{\V}$ to $\V_{\infty}$ extends to a unitary represention $\overline\pi^{\rm p}$ of $G$ into the unitary group of $\overline{H}^{\rm p}$.  Moreover, $\overline\pi^{\rm p}$ maps isometrically
$H^{\Gamma_0}$ onto  $H^{\sigma \Gamma_0\sigma^{-1}}$ for $\Gamma_0 \in \S, \sigma \in G$.

If $\Gamma_1 \subseteq \Gamma_0$, $\Gamma_1, \Gamma_0 \in \tilde \S$, then the inclusion $H^{\Gamma_0}$ into $H^{\Gamma_1}$ is isometrical by construction.  The orthogonal projection from $H^{\Gamma_1}$ onto $H^{\Gamma_0}$ is obtained by averaging over the cosets of $\Gamma_0$ in $\Gamma_1$.

  \end{lemma}

If the original representation $\pi$ is projective (see e.g \cite{BN} and the references therein) with cocycle $\varepsilon \in H^2(G,\mathbb{T})$, then assuming that the extension $\pi_{\V}$ has the same cocycle, the above construction still works.

In the sequel we will work with simultaneous representations of the group $G$ and with its Schlichting completion $\overline{G}$ (\cite{Sch}).
Consequently we introduce an universal $C^{\ast}$-algebra containing both $C^{\ast}(G)$ and $C^{\ast}(\overline G)$ as    $C^{\ast}$-subalgebras.

\vskip10pt

\begin{defn}\label{two} With $G, \overline{G}$ as above, let $\A(G, \overline{G})$ be the quotient of the universal crossed product $C^{\ast}$-algebra $C^{\ast}(G \rtimes C^{\ast}(\overline{G}))$, where $G$ acts by conjugation on $C^{\ast}(\overline{G})$ by the norm closed ideal generated by the relations of the form
$$
g\chi_{K_0} = \chi_{gK_0}, \quad g \in G, \quad K_0 = \overline{\Gamma}_0, \quad \Gamma_0 \in \S.
$$

\noindent Here, by $\chi_{gK_0}$ we denote the characteristic function of the coset $$gK_0 = \overline{g\Gamma_0},$$ where the closure operation  is  in $\overline{G}$.

Then $\A(G, \overline{G})$ is the norm closure of the span:
$$
{\rm Sp}\{g\chi_{K_0} | g \in G, \; K_0 = \overline{\Gamma}_0, \; \Gamma_0 \in \S\}.
$$

Assume we are given   a cocycle $\varepsilon \in H^2(G, \mathbb{T})$, that  also extends to $H^2(\overline{G}, \mathbb{T})$. Then, working with crossed products with cocycle, we obtain a similar $C^{\ast}$-algebra, that we denote with $\A_{\varepsilon}(G, \overline{G})$.

\end{defn}

\vskip10pt

Using the  previous  two definitions, we prove that the representation $\overline{\pi}^{\rm p}$ simultaneously extends to $G$ and
$\overline G$.
\vskip10pt

\begin{prop}\label{cala} Given a representation $\pi$ as in Definition \ref{formalism}, the corresponding representation $\overline{\pi}^{\rm p}$ from the above definition extends to a representation, also denoted by $\overline {\pi}^{\rm p}$, of the $C^{\ast}$- algebra $\A_\varepsilon(G, \overline{G})$ into $B(\overline{H}^{\rm p})$.

\end{prop}

\vskip10pt

\begin{proof} Let  $\Gamma_0$ be a subgroup of $G$ belonging to the class $\tilde \S$. Let $K_0=\overline {\Gamma_0}$, where the closure is taken in the topology of $\overline G$. Let $P_{H^{\Gamma_0}}$ be the orthogonal projection from $\overline H^{\rm p}$ onto the Hilbert space  $H^{\Gamma_0}$.
 The extended representation $\overline{\pi}^{\rm p}$ is constructed
  by mapping $\frac {1}{\mu(K_0)}\chi_{K_0}$ .  The normalization is necessary, since in
  $C^{\ast}(\overline G)$ the convolutor with a subgroup  $K_0$ of $K$ is a non-trivial scalar multiple of a projection, as $(\chi_{K_0})^2=\mu(K_0) \chi_{K_0}$.
The elements in $G$ are represented as unitary operators on $\overline H^{\rm p}$, through the representation $\overline{\pi}^{\rm p}$ introduced in Definition \ref{formalism}.  With this choice, all the relations defining the universal $C^{\ast}$-algebra $\A_\varepsilon(G, \overline{G})$ are  obviously verified.
\end{proof}

\vskip10pt

Given a representation $\pi$ of $G$, such that $\pi |_{\Gamma}$ admits a $\Gamma$-generating,  wandering subspace, we construct a representation as in the Definition $\ref{formalism}$. We will construct directly the Hilbert spaces of $H^{\Gamma_0}$-invariant vectors, without   constructing the space $\V$ from Definition \ref{formalism}.

\vskip10pt

\begin{lemma}\label{Lspacespart1} Let $\pi$ be a unitary representation of $G$
 with the properties introduced (i) and (ii) from  Definition \ref {TC}. In particular $\pi|_\Gamma$ is an integer multiple of the left regular representation $\lambda_\Gamma$. We use the notations from the above mentioned definition.

 Then
 \item(i)  For all $\Gamma_0 \in \S$, $g \in G$, the sum over the coset $ \Gamma_0 g$:
\begin{equation}\label{sums}
\mathop{\sum}\limits_{\theta \in \Gamma_0 g}P_L\pi(\theta)P_L
\end{equation}
\noindent is so-convergent in $B(L)$.

\item(ii)  The subspace
\begin{equation}\label{d}
 \D_{L, \pi} = \{ h \in H | \mathop{\sum}\limits_{\gamma \in \Gamma_0}P_L\pi(\gamma)h \mbox{\ is so-convergent for all} \ \Gamma_0 \in \S \},
\end{equation}
 is a dense $\pi(G)$-invariant subspace of H, containing $L$.


\end{lemma}
\begin{proof}To prove part (i), we note that it is sufficient to prove the statement for $\Gamma_0=\Gamma$, as $\pi|_{\Gamma_0}$ remains an integer multiple of the left regular representation $\lambda_{\Gamma_0}$ of the group $\Gamma_0$. If we sum over the double coset $\Gamma \sigma\Gamma$, we get exactly the Hecke operator associated with the double coset, acting on the space $L$, which is unitarily equivalent to the space of $\Gamma$-invariant vectors.

 In the case where the representation $\pi$ is as in Example \ref{koop},  the sum in formula (\ref{sums}) coincides with the representation in the Koopman unitary representation of the piecewise  bijective transformation $\dot{\wideparen{\Gamma g}}$ introduced in (\cite {Ra5}, Lemma 4(i)).  The sum is
 so-convergent
since we are adding partial isometries corresponding to transformations with disjoint domains. The operator associated to the piecewise bijective  transformation $\dot{\wideparen{\Gamma g}}$ is consequently (\cite{Ra5}) a finite sum (of cardinality
$[\Gamma:\Gamma_g]$) of partial isometries with orthogonal initial spaces. The argument transfers ad litteram to our case because of assumption (ii) in Definition \ref{TC}.
The key feature that is making this argument work is the fact that $\pi(\sigma)L$ is a $\sigma\Gamma\sigma^{-1}$ wandering subspace.

For part (ii), we note that to prove the  density assumption  it suffices to assume that $\Gamma$-equivariantly $H=\ell^2(\Gamma)$. In this case the domain $\D_\V$ is simply $\ell^1(\Gamma)\cap \ell^2(\Gamma)$. The $\pi(G)$-invariance of $H$ is now a consequence of part (i).

\end{proof}

\begin{defn}\label{Lspacespart2}
We use the notations and definitions introduced above.
Let  $\Gamma_0$ be a subgroup  $\S$. Assume that $\Gamma$ is decomposed in cosets over $\Gamma_0$ as
 $\Gamma=\cup\Gamma_0 r_j$, where $r_j$ are  coset representatives for $\Gamma_0$.  Let $L^{\Gamma_0}$  be the subspace of $H$ obtained as a  sum of the following  orthogonal subspaces of $H$:\begin{equation}\label{lgamma0}
L^{\Gamma_0}=\sum\pi(r_j)L.
\end{equation}
\noindent Denote the orthogonal projection from $H$ onto $L^{\Gamma_0}$ by
$P_{L^{\Gamma_0}}$.

We define the Hilbert space  $H^{\Gamma_0}=
\V^{\Gamma_0}$
as the space of formal sums
\begin{equation}\label{sumformula}
\{ \mathop{\sum}\limits_{\gamma_0 \in \Gamma_0}\pi(\gamma_0)h | h \in
\D_{L, \pi} \}.
\end{equation}

\noindent subject to the identification:
\begin{equation}\label{ident}
\mathop{\sum}\limits_{\gamma_0 \in \Gamma_0}\pi(\gamma_0)h = \mathop{\sum}\limits_{\gamma_0 \in \Gamma_0} \pi(\gamma_0)l_0,
\end{equation}
\noindent if $ h\in
\D_{L, \pi}$ and   $l_0$ is the vector in $L^{\Gamma_0}$ given by the formula
\begin{equation}\label{justification}
l_0 = \mathop{\sum}\limits_{\gamma_0 \in \Gamma_0}P_{L^{\Gamma_0}}(\pi(\gamma_0)h).
\end{equation}
The infinite sum in formula (\ref{justification}) is convergent since $h$ belongs to  $
\D_{L, \pi} $.

\end{defn}

\

\

The condition in formula (\ref{ident}) corresponds to the fact that the sum over $\Gamma_0$ is  invariant under changing the summation variable from $\gamma$ into $\gamma\gamma_0$, for a fixed $\gamma_0$ in $\Gamma_0$. This condition   should necessarily hold  true if the vector $h$ is a sum of translates, by elements in $\Gamma_0$, of vectors in $L^{\Gamma_0}$.
Using the above construction, we can introduce the unitary representation of $\overline G$ acting on
vectors invariant to subgroups in $\S$.

\begin{prop}\label{Lspaces}

The positive definite prehilbertian  scalar product $\langle \cdot, \cdot \rangle_\infty$ on $\V_{\infty}=
\mathop{\bigvee}\limits_{\Gamma_0 \in \S} H^{\Gamma_0}$
is defined for $\Gamma_0\in \S$, $l_1 \in
\D_{\Gamma, \pi}$, $l_2 \in L^{\Gamma_0}$, by the formula
\begin{equation}\label {scalar}
\langle \mathop{\sum}\limits_{\gamma_0\in \Gamma_0}\pi(\gamma_0)l_1, \mathop{\sum}\limits_{\gamma_0' \in \Gamma_0}\pi(\gamma_0')l_2 \rangle_{\infty} = \frac{1}{[\Gamma : \Gamma_0]}\langle \mathop{\sum}\limits_{\gamma_0 \in \Gamma_0} \pi(\gamma_0)l_1, l_2 \rangle.
\end{equation}
O the right hand side of the above equality we use the scalar product on $H$.
Clearly $\V^{\Gamma_0}$ embeds isometrically into $\V^{\Gamma_1}$ for $\Gamma_1 \subseteq \Gamma_0$.
Let  $\overline H^{\rm p}$ be the Hilbert space completion of  $\V_{\infty}$ with respect to the scalar product $\langle \cdot, \cdot \rangle_\infty$.

Then, the unitary
 representation $\pi$ determines a  unitary representation  $\overline\pi^{\rm p}$ into the unitary group of $\overline H^{\rm p}$, having the properties from Lemma \ref{propertiespi}   .

 \end{prop}

 Before proving the proposition, we note that formula (\ref{scalar}), which is equivalent to formula (\ref{scalarpet}) below,   is  a generalization of the Petersson scalar product formula \cite{Pe}.

  Indeed, with the notations from the introduction, consider the case of two automorphic forms of weight $n\in \mathbb N$, which are hence $\Gamma$-invariant vectors, as above, for the representation $\pi_n$. To obtain the scalar product of the two automorphic forms  one multiplies one of them with the characteristic function $\chi_F$ of a fundamental domain, and then uses  the usual scalar product  from $L^2(\mathbb H,\nu_n)$, which extends the scalar product on $H_n$. This is exactly what is performed in the next formula, by replacing the operator $M_{\chi_F}$ (from formula (\ref{suma})) with the projection $P_L$, which has similar properties to  $M_{\chi_F}$.

  One can  establish an equivalent expression  of formula (\ref{scalar}), analogous to the Petersson scalar product formula.
For $h_1,h_2\in   \D_{L, \pi}$, we let $$l_i= \mathop{\sum}\limits_{\gamma \in \Gamma}P_L \pi(\gamma)h_i, \quad  i=1,2.$$
 Using  the identification in formula  (\ref{ident}), we obtain that formula
 (\ref{scalar}) is equivalent to:
 \begin{equation}\label {scalarpet}
\langle \mathop{\sum}\limits_{\gamma\in \Gamma}\pi(\gamma)h_1, \mathop{\sum}\limits_{\gamma' \in \Gamma}\pi(\gamma')h_2 \rangle_{\infty} = \langle P_L\big( \mathop{\sum}\limits_{\gamma \in \Gamma} \pi(\gamma)\big)h_1, \mathop{\sum}\limits_{\gamma' \in \Gamma}\pi(\gamma') h_2 \rangle.
\end{equation}
This is further equal to
$$
\langle P_L\big( \mathop{\sum}\limits_{\gamma \in \Gamma} \pi(\gamma)\big)h_1,P_L\big( \mathop{\sum}\limits_{\gamma' \in \Gamma}\pi(\gamma') h_2\big) \rangle=\langle l_1,l_2\rangle.
$$
For more general subgroups $\Gamma_0\in \S$, formula (\ref{scalarpet}) of the scalar product is similar, with the difference that instead of $P_L$, one uses the projection $P_{L^{\Gamma_0}}$ onto a  $\Gamma_0$-wandering, generating subspace $\pi$ .

\vskip10pt

\begin{proof}[Proof of Proposition \ref{Lspaces}]  Let $\Gamma_0, \Gamma_1$ be two subgroups $\S$ such that $\Gamma_1 \subseteq \Gamma_0$.  We split first $\Gamma_0$ into cosets over $\Gamma_1$.
Using the coset representation, we split  the sum in formula (\ref{sumformula}) for the vectors in $\V^{\Gamma_0}$ into $[\Gamma_0 : \Gamma_1]$ vectors, which all belong to $\V^{\Gamma_1}$.   Then $\V^{\Gamma_0}$ is embedded into $\V^{\Gamma_1}$.
%
Indeed, if $\Gamma_0=\bigcup r_j \Gamma_1$, the embedding is realized as follows: if
\begin{equation}\label{formeta}
\eta=\mathop{\sum}\limits_{\gamma_0 \in \Gamma_0} \pi(\gamma_0)l_0,\quad  l_0\in L^{\Gamma_0}
\end{equation}
 is a generic vector in $\V^{\Gamma_0}$, then we identify the vector  $\eta$
with the following element of the vector space $\V^{\Gamma_1}$:
$$\eta_1=\mathop{\sum}\limits_{j}\mathop{\sum}\limits_{\gamma_1 \in \Gamma_1}\pi(\gamma_1)  l_0= \mathop{\sum}\limits_{\gamma_1 \in \Gamma_1}\pi(\gamma_1)\big[\mathop{\sum}\limits_{j}\pi(r_j) l_0\big]\in  \V^{\Gamma_1} .$$

The embedding $\V^{\Gamma_0}\subseteq \V^{\Gamma_1}$is isometric. Indeed, for  $l_0$ in $L^{\Gamma_0}$ as above,  the square of the  norm of the vector $\eta\in H^{\Gamma_0} $ introduced in formula (\ref{formeta}), is determined according to formula (\ref{scalar}). This  is equal to $$ \frac{1}{[\Gamma : \Gamma_0]}\langle l_0,l_0\rangle,$$
where the scalar product is computed in $H$.

On the other hand, according to the same formula, the  norm of the vector $\eta_1$ in  $\V^{\Gamma_1}$ is
$$ \frac{1}{[\Gamma : \Gamma_1]}\langle \big[\mathop{\sum}\limits_{j}\pi(r_j) l_0\big],\big[\mathop{\sum}\limits_{k}\pi(r_k) l_0\big] \rangle.$$
The set $\{r_j\}$ has cardinality $[\Gamma_0 :\Gamma_1]$. Moreover the vectors $\pi(r_j) l_0$ are pairwise orthogonal.
Hence the square of the norm of the vector $\eta_1$ is
$$\frac{1}{[\Gamma : \Gamma_1]}[\Gamma_0 :\Gamma_1]\langle l_0,l_0\rangle=[\Gamma : \Gamma_1]\langle l_0,l_0\rangle.$$
\noindent Hence the embedding  $\V^{\Gamma_0}$  into $\V^{\Gamma_1}$ is isometric.

The representation $\pi_{\V}$ is defined as follows. Let $g \in G$, $\Gamma_0 \in \S$, $l \in L$, and consider the vector
$$
\eta = \mathop{\sum}\limits_{\gamma_0 \in \Gamma_0}\pi(\gamma_0)l \in \V^{\Gamma_0}.
$$

Then we split the coset $g\Gamma_0$ as a disjoint union $$g\Gamma_0=\mathop{\bigcup}\limits_{j}\Gamma_0^jy_j$$ of  cosets of smaller subgroups $\Gamma_0^j$  in $\S$,  such that $$g\Gamma_0^j g^{-1}= \Gamma_1^j\subseteq \Gamma,\quad \Gamma_1^j\in \S. $$
This  is always possible,  by considering cosets of $\Gamma_0$ over subgroups of $\Gamma_0\cap \Gamma_{g^{-1}}$. Then we define
\begin{equation}\label{splitting}
\pi_{\V}(g)\eta = \mathop{\sum}\limits_{j}\mathop{\sum}\limits_{\gamma_j \in \Gamma_1^j}\pi(\gamma_j)(\pi(gy_j)l).
\end{equation}

By the   assumptions on the domain in formula (\ref{d}), it follows that $\pi_\V(g)\eta$ belongs to $\V_{\infty}$. This is because $\pi_\V(g)$ maps $\V^{\Gamma_0^j\cap \Gamma_{g^{-1}}}$ onto $\V^{\Gamma_1^j\cap \Gamma_{g}}$.   Since for all $g\in G$ the indices of the subgroups $\Gamma_{g^{-1}}$ and $\Gamma_g$ are equal, the definition of the scalar product proves that  $\pi_{\V}$ maps isometrically $\V^{\Gamma_0}$ into $\V_{\infty}$.
Obviously, if $g\in G$ and $\Gamma_0\in S$, then $\V^{g\Gamma_0g^{-1}}$ is contained in $\V^{g\Gamma_0g^{-1}\cap \Gamma_0}$. But
$g\Gamma_0g^{-1}\cap \Gamma_0$ is a subgroup in $\S$, and hence we have the alternative formula $\V_{\infty}=
\mathop{\bigvee}\limits_{\Gamma_0 \in \tilde \S} H^{\Gamma_0}$.

Consequently, we obtain a unitary representation $\overline{\pi}^{\rm p}$ into the unitary group of the Hilbert space $\overline{H}^{\rm p}$, as in Definition $\ref{formalism}$.

\end{proof}

\vskip10pt

  Recall that for $\Gamma_0$ in $\S$ and $\sigma\in G$, we denote
  $$(\Gamma_0)_\sigma=
 \sigma \Gamma_0\sigma^{-1}\cap \Gamma_0.$$ The index
 $[\Gamma_0: (\Gamma_0)_\sigma]$ will intervene in the following computations.

   For $\Gamma_0$ as above, let $K_0= \overline{\Gamma_0}$ be the closure of $\Gamma_0$ in the profinite completion  $K$ of $\Gamma$.  In the next statement, we find an explicit matrix representation of the image through the representation $\overline {\pi}^{\rm p}$ of the convolution operator with the characteristic function of the double coset $K_0\sigma K_0$. This is obviously  the Hecke operator associated to the double coset $\Gamma_0\sigma\Gamma_0$, on $\Gamma_0$-invariant vectors, normalized by a constant.
   We obtain first the precise normalization constants required for the Hecke operators.
   The normalization factor that we obtain is the index of the subgroup $(\Gamma_0)_\sigma$ in $\Gamma_0$. It is necessary, because in the $C^\ast$-algebra $C^\ast(\overline G,G)$, if $K_0$ is a subgroup of $K$, then the convolution operator with $\chi_{K_0\sigma K_0}$ is a  scalar multiple, by the factor   
   $ [K_0:(K_0)_\sigma]$, of the ordered  product of the convolution operators by $\chi_{K_0}$, $\sigma$ and  $\chi_{K_0}$.

  \begin{lemma}\label{heckeconverted}
  We refer to the notations introduced above. For every $\sigma \in G$ and every subgroup $K_0$ as above, we have:
    \begin{equation}\label{completehecke}
    \overline {\pi}^{\rm p}
    \big(
     \chi_{K_0\sigma K_0}
     \big)=
     \overline {\pi}^{\rm p}
     \big(\chi_{K_0\sigma K_0}
     \big)=
{ [\Gamma:(\Gamma_0)_\sigma]}
    \big(
    \overline{\pi}^{\rm p}(\chi_{K_0})
    \overline{\pi}^{\rm p}(\sigma)
    \overline{\pi}^{\rm p}(\chi_{K_0})
    \big).
    \end{equation}
\end{lemma}
\begin{proof}
We work in the universal algebra $\A=C^\ast(G, \overline G)$. For $\sigma\in \overline G$ denote the convolution by  $\sigma$ by $L_\sigma$.   
Denote the convolution  operator   by a continuous function $f$ on $\overline G$ by $L(f)$.
Clearly, for every measurable subset $A$ of $\overline G$ we have
$$L(\chi_A) L_\sigma=L(\chi_{A\sigma}),\quad  L_\sigma L(\chi_A)= L(\chi_{\sigma A}),\quad \sigma\in G.$$
 Then, obviously,
\begin{equation}\label{convolution1}
L(f)=\int_{\overline G} f(g)L_g{\rm d}g.
\end{equation}\label{projections}
\begin{equation}
(L(\chi_{K_0}))^2= \nu(K_0) L(\chi_{K_0}).
\end{equation}
\noindent In particular, $\nu(K_0)^{-1} L(\chi_{K_0})$ is a projection.

The same argument gives that if $K_0,K_1$ are   subgroups, as above, and $K_1$ is a subgroup of $K_0$, then
\begin{equation}\label{lemmasubgroup}
L(\chi_{K_1})L(\chi_{K_0})= L(\chi_{K_0})L(\chi_{K_1})= \nu(K_1) L(\chi_{K_0}).
\end{equation}
We decompose $K_0$ as a reunion right  cosets for
the subgroup $(K_0)_\sigma= K_0\cap \sigma K_0\sigma^{-1},$ with coset representatives $s_i\in K_0$, $i=1,2,..., [K_0:(K_0)_\sigma]$. These  are the same as the coset representatives for $(\Gamma_0)_\sigma$ in $\Gamma_0$. Using formula (\ref{lemmasubgroup}) we obtain
$$
\begin{aligned}
L(\chi_{K_0}) & L_\sigma L(\chi_{K_0}) =\sum_i L(s_i\chi_{(K_0)_\sigma})L_\sigma L(\chi_{K_0}) \\
& =\sum_i L(s_i\sigma) L(\chi_{(K_0)_{\sigma^{-1}}}) L(\chi_{K_0})  =
\nu((K_0)_{\sigma^{-1}})\sum_i L(s_i\sigma)  L(\chi_{K_0}) \\ & 
 =\nu((K_0)_{\sigma^{-1}})L(\chi_{K_0\sigma K_0}).
\end{aligned}$$

\end{proof}

\vskip10pt

As we previously noted, the Hilbert spaces $H^{\Gamma_0}$, $\Gamma_0\in \S$ are isometrically isomorphic to $l^2(\Gamma_0 \setminus \Gamma)\otimes L$, where the scalar product on $l^2(\Gamma_0 \setminus \Gamma)$ is chosen so  that the embeddings $l^2(\Gamma_0 \setminus \Gamma) \subseteq l^2(\Gamma_1 \setminus \Gamma_0)$ are isometric for all $\Gamma_0 \subseteq \Gamma_1$.
It turns out that the entries  of the matrices representing  Hecke operators are sums as in formula (\ref{sums}). The previous lemma indicates clearly that, after normalization, the formula of the classical Hecke operator corresponding to  the sum over cosets is ${[\Gamma:\Gamma_\sigma]}P_{H^{\Gamma_0}}\overline{\pi}^{\rm p}(\sigma)P_{H^{\Gamma_0}}$, $\sigma \in G$, $\Gamma_0\in S$

\begin{thm}\label{hecke} We use the notations and definitions  previously introduced in this section.
Fix  a subgroup $\Gamma_0$ in $\S$. We choose a family $(s_i)$ of  right coset representatives for $\Gamma_0 \subseteq \Gamma$.
Consider, as  in  formula (\ref{lgamma0}), the following Hilbert space:
\begin{equation}\label{lgamma}
L^{\Gamma_0} = \mathop{\oplus}\limits_{i=1}^{[\Gamma : \Gamma_0]}\pi(s_i)L.
\end{equation}

\noindent
The  Hilbert space norm on
$L^{\Gamma_0}$ is normalized so that the embedding of $L$ into $L^{\Gamma_0}$, defined by the correspondence
\begin{equation}\label{renorm}
l\in L\rightarrow  \mathop{\oplus}\limits_{i=1}^{[\Gamma : \Gamma_0]}\pi(s_i)l \in L^{\Gamma_0} = \mathop{\oplus}\limits_{i=1}^{[\Gamma : \Gamma_0]}\pi(s_i)L,
\end{equation}
is isometric.
The space $L^{\Gamma_0}$ is obviously identified with a subspace of $H$. In this case the space $L^{\Gamma_0}$  is endowed  with a non-normalized scalar product inherited from $H$. Let $P_{L^{\Gamma_0}}$ be the orthogonal projection from $H$ onto $L^{\Gamma_0}$.

Then,  for $\sigma \in G$, the Hecke operator ${[\Gamma:\Gamma_\sigma]}P_{H^{\Gamma_0}}\overline{\pi}^{\rm p}(\sigma)P_{H^{\Gamma_0}}$ is unitarily equivalent to the bounded operator
\begin{equation}\label{agamma}
A(\Gamma_0\sigma\Gamma_0) = \mathop{\sum}\limits_{\theta \in \Gamma_0\sigma\Gamma_0}P_{L^{\Gamma_0}}\pi(\theta)P_{L^{\Gamma_0}}.
\end{equation}

\end{thm}

\vskip10pt

\begin{proof} Given $\Gamma_0 \in \S$, and a choice for the coset representatives $$\Gamma = \bigcup \Gamma_0s_i,$$ we construct a unitary operator $W^{\Gamma_0}$ from $L^{\Gamma_0} = \oplus \pi(s_i)L$ into $H^{\Gamma_0}$ as follows. We define, for vectors $l_i \in L$, the following isometry:
$$
W^{\Gamma_0}(\oplus \pi(s_i)l_i) = \mathop{\sum}\limits_{i}\mathop{\sum}\limits_{\gamma \in \Gamma_0}\pi(\gamma_0)\pi(s_i)l_i.
$$

We prove that, for $\sigma \in G$, the following diagram,
$$
\begin{array}{ccccc}
& H^{\Gamma_0} & \mathop{\longleftarrow}\limits^{W^{\Gamma_0}} & \oplus \pi(s_i)L & \\[3mm]
{ [\Gamma_0:(\Gamma_0)_\sigma]}P_{H^{\Gamma_0}}\overline{\pi}^{\rm p}(\sigma)P_{H^{\Gamma_0}} & \downarrow & & \downarrow & \mathop{\sum}\limits_{\theta \in \Gamma_0\sigma\Gamma_0}P_{L^{\Gamma_0}}\pi(\theta)P_{L^{\Gamma_0}}\\[3mm]
& H^{\Gamma_0} & \mathop{\longleftarrow}\limits^{W^{\Gamma_0}} & \oplus \pi(s_i)L &
\end{array}
$$
is  commutative. To do this we  use  the formula of the unitary operators $\overline{\pi}^{\rm p}(\theta)$, $\theta\in G$, defined in the proof of Proposition \ref{Lspaces}.

It is sufficient to verify the above commutativity of the diagram   in the case $\Gamma = \Gamma_0$; the cases corresponding to other  subgroups $\Gamma_0\in \S$ are a consequence.
 Consider a vector $l \in L$. We have that
$$
W^{\Gamma}l = \mathop{\sum}\limits_{\gamma \in \Gamma}\pi(\gamma)l \in \V^{\Gamma}.
$$

If the decomposition of $\Gamma$ into right cosets over $\Gamma_{\sigma^{-1}}$ is $\Gamma = \bigcup \Gamma_{\sigma^{-1}}r_j$, then $W^{\Gamma}l$ is further equal to
$$
\mathop{\sum}\limits_{j}\mathop{\sum}\limits_{\gamma \in \Gamma_{\sigma^{-1}}}\pi(\gamma)\pi(r_j)l.
$$

Then applying $\overline{\pi}^{\rm p}(\sigma)$, we obtain
$$
\mathop{\sum}\limits_{j}\mathop{\sum}\limits_{\gamma_1 \in \Gamma_{\sigma}}\pi(\gamma_1)\pi(\sigma r_j)l.
$$

Projecting on $H^{\Gamma}$, this gives
$$\frac{1}{[\Gamma:\Gamma_\sigma]}
\mathop{\sum}\limits_{j}\mathop{\sum}\limits_{\gamma \in \Gamma}\pi(\gamma)\pi(\sigma r_j)l,
$$
\noindent and this is equal to
$$
\frac{1}{[\Gamma:\Gamma_\sigma]}\mathop{\sum}\limits_{\theta \in \Gamma\sigma\Gamma}\pi(\theta)l.
$$

On the other hand the sum
$$
\mathop{\sum}\limits_{\theta \in \Gamma\sigma\Gamma}P_{L}\pi(\theta)P_L
$$
\noindent applied to the vector $l$, gives
$$
\mathop{\sum}\limits_{\theta \in \Gamma\sigma\Gamma}P_L\pi(\theta)l = \mathop{\sum}\limits_{j}\mathop{\sum}\limits_{\gamma \in \Gamma}P_L\pi(\gamma)\pi(\sigma r_j)l.
$$

We apply the isometry $W^\Gamma$ to this vector.  We use the identifications  assumed in the structure of the space $H^{\Gamma}$ (see formula (\ref{ident}) in the statement of  Definition \ref{Lspacespart2}). Then the  above sum corresponds to the vector
$$
\mathop{\sum}\limits_{r_j}\mathop{\sum}\limits_{\gamma \in \Gamma}\pi(\gamma\sigma r_j)l = \mathop{\sum}\limits_{\theta \in \Gamma\sigma\Gamma}\pi(\theta)l.
$$
\end{proof}

We describe below the straightforward inclusions between spaces of vectors invariant to subgroups in $\S$.

\begin{lemma}\label{heckepart2bis}
Consider the   family of Hilbert spaces
$l^2(\Gamma_0 \setminus \Gamma)$, $  \Gamma_0\in \S$.
On this family of Hilbert spaces, the scalar product is normalized so that the embeddings $l^2(\Gamma_0 \setminus \Gamma) \subseteq l^2(\Gamma_1 \setminus \Gamma)$ are isometric for all $\Gamma_0 \subseteq \Gamma_1$.
Then
\item(i)
 The Hilbert spaces $H^{\Gamma_0}$ are isometrically isomorphic to $l^2(\Gamma_0 \setminus \Gamma) \otimes L$, for all $\Gamma_0\in \S$,
 with the inclusion $$H^{\Gamma_0} \subseteq H^{\Gamma_1}$$ obtained for
$\Gamma_1 \subseteq \Gamma_0$ by tensoring with $L$, the isometric inclusion
$$
l^2(\Gamma_0 \setminus \Gamma) \subseteq l^2(\Gamma_1 \setminus \Gamma).$$
\item (ii) Let $\pi$ be a representation with properties (i), (ii) in Definition \ref {TC}. Then, the Hilbert space  $\overline{H} ^{\rm p}$ is $\Gamma$-equivariantly isometrical isomorphic to $L^2(K, \mu) \otimes L $. Hence
 $\overline{\pi}^{\rm p} |_{K}$ is a multiple of the left regular representation.

\end{lemma}

Using the above identification, we construct a block matrix representation for the Hecke operators.

 \begin{cor}\label{heckepart2} We use the notations introduced in Lemma \ref{heckepart2bis}.
 Denote the canonical matrix unit of $B(l^2(\Gamma_0 \setminus \Gamma))$ by
$$
(e_{\Gamma_0s_i, \Gamma_0s_j})_{i, j = 1, 2, \ldots, [\Gamma : \Gamma_0]}.
$$

\noindent  We use the isomorphism defined  above
$$
B(H^{\Gamma_0}) \cong B(l^2(\Gamma_0 \setminus \Gamma)) \otimes B(L).
$$

\noindent Then, the Hecke operator ${ [\Gamma_0:(\Gamma_0)_\sigma]}P_{H^{\Gamma_0}}\overline{\pi}^{\rm p}(\sigma)P_{H^{\Gamma_0}}$ is represented as
\begin{equation}\label{heckematrix}
\mathop{\sum}\limits_{i, j}\mathop{\sum}\limits_{\theta \in s_i^{-1}\Gamma_0\sigma\Gamma_0s_j}P_L\pi(\theta)P_L \otimes e_{\Gamma_0s_i, \Gamma_0s_j}.
\end{equation}

\end{cor}

\begin{proof}
The statement  is  a consequence of the fact that the projection onto $L^{\Gamma_0} = \oplus \pi(s_i)L$ is $\mathop{\sum}\pi(s_i)P_L\pi(s_i)^{\ast}$, where $s_i$ are the coset representatives introduced at the beginning of the proof.
We use the coset  representatives to construct a unitary operator, mapping $L^{\Gamma_0}$ onto $\ell^2(\Gamma_0)\otimes L$. This will map 
$\oplus \pi(s_i) l_i$ onto $\oplus [\Gamma_0 s_i]\otimes l_i$, for all $l_i\in L$.  Conjugating the operator in formula (\ref{agamma})
by this unitary, we obtain the formula (\ref{heckematrix})  in the statement.

No further renormalization in formula (\ref{tr}) is needed, as
it can be directly checked by letting $\sigma$ be the identity element in $G$. Then,
for $\Gamma_0\in \S$, the left hand side of the equation is ${\rm dim}\ H^{\Gamma_0}$. Since $L$ is a $\Gamma$-wandering subspace of $H$, the right hand side counts how many times the identity element belongs to $s_i^{-1}\Gamma_0s_i$ and multiplies the result by the dimension of the space $L$.


\end{proof}

\vskip10pt

\begin{rem}
In particular, if $L$ is of finite dimension, then we have the following formula for the traces of the Hecke operators:
\begin{equation}\label{tr}
{\rm Tr}(
{ [\Gamma_0:(\Gamma_0)_\sigma]}\big[
P_{H^{\Gamma_0}}\overline{\pi}(\sigma)P_{H^{\Gamma_0}}\big]) = \mathop{\sum}\limits_{i}\mathop{\sum}\limits_{\theta \in s_i^{-1}\Gamma_0\sigma\Gamma_0s_i}{\rm Tr}(P_L\pi(\theta)P_L).
\end{equation}
\end{rem}

We note that the correspondence $\pi\rightarrow\overline{\pi}^{\rm p}$ also
preserves the dimension function $\dim_{\{\pi(\Gamma)\}''}H$. This is explained in the following result:
\vskip10pt

\begin{lemma}\label{compact} Let $\pi$ be a unitary representation as above.
 Then
 \item(i)
The unitary representation $ \overline{\pi}^{\rm p}$ extends to a  $C^\ast$-algebra representation of the full amalgamated free product $C^\ast$-algebra $C^\ast(\overline G) \ast_{{\rm C}^\ast(K)} B(L^2(K))$ into $B(\overline {H}^{\rm p})$.

\item(ii)
The unitary representation  $\overline{\pi}^{\rm p}$  of $\overline{G}$ has the property that $\overline{\pi}_0^{\rm p} |_{K}$ is  a finite multiple of the left regular representation of $K$ on $L^2(K, \mu)$.

\end{lemma}

\vskip10pt

\begin{proof}
To prove (i), we have to construct the representation of $B(L^2(K))$ into $B(\overline {H}^{\rm p})$
Since $B(L^2(K))$ is generated by the algebra of convolutors with continuous functions on $K$ and by the functions in $C(K)$, viewed as multiplication operators, it is sufficient to describe the representation of the second algebra into $B(L^2(K))$, that is compatible with the previous one. By translation invariance, it is sufficient to consider the following case: Let $K_0$ be the closure, in the profinite topology, of a subgroup $\Gamma_0\in \S$. Then $\chi_{K_0}$, viewed as an element of the latter algebra, will act on a vector of the form $ \mathop{\sum}\limits_{\gamma_0 \in \Gamma_1}\pi(\gamma_0)l$, $\Gamma_1\in \S, l\in L$, by mapping it into
$ \mathop{\sum}\limits_{\gamma_0 \in \Gamma_0\cap \Gamma_1}\pi(\gamma_0)l$.  It is easy to check that this defines a $C^\ast$-algebra representation of the $C^\ast$-algebra $C^\ast(\overline G) \ast_{{\rm C}^\ast(K)} B(L^2(K))$.

To prove (ii), we note that  by restricting $\overline{\pi}^{\rm p} $ to $K$, the operators $W^{\Gamma_0}$, $\Gamma_0 \in \S$ become intertwiners between $\overline{\pi |_K}$ and the representation $\lambda_K\otimes {\rm Id}_L$ acting on $L^2(K) \otimes L.$


\end{proof}

\vskip10pt

\section{Examples of representations and associated spaces of $\Gamma$-invariant vectors}\label{theexamples}

In this section we discuss concrete examples  which  verify the axioms introduced in Definition \ref{formalism}, and the construction  in Proposition \ref{Lspaces} may be  applied.
We will freely use in this section the notations from the above mentioned definition, and the successive statements.

The following example plays, for unitary representations on spaces of $\Gamma$-invariant vectors, the role that the left regular representation of a discrete group plays in the space of representations of a discrete group.  It would be tempting to address this representation with the terminology "regular representation".

\begin{ex}\label{regular}Let $\pi=\lambda_G$ be the left regular representation of   $G$ acting on $H = l^2(G)$. In this case, the vector space  $\V$ is the linear space of functions on $G$. For $\Gamma_0 \in \S$, $\V^{\Gamma_0}$ is the space of left $\Gamma_0$ invariant functions on $G$, that is functions on $\Gamma_0\backslash G$.
Then the Hilbert space
$H^{\Gamma_0}$ is
 $l^2(\Gamma_0 \setminus G)$. The scalar product is defined so that, if $\Gamma_1\subseteq \Gamma_0$, $\Gamma_0,\Gamma_1\in \S$, the inclusions 
 $$l^2(\Gamma_0 \setminus G)\subseteq l^2(\Gamma_1 \setminus G), $$ are isometric.

  Clearly $\V_{\infty}$ is in this case the Hilbert space completion of the space $$\mathop{\bigvee}\limits_{\Gamma_0 \in \S}l^2(\Gamma_0 \setminus G).$$

\noindent This  is    $L^2(\overline{G}, \mu).$
Then, the representation $\overline{\pi}^{\rm p}$ is  the left regular representation $\lambda_{\overline{G}}$ acting on on $L^2(\overline{G}, \mu)$.

Note that here we implicitly use an identification on the vector spaces having as a basis the set  of cosets. This is the following: if $\Gamma_1 \subseteq \Gamma_0$ and $\Gamma_0 = \bigcup\Gamma_1s_i$, then, as vectors
in $l^2(\Gamma_0 \setminus G)\subseteq l^2(\Gamma_1 \setminus G)$, we have
\begin{equation}\nonumber
[\Gamma_0 g] = \mathop{\sum}\limits_{i}[\Gamma_1 s_i g], \quad  g \in G.
\end{equation}
In particular,  all quasi-regular representations $\lambda_{G/\Gamma_0}$ of $G$ onto spaces of cosets are subrepresentations of $\lambda_{\overline{G}} |_{G}$. Indeed, by using the equality in the above formula,
  we obtain that for all $\Gamma_1 \in \S$, $$l^2(G/\Gamma_1) \subseteq \mathop{\bigvee}\limits_{\Gamma_0 \in \S}l^2(\Gamma_0 \setminus G).$$
The quasi-regular representations   occur   with infinite multiplicity in the left regular representation $\lambda_{\overline{G}} |_{G}$, as they commute with the right action of $G$.

A standard choice of a $\Gamma$-wandering, generating subspace of $l^2(G)$, will consists into a choice $\cC \subseteq G$ of right coset representatives of $\Gamma$ in $G$ (thus $G$ would be the disjoint union $\mathop{\bigcup}\limits_{\sigma \in \cC}\Gamma{\sigma}$).
Since $P_{H^{\Gamma_0}}$ is the projection onto $l^2(\Gamma_0 \setminus G)$, we obtain,
using the above construction,  the standard representation of the Hecke operators and Hecke algebras (\cite{BC}, \cite{Bi}, \cite{Hal}, \cite{Tz}, \cite{LLN}).



\vskip10pt

\end{ex}

We describe a second standard example of the construction in Proposition \ref{Lspaces} corresponding to the case dim$_{\{\pi(\Gamma)\}''}H=\infty$.

\vskip10pt

\begin{ex}\label{koop} Assume that $(\cX, \nu)$ is an infinite  measure space such that $G$ acts by measure preserving transformations.
We assume that the restriction of the action of $G$ to $\Gamma$admits a fundamental domain $F$ in $\cX$, with measure $\nu(F) = 1$. For every $\Gamma_0$ in $\S$,  fix a system of representatives of cosets $$\Gamma = \bigcup \Gamma_0 s_i.$$ Let $$F_{\Gamma_0} = \bigcup s_i F.$$ Then $F_{\Gamma_0}$ is a fundamental domain for $\Gamma_0$. We renormalize the measure $\nu$ on $F_{\Gamma_0}$, and consider
 $$\nu_{\Gamma_0} = \dfrac{1}{[\Gamma : \Gamma_0]}\nu.$$
\noindent The choice of representatives induces a  projection $\pi_{\Gamma_0} : F_{\Gamma_0} \to F$, which simply maps $s_if$ into $f$, for $f$ in $F$. Taking the adjoint we obtain an isometric inclusion $$L^2(F, \nu) \subseteq L^2(F_{\Gamma_0}, \nu_{\Gamma_0}).$$

The unitary representation of $G$ on $L^2(\cX, \nu)$ is simply the Koopman representation
$$
\pi_{\rm Koop}(g)f(x) = f(g^{-1}x), \quad x \in \cX, \; g \in G, \; f \in L^2(\cX, \nu).
$$

We  use the formalism in Definition $\ref{formalism}$, and let $\V$ be the linear space of measurable functions on $\cX$. The subspace $\V^{\Gamma_0}$ 
clearly consists of functions in $\V$ that are $\Gamma_0$- equivariant.  Then $H^{\Gamma_0}$ is canonically identified to  the Hilbert space $L^2(F_{\Gamma_0}, \nu_{\Gamma_0})$. This space is also  identified with a subspace of the $\Gamma_0$-invariant functions on $\cX$.

It is clear that in this case the Hilbert space $\overline{H}^{\rm p}$ is isometrically isomorphic to
 $$L^2(K, \mu) \otimes L^2(F,\nu) =L^2(K\times F, \mu\times\nu).$$
   The representation $\overline{\pi}^{\rm p} |_{K}$ is simply ${\rm Id}_{L^2(F)} \otimes \lambda_K$, where $\lambda_K$ is the left regular representation of $K$ on $L^2(K, \mu)$.

    The above construction also proves that the representation $\overline{\pi_{\rm Koop}}^{\rm p}|_G$ is a Koopman unitary representation itself. It is easily recovered from the initial representation $\pi$. Indeed, taking the counting measure $\varepsilon$ on $\Gamma$, one has an isomorphism of measure spaces $$(\cX, \nu)\cong(\Gamma, \varepsilon)\times (F, \nu).$$
The action of $G$ on $\cX$, in the above identification is described in terms of a cocycle on $G\times F$ with values in $\Gamma$, where $\Gamma$ acts by left multiplication on the factor $\Gamma$ in the product $\Gamma\times F$.
When replacing the factor   $\Gamma$ in the above product, by the factor  $K$ in the product $K \times F$, we obtain a measure preserving action of $G$ on  the measure space $(K\times F, \mu\times\nu)$, having the same cocycle as the action of $G$ on $\Gamma\times F$.  Then, the unitary represention $\overline{\pi_{\rm Koop}}^{\rm p}|_G$ is in fact the unitary Koopmann representation corresponding to the action of $G$ on $(K\times F, \mu\times\nu)$

In the above construction, the projection  $P_L$ is  the multiplication operator by the characteristic function of $\chi_F$.
The convergence condition requiring that $$\mathop{\sum}\limits_{\theta \in \Gamma\sigma\Gamma}P_L \pi_{\rm Koop}(\theta)P_L,$$ be so-convergent is obvious in this case, since the above sum is the Hecke operator (see e.g. \cite{Ra5}).

\end{ex}

\vskip10pt

We also briefly describe below  how    the   framework from Proposition \ref{Lspaces} may be used    for spaces of  automorphic forms. This will be also analyzed in detail in the next section, in a more general setting.

\begin{ex}\label{auto}
Consider   the measures $\nu_n=({\rm Im} z)^{n-2}d\bar{z}dz$ on $\mathbb H$.
 Let $\pi_n, n>1, n\in \mathbb N$ be  the discrete series (\cite{la}) of (projective) unitary representations of  $\PSL(2,\R)$. This representations act on the Hilbert space $H_n = H^2(\bH, \nu_n)$.
 Their formula is similar to the Koopmann unitary representation, except for a modularity factor (see e. g. \cite{la}).  We also consider the larger Hilbert space $\wideparen{H}_n = L^2(\bH, \nu_n),n\geq 1$.

  We let $G = \PGL(2,\Z [\frac{1}{p}])$, where $p$ is a prime, and let $\Gamma = \PSL(2,\Z)$. If $\pi_n$ is a projective unitary representation, i.e. if $n$ is odd, then we are also given a two cocycle (expressing the projectivity of $\pi_n$). This cocycle is  in this case  $\mathbb Z_2$-valued, and hence it extends to a 2-cocycle on the Schlichtling completion.
    In this case this completion  is $PGL(2,\mathcal Q_p)$, where $\mathcal Q_p$ is the $p$-adic field.

By the results in \cite{GHJ}, Section 3.3, it follows  that $\pi_n |_{\Gamma}$ is a (not necessarily integer) multiple of the left regular representation $\lambda_{\Gamma}$. Indeed,
 $${\rm dim}_{\{\pi_n(\Gamma)\}''}H_n = \frac{n-1}{12}.$$ Consequently, if $\frac{n-1}{12}$ is not an integer, then there is no Hilbert space $L$ such that $H_n \cong l^2(\Gamma) \otimes L$ as $\Gamma$-modules.
Moreover, even if $\frac{n-1}{12}$ is an integer, there is no canonical choice of $L$, which would allow to proceed as in Proposition \ref{Lspaces}.

To overcome this problem we use a $\Gamma$-wandering, generating subspace of a representation $\wideparen{\pi}_n$  that contains $\pi_n$ as a subrepresentation.
We let $\wideparen{\pi}_n$ be the unitary representation of $\PSL(2,\R)$, given by the same algebraic formula on functions on $\mathbb H$, as the algebraic formula that determines the representation $\pi_n$.
%
The same computation that shows that $\pi_n$ is a unitary representation also plainly proves that $\wideparen{\pi}_n$ is a unitary representation.

It is well known (e.g. \cite {He}, \cite {Pe}) that the associated Hilbert space $H_n^{\Gamma}$ is the finite dimensional Hilbert space consisting of automorphic forms, of weight $n$,  for the group $\Gamma = \PSL(2,\Z)$.
To apply the formalism  in Definition $\ref{formalism}$, we let $\V$  be the space of analytic functions in $\bH$. 
In the next section we will use this framework to compute traces of Hecke operators.

Then, to describe the scalar product on $H_n^{\Gamma}$, one uses the Hilbert space scalar product from the previous example for the unitary representation $\wideparen{\pi}_n$.   If $F$ is a fundamental domain for the action of $\Gamma$ on $\mathbb H$, then we let $P_L$ be the multiplication operator $M_{\chi_F}$ with the characteristic function $\chi_F$, acting on $L^2(\bH, \nu_n)$.
The scalar product is explicitly constructed in formula (\ref{genpet}) in the next section.
In the particular case treated in this example, the above mentioned formula for the scalar product  turns out to be the canonical   Petersson (\cite{Pe}) scalar product.
The procedure described above will  be later used    to obtain an explicit description of the Hecke operators and to compute their traces.

\end{ex}

\

\

We give one more example of when  the  framework in Theorem \ref{Lspaces} applies   for the construction of $\Gamma$-invariant vectors. This example  corresponds to representations of the form $\pi \otimes \pi^{\rm op}$, where $\pi^{\rm op}$ is the complex conjugate.

\begin {ex}\label{ad}
Let $\Gamma \subseteq G$, $\pi, \pi_0, P_0, P_L$ be as in Definition \ref{TC}. Consider the diagonal unitary representation $\tilde{\pi} = \pi_0 \otimes \pi_0^{\rm op}$ of $G$. In this example we assume, for the simplicity of the presentation, that all the subgroups in $\S$ have infinite, non-trivial conjugacy classes (briefly i.c.c). This corresponds to the fact that all the associated von Neumann algebras have unique traces (are factors, i. e. have no center).

   Note that even if $\pi_0$ is projective, the representation $\pi_0 \otimes \pi_0^{\rm op}$ is unitary, with no cocycle. Moreover, since in this case the Murray-von Neumann dimension is infinite, it follows that the unitary representation $\pi_0\otimes \pi_0^{\rm op}$ verifies the conditions from Theorem \ref {Lspaces}.  Since we reserved the notation
 $\overline {H_0}^{\rm p}$ for other purposes,  we use here the notation $H_0^{\rm op}$ for the conjugate Hilbert space of $H_0$.

 Then the representation  $\pi_0\otimes \pi_0^{\rm op}$ is unitarily equivalent to the representation ${\rm Ad\, } \pi_0$, defined on $G$, with values into the unitary group of the Hilbert space
 $$ H_0\otimes H_0^{\rm op}\cong \cC_2(H_0)\subseteq B(H_0). $$ Here $\cC_2(H_0)$ is the ideal   consisting of the Hilbert-Schmidt operators acting on $H_0$ ([Ta]). The formula for the representation
 ${\rm Ad\, }\pi_0$ is
 $${\rm Ad\, }\pi_0(g)(X) = \pi_0(g)X\pi_0(g)^{-1}, \quad X \in \cC_2(H_0), g\in G.$$
 It obviously extends to a (non unitary) representation of  $G$ into the (inner) automorphisms of $B(H_0)$

In the setting of  Definition \ref{formalism},  we let the space $\V$ be  $B(H_0)$.
Then
$$\V^{\Gamma}=\{\pi_0(\Gamma)\}'=\{X\ | \ [X, \pi(\gamma)]=0 {\rm\ for\ all\ } \gamma\in \Gamma\} \subseteq B(H_0).$$
\noindent  More generally, for $\Gamma_0 \in \S$ we have that
 $$\V^{\Gamma_0} = \{ \pi(\Gamma_0) \}' \subseteq B(H_0).$$
\noindent Since we  assumed that all the groups  $\Gamma_0$ in $\S$ are i.c.c. groups, it follows
that  the algebras $\{\pi_0(\Gamma_0)\}'$, $\Gamma_0 \in \S$, are type II$_1$ factors, and consequently  each of them is endowed with a unique   normalized trace $\tau_{\Gamma_0}$.

We let $\A_{\infty}$ be the type II$_1$ factor obtained as the inductive, trace preserving, directed limit of the factors $\{ \pi_0(\Gamma_0) \}'$, $\Gamma_0 \in \S$.  Then $\A_{\infty}$ has a unique  trace $\tau$ defined by the requirement that
 $$\tau|_{\{\pi_0(\Gamma_0)\}'} = \tau_{\Gamma_0},\quad\Gamma_0 \in S.$$

  For $\sigma \in G$, $\Gamma_0\in\S$,   ${\rm Ad\, }\pi_0(\sigma)$ maps $$\{ \pi_0(\Gamma_0\cap \Gamma_{\sigma^{-1}} )\}'$$
 \noindent into
   $$\{ \pi_0(\sigma\Gamma_0\sigma^{-1} \cap \Gamma_\sigma) \}'.$$
 \noindent It follows that
 ${\rm Ad\, }\pi_0(\sigma)$ also maps $\A_{\infty}$ onto $\A_{\infty}$.
Thus ${\rm Ad\, }\pi_0(\sigma)$, $\sigma \in G$, extends to an element in the automorphism group $\Aut(\A_{\infty})$ of the factor $A_\infty$.

To obtain the Hilbert space of $\Gamma_0$-invariant vectors, we use the standard $L^2$-spaces associated to the corresponding II$_1$ factors (for notations see e.g. \cite {Ta}). Thus $$(H_0\otimes H_0^{\rm op})^{\Gamma_0} = L^2(\{ \pi(\Gamma_0) \}', \tau_{\Gamma_0})$$
 \noindent and
$$\overline{(H_0\otimes H_0^{\rm op})}^{\rm p} = L^2(\A_{\infty}, \tau).$$
 \noindent In particular, if dim$_{\{ \pi(\Gamma) \}''}H_0=1$, then
 $$(H_0\otimes H_0^{\rm op})^{\Gamma}\cong L^2(\L(\Gamma),\tau)\cong\ell^2(\Gamma).$$

 The unitary representation ${\rm Ad\ }\pi(\sigma)$, $\sigma \in G$, induces consequently  the unitary representation   
 $$ \overline{{\rm Ad\ }\pi}^{\rm op}= \overline{\pi_0 \otimes \pi_0^{\rm op}} $$
corresponding to $\pi_0 \otimes \pi_0^{\rm op}$, as defined   in Theorem \ref{Lspaces}.

Although this is not needed in this paper, we note that by Jones's index theory (\cite{Jo}), by identifying  the Jones's projection for the inclusion
$$\{\pi_0(\Gamma_0)\}''\subseteq \{\pi_0(\Gamma)\}''$$
 \noindent with the characteristic function of the closure of the subgroup $\Gamma_0$ in $K$, it follows (see \cite{Ra4}) that
 $\A_\infty$ is isomorphic to the von Neumann algebra crossed product algebra
 $\mathcal L(\Gamma\rtimes L^\infty(K,\nu))$, where $\Gamma$ acts by left translations on $K$. The representation $\overline{{\rm Ad\ }\pi}^{\rm p} |_{\Gamma}$  acts identically on $\L(\Gamma) \subseteq \A_{\infty}$, and by right translations on $K$.

\end{ex}

\vskip20pt

\section{Construction of the representation $\overline{\pi_0}^{\rm p}$ in the absence of $\Gamma$-wandering, generating subspace}\label{pi0}

In this section we are analyzing the case of a unitary representation $\pi_0$ of $G$ on a Hilbert space $H_0$ such that   ${\rm dim}_{\{\pi_0(\Gamma)\}''}H_0$ is not necessary an integer. We assume that the above mentioned dimension is a finite, positive number. Thus there might be   no $\Gamma$-wandering, generating subspace $L \subseteq H_0$, as in Definition \ref{wandering}.


We assume  the conditions in Definition \ref{TC}. We recall that the assumptions in the above definition require  that  there exists a  unitary   representation $\pi$ of $G$, on a larger Hilbert space $H$ and having a $\Gamma$-wandering generating subspace $L$. The initial representation $\pi_0$ is required to be a subrepresentation of $\pi$.
We recall that, as in Definition \ref{TC}, by $P_L, P_0$ we denote, respectively  the orthogonal projection from $H$ onto $L$ and, respectively onto $H_0$.



The spaces of $\Gamma_0$-invariant vectors are constructed, as in Proposition \ref {Lspaces}, as spaces of formal  sums,
over the group $\Gamma_0\in \S$. The $\Gamma_0$-invariant vectors  are consequently identified, as in Section \ref{axioms}, with $\Gamma_0$-invariant, unbounded linear forms on the  Hilbert space $H_0$.

The use of an auxiliary representation $\pi$, that possesses a $\Gamma$-wandering, generating subspace, and that contains the original representation $\pi_0$ as a subrepresentation, is  suggested  by the case of automorphic forms (see the description in Example \ref{auto}). Formally automorphic forms are vectors fixed by  $\pi_n|_\Gamma$ acting on the Hilbert space $H_n$. It is impossible  to find genuine vectors with this property since square summability fails. On the other hand the algebraic formula defining the representation $\pi_n$ admits $\Gamma$-invariant, analytic functions. These are the automorphic forms (\cite{He}).  To construct a Hilbert space structure on the space of automorphic forms of a given weight one uses the Petersson scalar product. To define the Petersson scalar product (\cite{Pe}), one uses a fundamental domain $F$ for the action of $\Gamma$ on $\bH$.
The space $L^2(F, \nu_n)$ is a $\Gamma$-wandering generating subspace for the larger unitary representation, containing $\pi_n$ as a subrepresentation, and acting on $L^2(\bH, \nu_n)$.

In the framework of this section,  the unitary representation $\pi_0$ is $\pi_n$, and the larger representation, having $\Gamma$-wandering, generating subspace acts on $L^2(\bH, \nu_n)$.
The projection $P_L$ is the multiplication operator on $L^2(\bH, \nu_n)$ with the characteristic function $\chi_F$ of the fundamental domain. Also, the projection $P_0$ is the (Bergman) projection onto the space of analytic functions $H^2(\bH, \nu_n)$.

We describe below the variations from the  procedure from Proposition $\ref{Lspaces}$ and Definition $\ref{formalism}$, needed to address  the present situation.
We start first with the construction of the  space of $\Gamma_0$-invariant vectors.

\vskip10pt

\begin{lemma}\label{tracespart1} We consider the groups $\Gamma \subseteq G$, and consider a   representation $\pi_0$  of $G$ with the properties from Definition \ref{TC}. We use the notations introduced in the above definition.
%
%
%
%
%


For $\Gamma_0$ in $\S$, we fix a system of right  coset representatives $s_i$ for $\Gamma_0$ in $\Gamma$, so  that   $\Gamma = \bigcup \Gamma_0 s_i.$ Let $L^{\Gamma_0}$ be the Hilbert space introduced in formula (\ref{lgamma}) with the norm subject to the renormalization condition defined in formula (\ref{renorm}) from the statement of Theorem \ref{hecke}.
%
  Let $P_{L^{\Gamma_0}}$ be the orthogonal projection from $H$ onto
$L^{\Gamma_0}$.

Then, the formula
\begin{equation}\label{proj}
\P_{\Gamma_0, L} = \mathop{\sum}\limits_{\gamma \in \Gamma_0} P_{L^{\Gamma_0}}\pi_0(\gamma)P_{L^{\Gamma_0}}\in B(L^{\Gamma_0})
\end{equation}
defines a projection in $B(L^{\Gamma_0})$.

\end{lemma}

\begin{proof}
The fact that $\P_{\Gamma_0, L}$ is a projection, and more generally,  the fact that  formula (\ref{a0}) in the next proposition defines a representation of the Hecke algebra of double cosets for $\Gamma_0 \in \S$, is a straightforward consequence of the following identity, valid for $\sigma_1, \sigma_2 \in G$, $\Gamma_0 \in \S$:
\begin{equation}\label{unitary}
\mathop{\sum}\limits_{\gamma \in \Gamma_0}P_{L^{\Gamma_0}}\pi_0(\sigma_1 \gamma)P_{L^{\Gamma_0}}\pi_0(\gamma^{-1}\sigma_2)P_{L^{\Gamma_0}} = P_{L^{\Gamma_0}}\pi_0(\sigma_1\sigma_2)P_{L^{\Gamma_0}}.
\end{equation}
The convergence in the above equation follows from the technical condition assumed in Definition \ref{TC}. Here, we are taking the pointwise operator product of two  operators series that are convergent in $\mathcal C_2(L)$.
Formula (\ref{unitary}) is  then a direct consequence of the fact that $$\mathop{\sum}\limits_{\gamma_0 \in \Gamma_0}\pi(\gamma)P_{L^{\Gamma_0}}\pi(\gamma^{-1})$$ is the identity operator on $H$. \end{proof}

In the following, we adapt the content of Proposition  \ref{Lspaces} to the context of a representation
$\pi_0$ with the properties introduced in Definition \ref{TC}.  In the next statement,  we construct the Hilbert space of $\Gamma_0$ "virtual" invariant vectors associated to $\pi_0$.


\begin{lemma}\label{tracespart2}
We assume the notations and definitions from Lemma \ref{tracespart1}.
Let  $H_0^{\Gamma_0}$ be  the space of formal series, of the following form
$$
H_0^{\Gamma_0} = \{ \mathop{\sum}\limits_{\gamma \in \Gamma_0}\pi_0(\gamma)l \; |\; l \in L^{\Gamma_0} \}.
$$
Let  $\D_{L, \pi}$ be defined  as in formula (\ref{d})  in Lemma \ref{Lspacespart1}.
This space  is subject to the same   identification as formula (\ref{ident}) in Definition \ref{Lspacespart2}. Hence
  \item {(i)}
$$
H_0^{\Gamma_0} = \{ \mathop{\sum}\limits_{\gamma \in \Gamma_0}\pi_0(\gamma)h \; |\; h \in \D_{L, \pi} \}.
$$
\item{(ii)}
The representation $\pi_0$ extends to a canonical representation $\pi_0^{\rm p}$ of $G$ into the linear isomorphism group of the vector space
$\mathop{\bigvee}\limits_{\Gamma_0 \in \S} H^{\Gamma_0}$.

\end{lemma}
\begin{proof} The part (i) is a straightforward consequence of the identification in   formula (\ref{ident}). The representation of $\pi_0^{\rm p}$ is constructed using the same procedure as in 
Proposition \ref{Lspaces}, formula (\ref{splitting}).
\end{proof}

In the following result, we define a compatible scalar product on  the spaces of $\Gamma_0$-invariant vectors introduced above, $\Gamma_0\in \S$. We also construct a unitary representation on the inductive limit of the Hilbert  spaces of $\Gamma_0$-invariant vectors.

\begin{thm}\label{traces}
In the context introduced above, we have:

\item{(i)} 
Let $\Gamma_0\in \S$, let   $h_1,h_2 \in \D_{L, \pi}$. By analogy with formula (\ref{scalarpet})   in Proposition \ref{Lspaces} we define
\begin{equation}\label {genpet}
 \langle \mathop{\sum}\limits_{\gamma_0 \in \Gamma_0}\pi_0(\gamma_0)h_1, \mathop{\sum}\limits_{\gamma_0' \in \Gamma_0}\pi_0(\gamma_0')h_2 \rangle_{\infty} 
=\langle P_{L^{\Gamma_0}}\big(\mathop{\sum}\limits_{\gamma_0 \in \Gamma_0}\pi_0(\gamma_0)h_1\big), \mathop{\sum}\limits_{\gamma_0' \in \Gamma_0}\pi_0(\gamma_0')h_2\rangle
\end{equation}
%
%
%
%


For $\Gamma_1\subseteq  \Gamma_0$ subgroups in $\S$ the inclusions $H_0^{\Gamma_0}\subseteq H_0^{\Gamma_1}$ are isometric. Hence formula (\ref {genpet})
extends to a scalar product  $\langle \cdot, \cdot \rangle_{\infty}$ on $\mathop{\bigvee}\limits_{\Gamma_0 \in \S} H^{\Gamma_0}$.

\item(ii)The  scalar product $\langle \cdot, \cdot \rangle_{\infty}$
 on the space $\mathop{\bigvee}\limits_{\Gamma_0 \in \S} H^{\Gamma_0}$ verifies the assumptions in Definition \ref{formalism}. 

\item{(iii)}. Let $\tilde L^{\Gamma_0}$ be another $\Gamma_0$-wandering, generating subspace of $H$ for the representation $\pi|_{\Gamma_0}$. Assume that there exists two orthogonal partitions  $\tilde L^{\Gamma_0}=\oplus_{\gamma\in \Gamma_0} \tilde L_{\gamma }$ and
 $ L^{\Gamma_0}=\oplus_{\gamma\in \Gamma_0}  L_{\gamma }$, such that $\tilde L_{\gamma }=\pi(\gamma)L_{\gamma }$, for $\gamma \in \Gamma_0$. Then,
   substituting in the  formula (\ref{genpet}) the space 
$ L^{\Gamma_0}$ by $\tilde L^{\Gamma_0}$, does not change the value of the scalar product.

\item{(iv)} Let $\overline{H}^{\rm p}_0$ be the Hilbert space completion of the inductive limit $\mathop{\bigvee}\limits_{\Gamma_0 \in \S} H^{\Gamma_0}$  of the spaces constructed in (i).
For $g\in G$, we define the unitary $\overline{\pi}_0^{\rm p}(g)$ on $\overline{H}^{\rm p}_0$ by exactly the same formula as formula (\ref{splitting}) from  the proof of  Proposition \ref{Lspaces}.
Then
 $\overline{\pi}_0^{\rm p}$ is a unitary representation of $G$ into the unitary group of the Hilbert space $\overline{H}^{\rm p}_0$.

\end{thm}

\vskip10pt


Before proving the theorem, we make a few remarks. The formula (\ref{genpet}) in the definition of the scalar product may also be continued with
\begin{equation}\nonumber
\label{consistence}
\langle \mathop{\sum}\limits_{\gamma_0 \in \Gamma_0}
\pi(\gamma_0)P_0h,
\mathop{\sum}\limits_{\gamma_0' \in \Gamma_0}
\pi(\gamma_0')P_0l
 \rangle_{\infty}.
 \end{equation}
\noindent This proves that  the scalar product that we are defining  on $H_0^{\Gamma_0}$ is consistent with the scalar product on the space $H^{\Gamma_0}$ introduced in Proposition \ref{Lspaces}.

In the case of automorphic forms, when $\pi_0$ is  the representation   $\pi_n$  and $P_0$ is a Bergman projection onto the associated space of square integrable analytic functions, the technical conditions (v), (vi)  in Definition \ref{TC} follow from the fact that the reproducing kernel for the space of automorphic forms is the sum, over $\Gamma$, of the reproducing kernels, restricted to the fundamental domain, for the operators $\chi_F\pi_0(\gamma)\chi_F$, $\gamma\in\Gamma$.  The same is valid for the sum over any double coset, the sum of the kernels being equal in this case to the reproducing kernel for the Hecke operator associated to a double coset. The convergence of reproducing  kernels holds true in the Hilbert-Schmidt norm ([Za]).  Note that in the same paper ([Za]) the absolute convergence for the sum of  traces is proved.

  The similarity with the Petersson scalar product formula follows from the fact that in the particular case corresponding to automorphic forms, the projection $P_L$ is substituted with the projection operator $M_ {\chi_F}$ obtained by multiplication with the characteristic function $\chi_F$ of the fundamental domain $F$. The fact that $P_0M_ {\chi_F}$ is trace class was checked in [GHJ], Section 3.3.
  The formula (\ref{genpet}) is reminiscent of the Pettersson scalar product. Indeed, to define the Petterson
  scalar product of two automorphic forms $f,g$, one proceeds with the $L^2$-scalar product of  $\chi_Ff$ and $g$.

  Note that one could have used directly formula (\ref{a0}) in the next proposition to define the Hecke operators.
There (see also [Ra2]) we give a direct proof that formula (\ref{a0}) is a representation of the Hecke algebra of double cosets of $\Gamma_0$ in $G$.
On the other hand, using the space $H_0^{\Gamma_0}$ as a space of averaging sums over $\Gamma$, implies  that the spaces of $\Gamma$-invariant vectors that we are considering in the theorem, correspond to   the spaces of automorphic forms.

The advantage of the approach considered in the  theorem, is the fact  that we have concrete formulae for the unitary representation $\overline{\pi_0}^{\rm p}$, directly described in the terms of the original representations $\pi_n$ and its interaction with $P_0M_ {\chi_F}$. This will be used later in the paper  for  computations of traces and of  characters of the associated unitary representations.

\vskip10pt

\begin{proof}[Proof of Theorem \ref{traces}]

We use the  fact that the sum in formula (\ref{proj}), Lemma \ref{tracespart1} determines a finite dimensional projection.    Then the scalar product in formula \ref{genpet} is further equal to
\begin{equation}\label{simplified}
 \langle P_{L^{\Gamma_0}}\big(\mathop{\sum}\limits_{\gamma_0 \in \Gamma_0}\pi_0(\gamma_0)h_1\big), P_{L^{\Gamma_0}} \big(\mathop{\sum}\limits_{\gamma_0' \in \Gamma_0}\pi_0(\gamma_0')h_2\big)\rangle 
 =
 \end{equation}
 \begin{equation}\nonumber
= \langle \P_{\Gamma_0, L}h_1, \P_{\Gamma_0, L}h_2\rangle 
 = \langle \P_{\Gamma_0, L}h_1, h_2\rangle
\end{equation}
 Since $\P_{\Gamma_0, L}$ is a finite dimensional projection in $L^{\Gamma_0}$, this is a well defined scalar product on the space $H_0^{\Gamma_0}$.

 The fact that for $\Gamma_0,\Gamma_1\in S$, $\Gamma_1\subseteq \Gamma_0$, the inclusions $H_0^{\Gamma_0}\subseteq H_0^{\Gamma_1}$ are isometric, is proved exactly
 as in Proposition \ref{Lspaces}.
  The statement (ii) is a straightforward consequence of the formula (\ref{genpet}) in the definition
 of the scalar product on $\Gamma_0$ invariant vectors. It corresponds,  to the independence from the choice of the fundamental domain, in the formula of the Peterson scalar product.


We prove the  invariance assumption from point 3) in Lemma \ref{formalism}.  
 We will prove that 
$\pi_0^{\rm p}(\sigma), \sigma \in G$
 maps 
 $H_0^{\Gamma_{\sigma^{-1}}}$
  isometrically  onto  
  $H_0^{\Gamma_{\sigma}}$.
  Fix as above $h_1,h_2$ in  $\D_{L, \pi}$. 
  Then, using formula \ref{simplified}, and using the construction of the representation 
  $\pi_0^{\rm p}$, we have 
  
 $$ \langle\pi_0^{\rm p}(\sigma)\big( \mathop{\sum}\limits_{\gamma_0 \in \Gamma_{\sigma^{-1}}}\pi_0(\gamma_0)h_1\big), \pi_0^{\rm p}(\sigma)\big(\mathop{\sum}\limits_{\gamma_0' \in \Gamma_{\sigma^{-1}}}\pi_0(\gamma_0')h_2\big) \rangle_{\infty} =$$
  $$\langle \mathop{\sum}\limits_{\gamma_0 \in \Gamma_{\sigma}}\pi_0(\gamma_0) \big(\pi_0(\sigma )h_1\big), 
\mathop{\sum}\limits_{\gamma_0' \in \Gamma_{\sigma}}\pi_0(\gamma_0')\big(\pi_0(\sigma )h_2\big) \rangle_{\infty}=
$$
$$\langle P_{L^{\Gamma_{\sigma}}}\big[ \mathop{\sum}\limits_{\gamma_0 \in \Gamma_{\sigma}}\pi_0(\gamma_0) \big(\pi_0(\sigma )h_1\big)\big], 
\big(\pi_0(\sigma )h_2\big) \rangle.
$$
 
  Since the $\pi_0(\sigma)$ is a unitary on $H_0$, this is further equal  to
\begin{equation}\label{intermediara}  
\langle \big[\pi(\sigma^{-1} )P_{L^{\Gamma_{\sigma}}}\pi(\sigma)\big] \pi_0(\sigma^{-1} )\big[ \mathop{\sum}\limits_{\gamma_0 \in \Gamma_{\sigma}}\pi_0(\gamma_0) \big(\pi_0(\sigma )h_1\big)\big], 
h_2\big) \rangle
\end{equation}
 In the preceding formula $\pi(\sigma^{-1} )P_{L^{\Gamma_{\sigma}}}\pi(\sigma)$ is the projection onto the space $\tilde L=\pi(\sigma^{-1} )L^{\Gamma_{\sigma}}$. This is a 
 $\Gamma_{\sigma^{-1}}$ wandering, generating subspace for the the representation  $\pi|_{\Gamma_{\sigma^{-1}}}.$ Denote by $P_{\tilde L}$ the orthogonal projection onto $\tilde L$.
 Then the term in formula (\ref{intermediara}) is further equal to
 \begin{equation}\label{intermediar1}
\langle P_{\tilde L}\big( \mathop{\sum}\limits_{\gamma_0 \in \Gamma_{\sigma^{-1}}}\pi_0(\gamma_0) h_1\big), h_2\rangle.\end{equation}
 Because of the technical condition (ii) in Definition \ref{TC}, the subspace $\tilde L$ is equivalent to the subspace $L^{\Gamma_{\sigma^{-1}}}$ in the sense of point (iii), in the present statement. Hence by point (iii), the term in formula (\ref{intermediar1}) is further equal  to 
 $$\langle P_{L^{\Gamma_{\sigma^{-1}}}}\big( \mathop{\sum}\limits_{\gamma_0 \in \Gamma_{\sigma^{-1}}}\pi_0(\gamma_0)h_1\big), h_2\rangle=$$
$$\langle  \mathop{\sum}\limits_{\gamma_0 \in \Gamma_{\sigma^{-1}}}\pi_0(\gamma_0)h_1, \mathop{\sum}\limits_{\gamma_0' \in \Gamma_{\sigma^{-1}}}\pi_0(\gamma_0')h_2 \rangle_{\infty}.$$
By the above chain of equalities, because of the definition of the scalar product in point (i)  we have proved   that 
$\pi_0^{\rm p}(\sigma)$
 maps 
 $H_0^{\Gamma_{\sigma^{-1}}}$
  isometrically  onto  
  $H_0^{\Gamma_{\sigma}}$. More generally, for $\Gamma_0\in S$, the same type of  argument gives that 
  $\pi_0^{\rm p}(\sigma)$ maps $H_0^{\Gamma_{\sigma^{-1}}\cap \Gamma_0}$ onto 
 $H_0^{\Gamma_{\sigma}\cap\sigma\Gamma_0\sigma^{-1}}$.
 Hence the invariance condition in point 3) from Definition \ref{formalism} holds true in the present case.
 Point (iv) is a straightforward consequence of the previous argument.

\end{proof}

In the next theorem, we describe the unitary equivalent representation of the Hecke operators acting on the spaces of $\Gamma_0$-invariant vectors  that were  introduced in the preceding theorem. We will  prove below that the projection $\P_{\Gamma_0, L}\in B(L^{\Gamma_0})$, introduced in formula (\ref{proj}), is   unitarily equivalent to a different   projection, introduced in formula (\ref{proj1}). This is useful for computing the dimensions of $\Gamma_0$-invariant vectors.  This is because we are proving that the range of  $\P_{\Gamma_0, L}$, which is a subspace of $L^{\Gamma_0}$,  is unitarily equivalent  to the space of vectors associated to $H_0$ that are fixed by ${\Gamma_0}$.

The reproducing kernel formula for  the projection onto the space of automorphic forms, and   the  reproducing kernel formula for the Hecke operators,  introduced in \cite{Za}, proves that the spaces  $H_0^{\Gamma_0}$, and the corresponding action of the Hecke operators, on the spaces of $\Gamma_0$-invariant vectors,   introduced in the previous proposition, and in the next theorem,  are the same (in the case of the upper halfplane) with  the ones in the classical case.

\begin{thm}\label{a0formula}
 We assume the context of the previous theorem. Let $\Gamma_0\in \S$.
By Lemma \ref{heckeconverted}, the representation of the Hecke algebra $\H_0(\Gamma_0,G)$ associated with the representation $\pi_0$ is  defined by the correspondence
\begin{equation}\label{cor1}
[\Gamma_0\sigma\Gamma_0]\rightarrow { [\Gamma_0:(\Gamma_0)_\sigma]}P_{H_0^{\Gamma_0}}\overline{\pi}_0^{\rm p}(\sigma)P_{H_0^{\Gamma_0}},
\quad\sigma \in G.
\end{equation}
\noindent Then
\item{(i)}
Consider the following infinite sums over cosets of $\Gamma_0$, which because of conditions (v), (vi) in Definition \ref{TC}, are convergent:
\begin{equation} \label {a0}
A_0(\Gamma_0 \sigma\Gamma_0) =
 \mathop{\sum}\limits_{\theta \in \Gamma_0\sigma\Gamma_0} P_{L^{\Gamma_0}}\pi_0(\theta)P_{L^{\Gamma_0}},\quad\sigma \in G.
 \end{equation}
 \noindent Then $$A_0(\Gamma_0 \sigma\Gamma_0) = \P_{\Gamma_0, L} A_0(\Gamma_0 \sigma\Gamma_0)\P_{\Gamma_0, L},$$ and hence
 $A_0(\Gamma_0 \sigma\Gamma_0)$ belongs to
   $B(\P_{\Gamma_0, L}L^{\Gamma_0})$.
 \item{(ii)}
The  correspondence
$$[\Gamma_0\sigma\Gamma_0]\rightarrow A_0(\Gamma_0\sigma\Gamma_0),\quad \sigma \in G,$$
determines a $\ast$-representation of $\H_0(\Gamma_0,G)$
 into  $B(P_{\Gamma_0,L} L^{\Gamma_0})$. This representation is unitarily equivalent to the
 $\ast$-representation (see the above formula  (\ref{cor1}))  of the Hecke algebra by means of Hecke operators associated with representation $\pi_0$.

\end{thm}

Note that, in particular, the operator  $A(\Gamma_0)= \P_{\Gamma_0, L}$ is unitarily equivalent to the projection on $\Gamma_0$ invariant vectors.

\begin{proof}
The  proof is similar to that of    Theorem \ref {hecke}.
As in Theorem \ref{hecke} and its proof, using formula (\ref{genpet}) for the scalar product, we may define  define partial isometries
$W^{\Gamma_0} : L^{\Gamma_0} \to H^{\Gamma_0}$ for $\Gamma_0 \in \S$ by
\begin{equation}\nonumber
W^{\Gamma_0}l = \mathop{\sum}\limits_{\gamma \in \Gamma_0}\pi_0(\gamma)l, \quad l\in L^{\Gamma_0}.
\end{equation}

Differently from   the case considered  in Theorem \ref {hecke},
the operators $W^{\Gamma_0}$ are partial isometries, having as  initial  space the projection  $\P_{\Gamma_0,L}$, introduced in formula (\ref{proj})  and images equal to the spaces
$H^{\Gamma_0}$, $\Gamma_0\in \S$.

We  use the partial isometry $W^{\Gamma_0}$, with initial space the projection $\P_{\Gamma_0, L}$ in $B(L^{\Gamma_0})$. Then $W^{\Gamma_0}$  transforms unitarily the Hecke operator $P_{H^{\Gamma_0}}\overline{\pi}^{\rm p}(\sigma)P_{H^{\Gamma_0}}$,  $\sigma \in G$ into the expression in formula (\ref{a0}).

Formula (\ref{unitary}) proves that for all $\Gamma_0 \in \S$, $\sigma \in G$, we have
$$
A_0(\Gamma_0\sigma\Gamma_0) = A_0(\Gamma_0\sigma\Gamma_0)\P_{\Gamma_0, L} = \P_{\Gamma_0, L}A_0(\Gamma_0\sigma\Gamma_0).
$$

The same formula proves that the operators
$$A_0(\Gamma_0\sigma\Gamma_0),\quad \sigma \in G,$$
determine a representation of the Hecke algebra of double cosets of $\Gamma_0$ in $G$.

\end{proof}

\vskip10pt

The previous proposition  gives an explicit representation of the Hecke operators, associated to the representation $\overline{\pi}_0^{\rm p}$ of $\overline{G}$ into $\overline{H}^{\rm p}_0$, by directly using the information from the original representation $\overline{\pi}_0^{\rm p}$. We summarize this in  Theorem \ref{L0}.

\vskip10pt

\begin{proof}[Proof of Theorem \ref{L0}] Let $L^{\Gamma_0}$ be the subspace defined in formula (\ref{lgamma}),  endowed with the normalized scalar  product defined in formula (\ref{renorm}). Using Theorem \ref {a0formula},  the proof of the formula (\ref{heckematrix0}) becomes  identical to the proof of the corresponding formula (\ref{heckematrix}) in Theorem \ref{hecke}. We use the choice of coset representatives from the statement. In passing from formula
(\ref{a0}) to formula (\ref{heckematrix0}), one simply uses the unitary operator 
$$\ell^2(\Gamma_0\backslash\Gamma)\otimes L\cong \oplus [\Gamma_0s_i]\otimes L,\quad L^{\Gamma_0}=\oplus \pi(s_i) L,$$
mapping $\oplus  [\Gamma_0s_i] \otimes l_i$ onto $\oplus \pi(s_i)l_i$, for $l_i\in L$.

\end{proof}

When $\sigma$ is the identity in the  formula (\ref{heckematrix0}), we obtain a  projection

 \begin{equation}\label{proj1}
 \tilde \P_{\Gamma_0, L}=
\mathop{\sum}\limits_{i,j}\mathop{\sum}\limits_{\gamma \in s_i^{-1}\Gamma_0s_j}P_{L}\pi_0(\gamma)P_L \otimes e_{\Gamma_0s_i, \Gamma_0s_j}.
\end{equation}
  Consequently, the operators in formula (\ref{heckematrix0}) belong to the algebra  $$\tilde \P_{\Gamma_0, L}\big[ B(l^2(\Gamma_0 \setminus \Gamma)) \otimes B(L)\big]  \tilde \P_{\Gamma_0, L}.$$

\vskip10pt

\begin{rem}\label{pin}
 We use the context of the previous theorem. Because of formula (\ref{heckematrix0}) one equivalent method to construct the    representation of the Hecke operators in formula (\ref{heckematrix0}) is as follows: consider, as in Example \ref{regular}, the vector space
 $$\V^{\Gamma_0}=l^2(\Gamma_0 \setminus G), \quad \Gamma_0\in S.$$
 On $\V^{\Gamma_0}$ we introduce the scalar product defined  as the linear extension of the following bilinear form
 $$\langle\Gamma_0\sigma_1, \Gamma_0\sigma_2\rangle_{\pi_0}=\frac{1}{[\Gamma:\Gamma_0]}
\sum_{\theta \in \sigma_1^{-1}\Gamma_0\sigma_2}{\rm Tr}(P_L\pi_0(\theta)P_L),\quad \Gamma_0\in \S, \sigma_1, \sigma_2\in G.$$

For $\Gamma_0\in \S$, we consider the usual algebraic representation of the Hecke operators on $\C (\Gamma_0 \setminus G). $
Considering on $\C (\Gamma_0 \setminus G) $ the  scalar product $\langle\cdot,\cdot\rangle_{\pi_0}$ introduced  in the above  formula, we obtain a unitary equivalent representation the Hecke algebra  to the representation introduced
in formula (\ref{heckematrix0}).

This corresponds to considering the state $\varepsilon$ on $C^\ast (\overline G)$, defined by
$$\varepsilon (\chi_{\sigma_1^{-1}\Gamma_0\sigma_2})= \frac{1}{[\Gamma:\Gamma_0]}
\sum_{\theta \in \sigma_1^{-1}\Gamma_0\sigma_2}{\rm Tr}(P_L\pi_0(\theta)P_L),\quad \Gamma_0\in \S, \sigma_1, \sigma_2\in G.$$

The state $\varepsilon$ can not be simultaneously used at all levels $\Gamma_0\in S$ because
of the renormalization factor $\frac{1}{[\Gamma:\Gamma_0]}$.
Note that the state $\varepsilon$ is in fact the composition of the trace with the family of completely positive maps constructed in Theorem \ref{cp}.
\end {rem}

\

\section{The values of the character $\theta_{\pi_0}$ associated  with representation $\pi_0$}\label{section traces}

In  this section we  derive a trace formula for the representation $\overline{\pi}_0^{\rm p}$. We note that $\overline{\pi}_0^{\rm p}$ is a type I representation of the $C^{\ast}$-algebra $C^{\ast}(\overline{G})$. As we noted in Proposition \ref{cala}, the representation  $\overline{\pi}_0^{\rm p}$ extends to a representation of $\A(G, \overline{G})$,
(or $\A_\epsilon(G, \overline{G})$ if a 2-cocycle is present). By Theorem \ref{L0},  we have  a formula  for the Hecke operator$P_{H_0^{\Gamma_0}}\overline{\pi}_0^{\rm p}(\sigma)P_{H_0^{\Gamma_0}}$,  associated with  the representation $\overline{\pi}_0^{\rm p}$. As explained below, this relates the trace formula for the representation $\overline{\pi}_0^{\rm p}$ with the trace formula for the representation $\pi_0$.

\vskip10pt

\vskip10pt

The following lemma is proved   in (\cite{Ca}, formula (13)). Here we give a different proof.
\begin{lemma}
The character   character $\theta_{\overline{\pi_0}^{\rm p}}$= "${\rm Tr}\; \overline{\pi}_0^{\rm p}$"  of the representation $\overline{\pi_0}^{\rm p}$  is computed by the formula (\ref{hyperfinite}).

\end{lemma}\label{cartiertr}
\begin{proof}
Using the fact that the character is locally integrable (\cite{Sal}) we have the following formula:
\begin{equation}\label{plancherel3}
\theta_{\overline{\pi_0}^{\rm p}}(\sigma)=
\mathop{\lim}\limits_{\mathop{\Gamma_0 \downarrow \; e}
\limits_{\Gamma_0 \in \S}} \frac{1}{\mu(\overline{\Gamma_0}\sigma
\overline{\Gamma_0})}
{\rm Tr}\big(\overline{\pi}_0^{\rm p}
({\chi}_{{\overline{\Gamma_0}}
\sigma
\overline{\Gamma_0}})\big),\quad \sigma \in G
\end{equation}
\noindent Clearly, using the notations  introduced before the statement of Lemma \ref{heckeconverted}, we have that
$$\mu(\overline{\Gamma_0}\sigma\overline{\Gamma_0})=
[\Gamma_0: (\Gamma_0)_\sigma]\mu(\overline{\Gamma_0}\sigma),\quad \Gamma_0 \in S, \sigma \in G$$
Since the measure $\mu$ is obtained from the Haar measure on the profinite completion of $\Gamma$, and since, by the general assumptions, $\mu$ is bivariant on $\overline{G}$, this is further equal to
$$[\Gamma_0: (\Gamma_0)_\sigma]\frac{1}{[\Gamma:\Gamma_0]}.$$
Hence, we  continue  the equality in the  formula \eqref{plancherel3} with:
\begin{equation}\nonumber
\mathop{\lim}\limits_{\mathop{\Gamma_0 \downarrow \; e}
\limits_{\Gamma_0 \in \S}}
 \frac{[\Gamma:\Gamma_0]}{[\Gamma_0: (\Gamma_0)_\sigma]}
 {\rm Tr}\big(
 \overline{\pi}_0^{\rm p}
({\chi}_{{\overline{\Gamma_0}}
\sigma
\overline{\Gamma_0}})
\big).
\end{equation}
\noindent Using formula (\ref{completehecke}) in the statement of Remark \ref{heckeconverted}, the above chain of equalities is continued with
\begin{equation}\label{cc}
\begin{aligned}
& \mathop{\lim}\limits_{\mathop{\Gamma_0 \downarrow \; e}
\limits_{\Gamma_0 \in \S}}[\Gamma:\Gamma_0]^2
{\rm Tr}\big( \overline{\pi}_0^p
({\chi}_{{\overline{\Gamma_0}}})
\overline\pi_0^{\rm p}({\sigma})
\overline\pi_0^{\rm p}({\chi}_{\overline{\Gamma_0}})
\big) \\ & \qquad =
\mathop{\lim}\limits_{\mathop{\Gamma_0 \downarrow \; e}
\limits_{\Gamma_0 \in \S}}\frac{1}{(\mu(\overline{\Gamma_0}))^2}
{\rm Tr}\big( \overline{\pi}_0^{\rm p}
({\chi}_{{\overline{\Gamma_0}}})
\overline\pi_0^{\rm p}({\sigma})
\overline\pi_0^{\rm p}({\chi}_{\overline{\Gamma_0}}))
\big).
\end{aligned}
\end{equation}
If $K_0$ is the closure of a subgroup in $\S$, then in $C^\ast (\overline G)$, using the product of convolutor operators we have that $(\chi_{K_0})^2= \mu(K_0)\chi_{K_0}$. Hence $\frac{1}{\mu(K_0)} \chi_{K_0}$ is a projection. We denote by  $\tilde\chi_{K_0}$ the renormalized convolutor operator $\frac{1}{\mu(K_0)} \chi_{K_0}$.
Thus the equality in formula (\ref{cc}) is continued with

$$
\mathop{\lim}\limits_{\mathop{\Gamma_0 \downarrow \; e}\limits_{\Gamma_0 \in \S}} {\rm Tr}(\overline{\pi}_0^{\rm p}(\tilde{\chi}_{\overline{\Gamma}_0})\overline{\pi}_0^{\rm p}(\sigma)\overline{\pi}_0^{\rm p}(\tilde{\chi}_{\overline{\Gamma}_0})).
%
$$


 Since $\overline{\pi}_0^{\rm p}$ is a representation of $C^{\ast}(G, \overline{G})$, this is equal to
$$
\mathop{\lim}\limits_{\mathop{\Gamma_0 \downarrow \; e}\limits_{\Gamma_0 \in \S}} 
{\rm Tr}(P_{H_0^{\Gamma_0}}\overline{\pi}_0^{\rm p}(\sigma)P_{H_0^{\Gamma_0}}).
$$



\end{proof}


\begin{proof}[Proof of Corollary \ref{plancherel}]
The formulae (\ref{hyperfinite}) and (\ref{f8}) imply immediately the conclusion of formula (\ref{realresult}).
We group together the terms in the  sum in formula (\ref{realresult}) according to classes of conjugation, and take the limit. The fact that the group $\Gamma^{\rm st}_g$ is trivial prevents any $\gamma \in \Gamma$ to show up more than once in the sum in formula (\ref{realresult}).
The absolute convergence of the sums involved in the limiting process
(Definition \ref{TC}) implies the result in the second formula of the Corollary.

\end{proof}

\

\

\begin{proof}[Proof of Corollary \ref {deligne}]
Recall that for  a group element  $g\in G$, the stabilizer group   $\Gamma^{\rm st}_g$ is $\{ g | \gamma g = g\gamma \}$.
 Let $\sigma$ be an element in
$G= \PGL(2,\mathbb Z[\frac {1}{p}])$ such that ${\Gamma^{\rm st}_\sigma}$ is trivial.
Then, the expression in formula (\ref{berezin}):  $$\mathop{\sum}\limits_{\gamma \in \Gamma}{\rm Tr}(P_L\pi_n(\gamma\sigma\gamma^{-1})),$$ may be  computed directly   by using Berezin's symbol function ([Be]).

Indeed, the above sum has the following expression:
$$
\begin{aligned}
& \mathop{\sum}\limits_{\gamma\in \Gamma}{\rm Tr}_{B(H_n)}
\big(M_{\chi_F}\pi_0(\gamma)\pi_0(\sigma)\pi_0(\gamma^{-1})\big) \\ & \qquad =
\mathop{\sum}\limits_{\gamma\in \Gamma}{\rm Tr}_{B(H_n)}
\big(M_{\chi_F}\pi(\gamma)\pi_0(\sigma)\pi(\gamma^{-1})\big) \\ & \qquad =
\mathop{\sum}\limits_{\gamma\in \Gamma}{\rm Tr}_{B(H_n)}
\big(\pi(\gamma^{-1})M_{\chi_F}\pi(\gamma)\pi_0(\sigma)\big) \\ & \qquad =
\mathop{\sum}\limits_{\gamma\in \Gamma}{\rm Tr}_{B(H_n)}
\big(M_{\gamma\chi_F}\pi_0(\sigma)\big).
\end{aligned}
$$
Let  $\nu_0 = ({\rm Im} \; z)^{-2}d\bar{z}dz$ be the canonical $\PSL(2,\R)$-invariant measure on $\mathbb H$  and let $\widehat{\pi_n(\theta)}(\bar{z}, z)$, $z\in \mathbb H$,   be the Berezin contravariant symbol ([Be]) of the unitary operator $\widehat{\pi_n(\theta)}$. Then the above chain of equalities is continued   with the following equality:
$$
\mathop{\sum}\limits_{\gamma \in \Gamma} \int_{\gamma F}\widehat{\pi_n(\theta)}(\bar{z}, z)d\nu_0(z).
$$

This sum is then
$$
\int_{\bH} \widehat{\pi_n(\theta)}(\bar{z}, z),
$$

This last term is   the character "${\rm Tr}\; \pi_n(\sigma)$" of the representation $\pi_n$ (see  [Ne]).  The formula for the above sum is also computed differently in [Za].

\end{proof}

\begin{proof}[Proof of Lemma \ref{characteroriginal}]
 Denote the Hilbert space  on which the representation $\overline{\pi}_0^R$ acts by $H_0^R$. Then
$$
\begin{aligned}
{\rm Tr}_{B(H_0^R)} & \left(\int_{\overline{G}^R}f(g)\overline{\pi}_0^R(g)dg \right) = \mathop{\sum}\limits_{\gamma}{\rm Tr}_{B(H_0^R)}\left(P_{[\pi(\gamma)L]}\int_{\overline{G}^R}f(g)\overline{\pi}_0^R(g)dg \right) \\ & =
 \mathop{\sum}\limits_{\gamma} \int_{\overline{G}^R}f(g){\rm Tr}_{B(H_0^R)}(P_{\pi(\gamma)L}\overline{\pi}_0^R(g))dg \\ & =
\mathop{\sum}\limits_{\gamma \in \Gamma}\int_{\overline{G}^R}f(g){\rm Tr}_{B(L)}(P_L\overline{\pi}_0^R(\gamma g\gamma^{-1})P_L)dg \\ & =
\int_{\overline{G}^R} f(g) \mathop{\sum}\limits_{\gamma}{\rm Tr}_{B(L)}(P_L\overline{\pi}_0^R(\gamma g\gamma^{-1})P_L)dg.
\end{aligned}
$$

The second part of the statement is a consequence of Corollary \ref{plancherel}.
\end{proof}

\vskip10pt

\vskip20pt

\section{The case when the representation $\pi$ admits a "square root" $\pi_0 \otimes \pi_0^{\rm op}$}\label{squareroot}

In this section we analyze the case where a unitary  representation $\pi$ as in  Section \ref{axioms} admits a square root $\pi_0 \otimes \pi_0^{\rm op}$, where $\pi_0$ is a (projective) unitary representation as in Section \ref{pi0}. Since the notation $\overline \pi^{\rm p}$ is reserved to denote the extension of the representation $\pi$ to the Schlichting completion, we will use in this section the notation   $\pi^{\rm op}$ to denote the conjugate representation of $\pi_0$.

This is the situation of Example \ref{koop},  Section \ref{axioms}, when $G=\PGL(2,\Z[\frac{1}{p}])$, $p$ a prime,  $\Gamma$ is the modular group, $\mathcal X=\bH$, and  $\pi_{\rm Koop}$ is the Koopmann representation on $L^2(\bH, \nu_0)$ corresponding to the action of $\PSL(2,\R)$ by M\" obius transformations on the upper halfplane. By Berezin's quantization techniques (\cite{Be}), independently noted in\cite{Be} (see also \cite{Ra3}), we have
$$
\pi_{\rm Koop} = \pi_n \otimes {\pi}^{\rm op}_n, \quad n \geq 1,
$$

\noindent where $\pi_n$ is  any representation in the discrete series of $\PSL(2,\R)$. We will use this as a motivation to analyze directly representations of the form $\pi_0 \otimes \pi_0^{\rm op}$, where $\pi_0$ is as in the previous section.

Before proceeding to this analysis, we note one additional property, common to all the representations   $\overline{\pi}_0^{\rm p}$, $\overline{\pi}^{\rm p}$ constructed in the previous two sections. We will prove that the  above representations are in one to one correspondence with the completely positive maps $\Phi$, $\ast$-representations of the operator system introduced in Definition \ref{canonicalos}. 

We introduce the following notations:
If $g\in G$, we let $L_g\in C^\ast (G)$ be the convolutor with $g$. If $f$ is a function in $C (\overline G)$ we denote by $L_f \in C^\ast (\overline G)$  the operator of convolution with $f$. 
Such a representation plays  the role of an "operator valued eigenvector" for the Hecke algebra.

Indeed, we prove in particular the following property of the map $\Phi$.  If $\sigma_1$, $\sigma_2 \in G$, and
$$[\Gamma\sigma_1\Gamma][\Gamma \sigma_2]= \sum_{j} [\Gamma\theta_j],$$
then
$$\Phi(L_{\chi_{[\Gamma\sigma_1\Gamma]}})\Phi(L_{\chi_{[\Gamma\sigma_2]}})=
\sum_{j}\Phi(L_{ \chi_{[\Gamma\theta_j]}}).$$
If $\Phi$ would take scalar values, then the above property would be exactly the property that an eigenvalue for the action of the Hecke algebra on $\ell^2(\Gamma\backslash G)$ would have.

Let
  $\mathcal A_0(G, \overline {G})\subseteq  \mathcal A(G, \overline {G})$ be the dense subalgebra generated by convolution operators with elements in $G$ and by convolution operators with characteristic functions of cosets of subgroups in $\S$. We denote in this section, by  $\P_{\pi_0, L}$,  the projection  introduced  in formula (\ref{proj}), corresponding to $\Gamma_0=\Gamma$.

\vskip10pt

\begin{thm}\label{cp} Assume that $\pi_0$ is a representation of $G$, as in Definition \ref{TC}.
 We define a linear  map
$
\Phi : \mathcal A_0(G, \overline {G}) \to \ B(L),
$
by the following formula. Let $g \in G$, $\Gamma_0 \in \S$ and let  $g\overline{\Gamma}_0$ be the corresponding  coset. We define
\begin{equation}\label{defphi}
\begin{aligned}
\Phi(L_{\chi_{g\overline{\Gamma}_0}}) & =  \mathop{\sum}\limits_{\theta \in g\Gamma_0}P_L \pi_0(\theta)P_L \\
\Phi(L_g) & =P_L\pi(L_g)P_L .
\end{aligned}
\end{equation}
%
%
Then $\Phi$ has the following properties:
\item {(i)}
$\Phi|_{\O(K, G)}$ is a $\ast$-representation of the operator system
$\O(K, G) $ introduced in Definition \ref {thecanonicalos}, with values in $B(L)$.

  \item(ii) Consider the vector space $L(K, \overline {G})\subseteq \O(K, G)   $   introduced in the   definition mentioned above.

Then $\Phi|_{L(K, \overline {G})}$ takes values in
$B(L)\P_{\pi_0, L}$ and
  $\Phi|_{\H_0(K,\overline G)}$ is $\ast$-algebra representation of $ \H_0(K,\overline G)$ into $\P_{\pi_0, L}B(L)\P_{\pi_0, L}$.
Consequently,   $\Phi|_{\O(K, G)}$   takes values in
$B(L)\P_{\pi_0, L}B(L)$.

\item {(iii)}
$\Phi$ is a completely positive map on $\mathcal A_0(G, \overline {G})$, that is it maps positive elements of the form $X^\ast X, X\in \mathcal A_0(G, \overline {G})$, into positive elements.

\end{thm}

\vskip10pt

\begin{proof} The statement (i) is a consequence of formula (\ref{unitary}). Formula (\ref{multiplicativity1}) in Lemma  \ref{rephecke}, combined with Proposition \ref{lifting},  also provide a proof of statement (i).
The statement (ii) is a consequence of statement (i). Using the fact that property (i) proves the  positivity of the map $\Phi$ on positive elements  $X^\ast X$, where $X\in \C(\chi_{\sigma K} | \sigma \in G)$. For subgroups $\Gamma_0\in \S$, we argue as follows. We let $K_0$ be the closure in $\overline G$ of $\Gamma_0$. We use Lemma \ref{remarca1} first to establish the positivity of $\Phi_{\Gamma_0}(X_0^\ast X_0)$ for $X_0$ in
$\C(\chi_{\sigma K_0} | \sigma \in G)$. The reduction formula (\ref{reduction}) implies that $\Phi(X_0^\ast X_0)$ is also positive.

\end{proof}

The above theorem may be easily generalized by replacing $\Gamma$ with a subgroup $\Gamma_0\in S$. In this case, there is an obvious relation between the corresponding operator system representations.

\begin{lemma}\label{remarca1}
 We use the definitions and notations from the previous theorem.
 Let $\Gamma_0$ be any subgroup in $\S$, and let
$L^{\Gamma_0}$ be as in Theorem \ref{traces}. Then, using $L^{\Gamma_0}$ instead of $L$, one may repeat the above construction for $\Gamma_0$ instead of $\Gamma$. The corresponding completely positive map $\Phi_{\Gamma_0}$ constructed in the previous statement will have the same properties as $\Phi$, replacing $\Gamma_0, \overline{\Gamma_0}$ for $\Gamma,K$.

We embed $L$ into $L^\Gamma_0=\mathop{\oplus}\limits_{i}\pi(s_i)L \subseteq H_0$ by mapping $l$ in $L$ into the vector $l \oplus 0\oplus 0 \dots$. Note that this is not the diagonal embedding of $L$ into $L^\Gamma_0$ that we use in
Proposition \ref{traces}. We denote by $\tilde {P}_L$ the projection from
$L^{\Gamma_0}$ onto $L$. Then
\begin{equation}\label{reduction}
\Phi= \tilde {P}_L \Phi_{\Gamma_0}\tilde {P}_L.
\end{equation}
\begin{proof}
This is straightforward from formula (\ref{defphi}).
\end{proof}

\end{lemma}

\vskip10pt
\begin{rem}\label{pieceofiso}
The operators $\Phi(\chi_{\sigma K})\in B(L)$, $\sigma \in G$ are not isometries, as
$\Phi$ is not a $\ast$-algebra representation.
 However,  as we show below, the operators $\Phi(\chi_{\sigma K})$   are the product of a projection with an isometry.
 
Indeed, for $\sigma\in G$, the partial isometry $L(\chi_{\sigma K})$  has initial space the projection $L(\chi_{\sigma K \sigma^{-1}})$ and range $L(\chi_K)$.

Consider the spaces
$L^{\Gamma_0}, \Gamma_0\in  S$ introduced in the statement of Definition \ref{Lspacespart2}.  The spaces $L^{\Gamma_0}$ were defined in the above mentioned definition  only for $\Gamma_0$ a subgroup of $\Gamma$. We define $L^{\sigma \Gamma \sigma^{-1}}\subseteq L^{\Gamma_\sigma} $ by the formula  $L^{\sigma \Gamma \sigma^{-1}}=\pi_0(\sigma)L$.
 Consequently:
$$\overline{\pi}^{\rm p}(L(\chi_{\sigma K}))=  P_{L^{\sigma \Gamma \sigma^{-1}}}\overline{\pi_0}^{\rm p}(L(\chi_{\sigma K})P_L.$$
 On the other hand, using the skewed  embedding of $L$ in
$L^{\Gamma_\sigma}$ from  Lemma \ref{remarca1}, we obtain  that
$$\Phi(\chi_{\sigma K})= \tilde {P}_L \overline{\pi_0}^{\rm p}(L(\chi_{\sigma K})) P_L.$$
Here the projection $P_L$ corresponds to the standard embedding of $L$ into
$L^{\Gamma_\sigma}$, as described in the statement of Definition \ref{Lspacespart2} and  $\tilde {P}_L$ is the projection from  the previous statement.
\end{rem}

\

\

The completely positive maps constructed in Theorem \ref{cp} are the building blocks of the Hecke operators.   In the next result we prove that a representation as in Theorem \ref{cp}   encodes all the properties of the representation $\overline{\pi_0}^{\rm p}$.  Hence, given $\Phi$, we may recover the representation $\overline{\pi_0}^{\rm p}$ and hence the representation $\pi$.

\begin{prop}\label{matrix1} In the context of Theorem \ref {cp},  the formula for the Hecke operators introduced in  Theorem $\ref{traces}$ is as follows.
Fix $\Gamma_0$ in $\S$ and choose the coset decomposition $\Gamma = \bigcup s_i\Gamma_0$. Then

\begin{equation}
\label{newhecke}
[\Gamma_0: (\Gamma_0)_\sigma]P_{H_0^{\Gamma_0}}\overline{\pi}_0^{\rm p}(\sigma)P_{H_0^{\Gamma_0}} = \mathop{\sum}\limits_{i, j}\Phi(\chi_{\overline{s_i^{-1}\Gamma_0\sigma\Gamma_0s_j}}) \otimes e_{\Gamma_0s_i, \Gamma_0s_j}, \  \sigma \in G.
\end{equation}

Consequently, there exists a one to one correspondence between   the representation $\pi_0$ as in Definition \ref{TC} and   completely positive map $\Phi$ with the properties (i), (ii), (iii)  in Theorem \ref{cp}.

\end{prop}

\vskip10pt

\begin{proof}
The  formula
(\ref{newhecke})
 is simply formula (\ref{a0}) in Theorem \ref{a0formula}  rewritten in the new context, by using formula (\ref{defphi}) in Theorem \ref{cp}.

Assume the  properties (i), (ii), (iii) in   Theorem \ref {cp}.  We prove that the operators in formula (\ref{newhecke}) define a representation of the Hecke algebra of double  cosets  for all $\Gamma_0\in\S$. 

Ultimately, the verifications of the multiplicativity of the Hecke operators, given in formula (\ref{newhecke}), come to identities of the form:
\begin{equation}\label{fundid}
\mathop{\sum}\limits_{\gamma_0 \in \Gamma_0}P_{L^{\Gamma_0}}\pi_0(\sigma_1\gamma_0)P_{L^{\Gamma_0}}(\pi_0(\gamma_0^{-1}\sigma_2))P_{L^{\Gamma_0}} = P_{L^{\Gamma_0}}(\pi_0(\sigma_1\sigma_2))P_{L^{\Gamma_0}}.
\end{equation}

The reason for which  this equality holds true, is that $P_{L^{\Gamma_0}}$ is the projection on a $\Gamma_0$-wandering, generating subspace of $H$ (see also the  proof of Lemma \ref{rephecke}).
Thus, the  only identity  because of which   the representation in formula
(\ref{newhecke}) is  a representation of the Hecke algebra,  is the identity:
$$
\mathop{\sum}\limits_{\gamma_0 \in \Gamma_0}\pi(\gamma_0)P_{L^{\Gamma_0}}\pi(\gamma_0^{-1}) = {\rm Id}_{H_0}.
$$

Decomposing  $P_{L^{\Gamma_0}} = \mathop{\sum}\limits_{i}\pi(s_i)P_L\pi(s_i)$, this is implied by   the identity:
$$
\mathop{\sum}\limits_{\gamma \in \Gamma}\pi(\gamma)P_L\pi(\gamma^{-1}) = {\rm Id}_{H_0}.
$$
 But this is exactly the identity  proving the multiplicativity property (ii) in Theorem \ref{cp}.

 Thus if we know that $\Phi$ is multiplicative, as in property (ii) in the previous theorem, then we automatically have that the completely positive maps $\Phi_{\Gamma_0}$ verify the corresponding multiplicativity property in (ii) on the corresponding operators systems $\O(\overline{\Gamma_0}, \overline G)$, for $\Gamma_0\in \S$.
 Let $\P_{\Gamma_0, L}$ be the projection introduced in foromla (\ref{proj}).

  Since the operator systems contain the corresponding Hecke algebras $\H_0(\overline{\Gamma_0}, \overline G)$, it follows that
 the representation in formula (\ref{newhecke}) is a $\ast$-algebra representation of the inductive limit of all the above Hecke algebras into the inductive limit of the spaces $\P_{\Gamma_0, L}B(L^{\Gamma_0})\P_{\Gamma_0, L}$. But this inductive limit is exactly the space of bounded operators acting on the Hilbert space $\overline H_0^{\rm p}$.  Since along with the Hecke algebras we also have a representation of the spaces of cosets, it follows that  we have reconstructed the unitary representation $\overline{\pi}_0^{\rm p}$ of $C^\ast(\overline{G})$. Hence we may recover $\pi_0$, because of  Theorem \ref{cp}.

\end{proof}

\vskip10pt

In the following  statements (Lemma \ref{rephecke}, Proposition \ref{lifting} and Proposition \ref{double}) we recall results from [Ra2]. We adapt the statement of the  results to the present framework.
In the next lemma we prove a  result complementing the statement in Theorem \ref{cp}.
We prove that the completely positive maps in Theorem \ref{cp} have a natural lifting to the algebra $\L(G) \otimes B(L)$. This lifting was essential tool in proving, in the paper [Ra], the essential norm estimates on the spectrum of the Hecke operators. In particular, we give an alternative interpretation for property (ii) in Theorem \ref{cp}.

\vskip10pt

\vskip10pt

\begin{proof} [Proof of Lemma \ref{rephecke}]
This was also proved in (\cite {Ra1}, Proposition 2.2 and Lemma 3.1).
The  main step of the proof of the multiplicativity property in formula (\ref{multiplicativity1}) is the following:  by identifying the coefficients of $\rho(g)$, $g \in G$, in both sides of the equation, one reduces the proof of the multiplicativity property to the following  equality (also used in the proof of Proposition \ref{matrix1}) :
$$
\begin{aligned}
& \mathop{\sum}\limits_{\gamma \in \Gamma}P_L\pi_0(\sigma_1\gamma)P_L\pi_0(\gamma^{-1}\sigma_2)P_L 
= \mathop{\sum}\limits_{\gamma \in \Gamma}P_L\pi_0(\sigma_1)\pi(\gamma)P_L\pi(\gamma^{-1})\pi_0(\sigma_2)P_L \\ & \qquad
= P_L\pi_0(\sigma_1)\pi_0(\sigma_2)P_L = P_L\pi_0(\sigma_1\sigma_2)P_L, \quad \sigma_1, \sigma_2 \in G.
\end{aligned}
$$

\end{proof}

\vskip10pt

The operators $\tilde{\Phi}_{\pi_0, L}$ introduced in Definition \ref{thecanonicalos} have a similar interpretation as in Remark \ref{pieceofiso}.  Moreover, because of the convergence assumptions in the statement of Theorem  \ref {traces}, for the sums  of the form $\mathop{\sum}\limits_{\theta \in C}P_L\pi_0(\theta)P_L$ with cosets $C$ in $\overline{G}$, the operators $\tilde{\Phi}_{\pi_0, L}(C)$ are liftings of the operators $\Phi(C)$ in the Proposition $\ref{cp}$.

This is proved  in the  following statement. The  statement  is a straightforward  consequence of Theorem \ref{cp} and of Lemma \ref{rephecke}. We mention the statement as a separate lemma since it also  provides
a different proof for  Theorem \ref{cp}.


\vskip10pt

\begin{lemma}\label{lifting} Let $\varepsilon$ be the unbounded character $\varepsilon$ on $\ell ^1(G) \subseteq \L(G)$ which associates to $x$ in  $\ell ^1(G)$ the sum of its coefficients.
We extend $\varepsilon$ to an unbounded character $\tilde{\varepsilon} = \varepsilon \otimes {\rm Id}_{B(L)}$,
$$\tilde{\varepsilon} :  \ell ^1(G) \otimes B(L)\subseteq \cR(G) \otimes B(L)\rightarrow B(L).$$

The convergence assumption in Definition \ref{TC}  implies that the image of the $\ast$-representation $\tilde{\Phi}_{\pi_0, L}$, constructed in Lemma \ref{rephecke}, is contained   in the domain of $\tilde{\varepsilon}$. 
 Consequently, with $\Phi$ as in Theorem \ref{cp}, we have the following commutative diagram:
$$
\tilde{\varepsilon}\circ \tilde{\Phi}_{\pi_0, L}|_{\O(K, \overline G)}= \Phi|_{\O(K, \overline G)}$$



\end{lemma}

\vskip10pt

\begin{proof} This is straightforward from the formulae of ${\Phi}$ and $\tilde{\Phi}_{\pi_0, L}$ from Theorem \ref{cp} and, respectively, from Lemma \ref{rephecke}.

\end{proof}

\vskip10pt

The operators $\tilde{\Phi}_{\pi_0, L}$ are  used to construct a unitary equivalent representation (see formula \eqref{formulapsi} for the Hecke operators associated to the unitary, diagonal    representation $\pi_0 \otimes \pi_0^{\rm op}$ of $G$, where $\pi_0$ is as in Definition \ref{TC}). This was first proved (in the case of Murray-von Neumann dimension equal to 1) in \cite{Ra}, Theorem 22 (see also \cite{Ra2} for a more concise exposition), and then generalized to arbitrary dimension in \cite{Ra1}, Theorem 3.2. For the convenience of the reader, since  we are  explaining   the example of the Hecke algebra representation associated with the unitary representation
$\pi_0 \otimes \pi_0^{\rm op}$ of $G$,
we  recall    the statement of Theorem  3.2 in \cite{Ra1}.  In Theorem \ref {multhecke} we provide an alternative proof of the fact that formula (\ref{formulapsi}) gives a representation of the Hecke algebra of double cosets of $\Gamma$ in $G$.


We  use the  identifications proved in Example \ref{ad}  and the operators $\tilde{\Phi}_{\pi_0, L}$ introduced in Lemma \ref{rephecke}, to explicitly describe the Hecke operators on $\Gamma$-invariant vectors associated  the unitary, diagonal  representation $\pi_0 \otimes \pi_0^{\rm op}$ of $G$.
In this case, the Hilbert  spaces of $\Gamma$-invariant vectors  are easier to handle, since we may canonically identify these spaces with the $L^2$-spaces of the von Neumann algebra of operators commuting with the image of the representation of the group $\Gamma$.

 In \cite{Ra}, by using Berezin's quantization methods (\cite{Be}), or alternatively using the results in \cite{Re}  we prove that the above model for the Hecke operators acting on $\Gamma$-invariant vectors for $\pi_0 \otimes \pi_0^{\rm op}$  is unitarily equivalent to the representation of the Hecke operators of the associated von Neumann algebra of Maass forms.
 
   The content of  Theorem \ref{double} is the  explicit operator algebra model of the former representation associated with $\pi_0 \otimes \pi_0^{\rm op}$. This theorem was obtained in the case of Murray von Neumann dimension 1 in (\cite{Ra}, Theorem 22)  and generalized to arbitrary dimension in (cite{Ra1}, Theorem 3.2).  We have adapted the statement of the theorem to the framework of the present paper.

\vskip10pt

\vskip10pt
\begin{proof}[Proof of Theorem \ref{double}]

This is  Theorem 3.2 in \cite{Ra1}.
The present framework proves that, once a canonical $L$ is chosen for the representation $\pi$, the representation of the Hecke operators for $\pi_0 \otimes \pi_0^{\rm op}$ becomes canonical.
\end{proof}

\vskip10pt

We are giving a direct proof, in the particular case ${\rm dim}_{\{\pi_0(\Gamma)\}''}H = 1$, (Theorem \ref{multhecke}) of the fact the Hecke operators introduced in Theorem  \ref{double}, formula (\ref{formulapsi}), define a multiplicative representation of the Hecke algebra of double cosets of $\Gamma$ in $G$.
The proof will show that the Hecke algebra representation determined by the Hecke operators for $\pi_0 \otimes \pi_0^{\rm op}$, is obtained from a canonical representation of the Hecke algebra, that is further composed with a quotient map.
For simplicity of the exposition we assume in the rest of the paper that the groups $\Gamma$ and $G$ have infinite, non-trivial conjugacy classes, and hence that the associated von Neumann algebras have unique traces.

\begin {rem}\label{1dim}
Assume that ${\rm dim}_{\{\pi_0(\Gamma)\}^{\prime\prime}}H = 1$. In the setting of Theorem \ref{hecke} in Section \ref{axioms},
we take $L=L_0 = \C\xi$ for a cyclic,  trace vector $\xi\in H_0$, for $\pi_0|_\Gamma$. The construction in Lemma \ref{rephecke}, gives a linear map $\tilde{\Phi}_{\pi_0, L_0}$, which we now denote by $t$:
$$
t : \H_{\rm red}(\Gamma, G) \to \cR(G)\otimes B(L_0)\cong \cR(G).
$$

\noindent Since $L_0$ is one dimensional, and using the vector $\xi$ to identify $L_0 = \C\xi$ with $\C$,  it follows that we may substitute
$P_{L_0} \pi_0(\theta)P_{L_0}$ by
$${\rm Tr}(P_{L_0} \pi_0(\theta)P_{L_0})=
\langle \pi_0(\theta)\xi, \xi\rangle,\quad \theta \in G.
$$

For a coset $C$ of a subgroup in $\S$, the formula  for $\tilde{\Phi}_{\pi_0, L_0}(\chi_C)$ from Lemma \ref{rephecke} is now
 $$t(\chi_C) = \mathop{\sum}\limits_{\theta \in C} \langle \pi_0(\theta)\xi, \xi \rangle \rho(\theta).$$
We  compose $t$ with the canonical anti-isomorphism between $\L(G)$ and $\cR(G)$. For simplicity we denote the composition map also by $t$. Thus  $t$ is a linear map from  C$^{\ast}(G)$ with values  in $\L(G)$.   We denote by $\lambda_g$ the left convolutors by elements $g\in G$. Then $t$ is given by the formula$$
t(\chi_C) = \mathop{\sum}\limits_{\theta \in C} \langle \overline{\pi_0(\theta)\xi, \xi \rangle} \lambda(\theta).
$$
Because of Lemma \ref{rephecke}, $t$ is a $\ast$-preserving, multiplicative representation of the operator system $\O_{\Gamma, G}=\O(K,\overline G)$ introduced in Definition \ref {thecanonicalos}:
$$
\O_{\Gamma, G} = \big[{\rm Sp}\{ \chi_{\sigma_1 K} | \sigma_1 \in G \}\big]\big[{\rm Sp}\{ \chi_{\sigma_2 K} | \sigma_2 \in G \}\big]^{\ast}.
$$

\noindent
We use a notational  convention, denoting the characteristic functions  $\chi_{\sigma_1 K}$, $\chi_{K\sigma_2}$, $\chi_{\sigma_1 K\sigma_2} $ simply by the  corresponding cosets in $G$: respectively $\sigma_1\Gamma$, $\Gamma\sigma_2$, $\sigma_1\Gamma\sigma_2$, for $\sigma_1, \sigma_2\in G$.
Thus the $\ast$-preserving, multiplicativity property for $t|_ {\O_{\Gamma, G}}$ reads as
$$t(\sigma_1\Gamma)t(\sigma_2^{-1}\Gamma)^{\ast} =
t(\sigma_1\Gamma)t(\Gamma\sigma_2) = t(\sigma_1\Gamma\sigma_2),\quad \sigma_1, \sigma_2 \in G.
$$

\end{rem}
\

\

\begin{rem}
In practice, it is difficult to find a cyclic trace vector $\xi$ as above.
So it is preferable to use the construction from Section \ref{pi0}, Theorem \ref{traces}. Thus $\pi_0$ comes from a larger representation $\pi$ of $G$ into the unitary group of a Hilbert space $H$, by restricting to a space $H_0\subseteq H$, that is  invariant under $\pi(G)$. In this case we use a choice of  $\Gamma$-wandering, generating  subspace $L$ for
$\pi|_\Gamma$. In the case of the analytic discrete series of unitary representations $\pi_n, n\geq 1$ of $\PSL(2,\mathbb R)$,  a choice described as above is almost canonical, as it consists into the selection of a fundamental domain for the action of $\Gamma$ on $\mathbb H$.

To obtain straightforwardly   the   representation $t$ from $\tilde{\Phi}_{\pi_0, L}$, one  proceeds directly as follows (\cite{Ra1}) :
Consider the conditional expectations
$$
E^{\cR(G) \otimes B(L)}_{\cR(G) \otimes \C{\rm Id}_{B(L)}},\quad E^{\cR(\Gamma) \otimes B(L)}_{\cR(\Gamma) \otimes \C{\rm Id}_{B(L)}}
$$
from $\B=\cR(G) \otimes B(L)$, and respectively $\A=\cR(\Gamma) \otimes B(L)$, onto the algebras $\cR(G) \otimes \C{\rm Id}_{B(L)}$, and respectively
$\cR(\Gamma) \otimes \C{\rm Id}_{B(L)}$. The conditional expectations are simply computed  by taking the operatorial trace on the tensor factor corresponding to $B(L)$.

For a coset $C$ as in the previous remark, we have
\begin{equation} \label {tildec}\tilde t (\chi_C)=
E^{\cR(G) \otimes B(L)}_{\cR(\Gamma) \otimes {\rm Id}_{B(L)}}(\tilde \Phi_{\pi_0, L}(\chi_C)) = \mathop{\sum}\limits_{\theta \in C}{\rm Tr}(P_L\pi_0(\theta))\rho(\theta).
\end{equation}
\noindent  We use the formula (\ref{newproj}) for the projection $P_0$ and  define $$
\xi_0=E^{\cR(\Gamma) \otimes B(L)}_{\cR(\Gamma) \otimes {\rm Id}_{B(L)}}(P_0) = \mathop{\sum}\limits_{\gamma \in \Gamma}{\rm Tr}(P_L\pi_0(\gamma))\rho(\gamma).
$$

Since ${\rm dim}_{\{\pi_0(\Gamma)\}^{\prime\prime}}H_0= 1$, and $P_0$ is a projection in $\A\subseteq \B$ of  trace 1, it follows that $\xi_0$ has zero kernel. Moreover (\cite{Ra1}, Proposition 3.3)  the conditional expectation map, corrected with   the inverse of the square root of $\xi_0$, is
a von Neumann algebra isomorphism when restricted to $P_0\B P_0$.
Thus \begin{equation}\label{tilde}
\tilde E=(\xi_0)^{-1/2} E^{\cR(G) \otimes B(L)}_{\cR(\Gamma) \otimes {\rm Id}_{B(L)}}|_{P_0\B P_0} (\xi_0)^{-1/2},
\end{equation}
\noindent
is a von Neumann algebra isomorphism $P_0\B P_0$ onto $\cR(G)$.
 We define
$$
 t(\chi_C)= \tilde E (\tilde t(C)).$$
Then $t|_ {\O_{\Gamma, G}}$ is an isomorphism from  $\H(\Gamma,G)$ into $\cR(G)$
 (see Lemma 3.3 in \cite{Ra1} for the proof).
 Combining the formulae (\ref{tildec}) and (\ref{tilde}), we obtain the following alternative formula for the representation $t$ from the previous remark:
 $$t(\chi_C)= (\xi_0)^{-1/2}\big[ \mathop{\sum}\limits_{\theta \in C}{\rm Tr}(P_L\pi_0(\theta))\rho(\theta)\big] (\xi_0)^{-1/2}.$$
 \end{rem}

\vskip10pt

\







\begin{proof}[Proof of Theorem \ref{multhecke}]

Fix $\sigma \in G$. Then the double coset $\Gamma\sigma\Gamma$ decomposes as $\bigcup \Gamma\sigma s_i = \bigcup r_j\sigma\Gamma$.  Hence
\begin{equation}\label{chain111}
\begin{aligned}
& \chi_K(L(\chi_{K\sigma K}) \otimes L(\chi_{K\sigma K})^{\rm op})\chi_K \\ & 
= \mathop{\sum}\limits_{i, j}[L(\chi_{K\sigma s_i}) \otimes L(\chi_{K\sigma s_j})^{\rm op}]\chi_{s_i\sigma^{-1}K\sigma s_j^{-1}\bigcap K} \\ &
= \mathop{\sum}\limits_{a, b} \chi_{r_a\sigma K\sigma^{-1}r_b\bigcap K} [L(\chi_{r_a\sigma K}) \otimes L(\chi_{r_b\sigma K})^{\rm op}] 
\end{aligned}
\end{equation}
\begin{equation}\nonumber
= \mathop{\sum}\limits_{i, j, a, b}\chi_{r_a\sigma K\sigma^{-1}r_b\bigcap K}[L(\chi_{r_a\sigma K_{\sigma^{-1}} s_i}) \otimes L(\chi_{r_b\sigma K_{\sigma^{-1}}s_j})^{\rm op}]\chi_{s_i\sigma^{-1}K\sigma s_j \bigcap K}.
\end{equation}

Here $K_{\sigma^{-1}}$ is the closure in $K$ of the subgroup $\Gamma_{\sigma^{-1}} = \sigma^{-1}\Gamma\sigma \bigcap \Gamma$.

Using the above equality, one proves immediately (see e.g. the computations in \cite{Ra}, Section 5, or \cite{Ra2}) that the linear map in the statement is multiplicative.
This completes the proof of part (i) of the statement.



 The representation $t$ of the operator system $\O_{\Gamma,G}=\O(K,\overline G)$ extends obviously to a representation $\tilde t$ of the operator system
$$
(L^{\infty}(\overline{G},\mu){\rm Sp}\{L( \chi_{\sigma_1 K}) | \sigma_1 \in G \})(L^{\infty}(\overline{G},\mu){\rm Sp}\{L( \chi_{\sigma_2 K}) | \sigma_2 \in G \})^{\ast}.
$$

Then $\tilde t$ extends to a "double" representation $t_2$ of an operator system contained in $C^{\ast}((G \times G^{\rm op}) \rtimes L^{\infty}(\overline{G}),\mu)$ containing the image of the Hecke algebra constructed in point (i).
This concludes the proof of part (ii).

The  important observation for the proof of part (iii)  is  the fact that all the operations that are involved in the multiplication of two  elements of the form $\chi_K(L(\chi_{K\sigma_1 K}) \otimes L(\chi_{K\sigma_1 K})^{\rm op})\chi_K$ remain inside the domain of the representation $t_2$ (see the first equality in the chain of equalities in formula (\ref{chain111})). Indeed these operations involve only convolutions of the form $$L(\chi_{\sigma_1 K})L(\chi_{K\sigma_2}),\quad
\sigma_1, \sigma_2\in G,$$ or their opposites.

Consequently, the composition of $t_2$ with the map in the preceding lemma gives a representation of the Hecke algebra.

\end{proof}

\end{document}